\documentclass[10pt,a4paper]{amsart}
\usepackage[english]{babel}
\usepackage[T1]{fontenc}
\usepackage{indentfirst}
\usepackage{fancyhdr}
\usepackage{amsfonts}
\usepackage{amsmath}
\usepackage{amsthm}
\usepackage{amssymb, enumerate} 
\usepackage{bbm}
\newtheorem{thm}{Theorem}[section]
\newtheorem{lem}[thm]{Lemma}
\newtheorem{cor}[thm]{Corollary}
\newtheorem{prop}[thm]{Proposition}
 
\newtheorem{probl}{Problem}

\newtheorem{ques}[thm]{Open Question}
\theoremstyle{definition}
\newtheorem{rem}[thm]{Remark}
\newtheorem{dfn}[thm]{Definition}
\numberwithin{equation}{section}
\newcommand{\be}{\begin{equation}}
\newcommand{\ee}{\end{equation}}
\def\CC{\mathbb{C}}

\def\NN{\mathbb{N}}

\def\RR{\mathbb{R}}

\def\PP{\mathbb{P}}
\def\K{{\mathcal K}}
\def\B{{\mathcal B}}

\def\P{{\mathcal P}}
\def\F{{\mathcal F}}

\def\C{{\mathcal C}}

\def\Z{{\mathcal Z}}
\def\T{{\mathcal T}}
\def\X{{\underline{X}}}

\def\R{{\mathcal R}}

\def\dd{{\mathrm{d}}}

\makeatletter
\@namedef{subjclassname@2020}{%
  \textup{2020} Mathematics Subject Classification}
\makeatother

\def\subindexset{J}
\def\indexset{I}
\def\superindexset{\Lambda}
\newcommand{\topchar}[1]{{\tau_{\!\scriptscriptstyle{X\!(\!#1\!)}}}}
\newcommand{\Borchar}[1]{{\B_{\!\scriptscriptstyle{X\!(\!#1\!)}}}}
\usepackage{tikz}
\usepackage{tikz-cd}
\usepackage{hyperref, xcolor}

\title[Projective limit techniques for the infinite dimensional MP]{\bf Projective limit techniques for the infinite dimensional moment problem}
\author[Infusino, Kuhlmann, Kuna, Michalski]{Maria Infusino, Salma Kuhlmann, Tobias Kuna, Patrick Michalski}
\address[M. Infusino]{\newline \indent Dipartimento di Matematica e Informatica, Universit\'a degli Studi di Cagliari, \newline \indent Palazzo delle Scienze, Via Ospedale 72, 09124 Cagliari, Italy.}
\email{maria.infusino@unica.it} 
\address[S. Kuhlmann, P. Michalski]{\newline \indent Fachbereich Mathematik und Statistik, Universit\"at Konstanz,\newline \indent
Universit\"atstrasse 10, 
78457 Konstanz, Germany.}
\email{salma.kuhlmann@uni-konstanz.de}
\email{patrick.michalski@uni-konstanz.de}
\address[T. Kuna]{\newline \indent Dipartimento di Ingegneria, Scienze dell'Informazione e Matematica,\newline \indent Universit\'a degli Studi dell'Aquila, Via Vetoio, 67100, L'Aquila, 
Italy}
\email{tobias.kuna@univaq.it}

\date{\today}

\begin{document}
\begin{abstract}
We deal with the following general version of the classical moment problem: when can a linear functional on a unital commutative real algebra $A$ be represented as an integral with respect to a Radon measure on the character space $X(A)$ of $A$ equipped with the Borel $\sigma-$algebra generated by the weak topology? We approach this problem by constructing $X(A)$ as a projective limit of the character spaces of all finitely generated unital subalgebras of $A$. Using some fundamental results for measures on projective limits of measurable spaces, we determine a criterion for the existence of an integral representation of a linear functional on $A$ with respect to a measure on the cylinder $\sigma-$algebra on $X(A)$ (resp. a Radon measure on the Borel $\sigma-$algebra on $X(A)$) provided that for any finitely generated unital subalgebra of $A$ the corresponding moment problem is solvable. We also investigate how to localize the support of representing measures for linear functionals on $A$. These results allow us to establish infinite dimensional analogues of the classical Riesz-Haviland and Nussbaum theorems as well as a representation theorem for linear functionals non-negative on a ``partially Archimedean'' quadratic module of $A$. Our results in particular apply to the case when $A$ is the algebra of polynomials in infinitely many variables or the symmetric tensor algebra of a real infinite dimensional vector space, providing a unified setting which enables comparisons between some recent results for these instances of the moment problem.
\end{abstract}
\maketitle
\vspace{-0.8cm}

\section*{Introduction}

\noindent
The main question of this paper is the following general version of the classical (full) moment problem:
\begin{probl}\label{GenKMP}\ \\
Let $A$ be a unital commutative real algebra and let $\Sigma$ be a $\sigma-$algebra on the character space $X(A)$ of $A$. Given $K\in \Sigma$ and a linear functional $L:A\to\RR$ with $L(1)=1$, does there exist a measure $\nu$ defined on $\Sigma$ and supported in $K$, i.e., $\nu(X(A)\setminus K)=0$, such that 
 $L(a)=\int_{X(A)}\alpha(a)\dd\nu(\alpha)$ for all $a\in A$?
\end{probl}
 \vspace{-0.1cm}
If such a measure $\nu$ does exist, we say that $\nu$ is a \emph{$K-$representing measure} on~$\Sigma$ for $L$ and that $L$ is represented by $\nu$ on $K$.  In the following we assume that $X(A)$ is non-empty and endow it with the weakest topology $\topchar{A}$ such that for all $a \in A$ the function $\hat{a} : X(A)
\rightarrow \mathbb{R}$ defined by $\hat{a}(\alpha):= \alpha(a)$ is continuous. Note that the topology $\topchar{A}$ coincides with the topology induced on $X(A)$ by the embedding 
$X(A)\to \RR^A$, $\alpha \mapsto \left(\alpha(a)\right)_{a\in A}$,
where $\RR^A$ is equipped with the product topology. Hence, $(X(A), \topchar{A})$ is Hausdorff and it is metrizable if and only if $A$ is countably generated. Here, however, we do not assume that $A$ is countably generated. We denote by $\Borchar{A}$ the Borel $\sigma-$algebra on $(X(A), \topchar{A})$.

Problem \ref{GenKMP} for $A=\RR[X_1,\ldots, X_d]=:\RR[\X]$, $d\in\NN$, $\Sigma=\Borchar{A}$ and $\nu$ being a Radon measure coincides with the classical $d-$dimensional $K-$moment problem. Indeed, using that the identity is the unique $\RR$--algebra homomorphism from $\RR$ to $\RR$, it can be easily proved that any $\RR$--algebra homomorphism from $\RR[\X]$ to $\RR$ corresponds to a point evaluation $p\mapsto p(\alpha)$ with $\alpha\in\RR^d$ and so that there exists a topological isomorphism $\RR^d\to X(\RR[\X])$ (see, e.g., \cite[Proposition~5.4.5]{MarshBook}). 

Problem \ref{GenKMP} is general enough to cover several infinite dimensional instances of the moment problem already considered in the literature, e.g., when $A$ is not finitely generated or when the representing measures are supported in an infinite dimensional vector space. For example, the case when~$A$ is the algebra of polynomials in infinitely many variables was considered in \cite{AJK15} and \cite{GKM16}. In particular, Problem \ref{GenKMP} is considered in \cite{GKM16} for the family of constructibly Borel sets and for constructibly Radon measures, while in \cite{Schmu-new} for the family of cylinder sets and (in our terminology, see Remark \ref{prop::cyl-quasi-measure}-(ii)) for cylindrical quasi-measures. In fact, both families of sets generate the same $\sigma-$algebra (see Proposition \ref{prop::constr-Borel-is-cylinder}), which is in general smaller than the Borel $\sigma-$algebra. Constructibly Radon measures and cylindrical quasi-measures will play a key role in this article as well. 

There has been also a lot of activity about Problem \ref{GenKMP} for topological algebras and representing Radon measures (see, e.g., \cite{BCR}, \cite{MGHAS}, \cite{MGES}, \cite{GKM13},  \cite{GMW}, \cite{IK-probl}, \cite{Jac},  \cite{Lass}, \cite{Schmu78}). For example, several works have been devoted to the study of Problem \ref{GenKMP} for $A$ being the symmetric algebra $S(V)$ of a locally convex space $V$ (not necessarily finite dimensional) and $\Sigma=\Borchar{A}$ (see, e.g.,\! \cite[Chapter~5, Section~2]{BK}, \cite{BS}, \cite{Bor-Yng75}, \cite{GIKM}, \cite{Heg75}, \cite{I}, \cite{IKM}, \cite{IK-new}, \cite{IKR}, \cite[Section 12.5]{Schmu90}, \cite{Schmu-new}). However, in this article we do not consider any topology on our algebras when studying Problem~\ref{GenKMP} and, as for the classical case, our final purpose is to find necessary and sufficient conditions for the existence of $K-$representing Radon measures on $X(A)$, i.e., locally finite and inner regular measures on $\Borchar{A}$. 

Our strategy is to construct $X(A)$ as a projective limit of the family of Borel measurable spaces consisting of all $(X(S), \Borchar{S})$ for $S\in I$, where~$I$ is the collection of all finitely generated unital subalgebras of $A$ (Theorem \ref{prop::X(A)-proj-lim}). This construction provides another measurable structure on $X(A)$, namely the cylinder $\sigma-$algebra~$\Sigma_I$, which is in general smaller than $\Borchar{A}$. In Theorem~\ref{cylinder-MP}-(i) we investigate under which conditions a linear functional $L: A\to\RR$ has an $X(A)-$representing measure on~$\Sigma_I$ whenever $L\!\restriction_S$ has an $X(S)-$representing Radon measure on $\Borchar{S}$ for all $S\in I$. The condition in Theorem~\ref{cylinder-MP}-(i) expresses that the support of the representing measure is described in terms of those characters in $X(S)$ which can be extended to all of $A$. Under natural assumptions on $A$, see Remark~\ref{rem::cylinder-MP-cofinal}-(i), no condition is required. We also show a more general version of this  result in Theorem~\ref{cylinder-MP-Support}-(i),\linebreak which deals with the support of the representing measures. Under the same conditions as in Theorem \ref{cylinder-MP-Support}-(i), in Corollary~\ref{main-KMP}-(i) we are able to reduce the solvability of Problem~\ref{GenKMP} for $\Sigma=\Sigma_I$ and $K\subseteq X(A)$ closed to the solvability of a whole family of classical $K_{Q\cap S}-$moment problems, where $Q$ is a quadratic module in $A$ such that $K=\{\alpha\in X(A):\alpha(a)\geq 0,\forall a\in Q\}=K_Q$ and $S\in I$ (classical because each $S\in I$ is finitely generated; however, $Q \cap S$ may be a not finitely generated quadratic module). Hence, exploiting the classical (finite dimensional) moment theory we get existence criteria for representing measures on~$\Sigma_I$.

The problem of extending measures defined on the cylinder $\sigma-$algebra corresponding to a projective limit of measurable spaces to a Borel measure was already considered in the general theory of projective limits (see references in Section~\ref{sec:Prelim}). In particular, in 1956 Prokhorov gave a necessary and sufficient  condition (similar to tightness) for extending a measure on a cylinder $\sigma-$algebra (corresponding to a projective limit of Borel measurable spaces associated to Hausdorff topological spaces) to a Radon measure. 
In Theorem~\ref{cylinder-MP}-(ii) we establish that a linear functional $L: A\to\RR$ has an $X(A)-$representing Radon measure on $\Borchar{A}$ if and only if there exists an $X(S)-$representing Radon measure for $L\!\restriction_S$ on $\Borchar{S}$ for all $S\in I$ and the collection of these representing measures satisfies the Prokhorov condition \eqref{epsilon-K-char}. Note that we do not require the $X(S)-$representing Radon measure on $\Borchar{S}$ to be unique (see Lemma~\ref{lem::ex-ex-proj-sys}). In Theorem \ref{cylinder-MP-Support}-(ii) we deal with the support of the representing measures.
From Theorem \ref{cylinder-MP}, Theorem~\ref{cylinder-MP-Support} and Corollary~\ref{main-KMP} we obtain infinite dimensional analogues of Riesz-Haviland's theorem (Theorem~\ref{cor::Haviland}) and Nussbaum's theorem (Theorem~\ref{thm::general-Nussbaum}). Concerning the localization of the support, we provide an alternative proof of the solution to Problem \ref{GenKMP} for Archimedean quadratic modules (Theorem~\ref{cylinder-KMP}) and we extend to this general setting \cite[Theorem 3.2]{Lass} and \cite[Theorem 5.1]{IKR} about representing Radon measures for linear functionals which are non-negative on a quadratic module and satisfy a Carleman type condition (Theorem~\ref{cylinder-MP2}). We actually determine sufficient conditions for the existence of representing Radon measures for a larger class of linear functionals joining the conditions in Theorem~\ref{cylinder-MP2} and the ones in Theorem~\ref{cylinder-KMP}, namely, for linear functionals non-negative on a ``partially Archimedean'' quadratic module (Theorem~\ref{hybrid-quadraticMod}).

Our main Theorem \ref{cylinder-MP-Support} allows us to systematically derive these infinite dimensional results from the corresponding finite dimensional ones. The projective limit approach provides a unified setting which enables to marshal several recent results (e.g., in \cite{AJK15}, \cite{GKM16}, \cite{Schmu-new}) about infinite dimensional instances of Problem \ref{GenKMP}.

Let us shortly describe the organization of this paper. The core material is contained in Sections \ref{sec:Characters} and~\ref{sec:PLMP}, whereas in Section~\ref{sec:Prelim} we collect some essential results about projective limits (see, e.g., \cite{BouInt}, \cite{S73}, \cite{Y85}). In particular, the proofs in Section~\ref{sec:Prelim} can be skipped by a reader already familiar with these subtle and complex techniques. In Section~\ref{sec:quasi-meas}, we study some properties of a projective limit of a projective system of measurable spaces, define the associated cylinder algebra and cylindrical quasi-measures. Section~\ref{sec:prokh-cond} is dedicated to the so-called Prokhorov condition, which will play a fundamental role throughout the whole paper. In Section~\ref{sec:ext-cyl} we deal with the problem of extending cylindrical quasi-measures to actual measures on the cylinder $\sigma-$algebra (Lemma~\ref{lem::content} and Theorem~\ref{ext-cylinderSigmaAlg}). In Section~\ref{sec:ext-Rad} we discuss the extension to Radon measures on the Borel $\sigma-$algebra on the projective limit space (Theorem~\ref{Prokh} and Corollary~\ref{Prokh-ctbl}). In Section~\ref{sec:constr}, we show how to construct the character space $X(A)$ of a unital commutative real algebra $A$ as projective limit of the projective system $\{(X(S), \Borchar{S}), \pi_S,I\}$, where $X(S), \Borchar{S}, \pi_S, I$ are the ones defined above (Theorem~\ref{prop::X(A)-proj-lim}). In Section~\ref{sec:sigma-alg}, we analyze the relation between the cylinder $\sigma-$algebra~$\Sigma_I$ given by this construction and the $\sigma-$algebra of constructibly Borel sets introduced in \cite{GKM16} (Proposition \ref{prop::constr-Borel-is-cylinder}). Section~\ref{sec:PLMP} is devoted to the main results on Problem~\ref{GenKMP} obtained through our projective limit approach (which have been already described above). Note that projective limit techniques have been previously used to study sparse polynomials, also in relation to the moment problem \cite{K-P07, K-P09}.

Summing up, the projective limit approach seems to be a promising method in the investigation of Problem~\ref{GenKMP}, because it provides a direct bridge to a spectrum of tools and results coming from two very rich theories: on the one hand, the theory of the classical moment problem and, on the other one, the theory of projective limits. 
In \cite{IKKM} we apply this method to Problem \ref{GenKMP} for topological algebras and in future work we intend to use it to investigate \emph{truncated} versions of Problem~\ref{GenKMP}.
Indeed, we are strongly convinced that this method can bring further advances on Problem~\ref{GenKMP} also in this case, opening the way to new applications in other fields.

\newpage
\section{Preliminaries}\label{sec:Prelim}

We collect below the notions and results from the theory of projective limits needed for our purposes, mainly following \cite{S73} and \cite{Y85} (see also \cite{B55}, \cite{BouInt}, \cite{Ch58}, \cite{Prokh}). \textit{Note that the proofs in this section can be skipped by a reader already familiar with this intricate theory.}

\subsection{Cylindrical quasi-measures}\label{sec:quasi-meas}\

Let $(I, \leq)$ be a directed partially ordered set and let $\{(X_i, \Sigma_i), \pi_{i,j}, I\} $ be a \emph{projective system of measurable spaces} $(X_i, \Sigma_i)$ with maps $\pi_{i,j}: X_j \rightarrow X_i $ defined for all $i\leq j$ in $I$. Recall that this means that $\pi_{i,j}$ is measurable for all $i\leq j$ in $I$ and that $\pi_{i,j} \circ \pi_{j,k} = \pi_{i,k}$ for all $ i\leq j\leq k$ (see Figure \ref{Fig1}). As a convention, for all $i\in I$, we set $\pi_{i,i}$ to be the identity on $X_i$.
\begin{figure}[htb!]
\begin{tikzpicture}[scale=1]
\draw (-2,0)  node[centered=0.5pt] (A) {$X_j$};
\draw (2,0) node (B) {$X_i$};
\draw (0,2) node[minimum size = 0.6cm, circle]  (C) {};
\draw (0.06,2.05) node {$X_k$};
\draw[->] (A) -- (B) node[ above, pos =0.5, very thick]{$\pi_{i,j}$};
\draw[->] (C) -- (A)  node[ above=7pt , pos =0.72, very thick]{$\pi_{j,k}$} ;
\draw[->] (C) -- (B) node[ above=7pt , pos =0.72, very thick]{$\pi_{i,k}$}  ;
\end{tikzpicture}
\caption{Compatibility condition among the maps in the projective system $\{(X_i, \Sigma_i), \pi_{i,j}, I\}$ (here $i\leq j\leq k$ in $I$).}
\label{Fig1}
\end{figure}
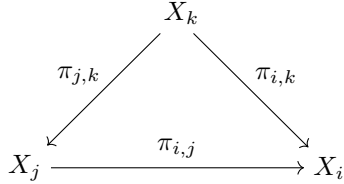

\begin{dfn}[Projective limit of measurable spaces]\label{def:projlim}\ \\
A \emph{projective limit} of the projective system $\{(X_i, \Sigma_i),\pi_{i,j}, I\} $ is a measurable space $(X_I, \Sigma_I)$ together with a family of maps $\pi_i : X_I \rightarrow X_i$ for $i\in I$ such that:\\
$\pi_{i,j}\circ \pi_{j}=\pi_i$ for all $i\leq j$ in $I$, $\Sigma_I$ is the smallest $\sigma-$algebra w.r.t.\! which all $\pi_i$'s are measurable, and the following universal property holds. For any measurable space $(Y, \Sigma_Y)$ and any family of measurable maps $f_i  :Y \rightarrow X_i$ defined for all $i\in I$ and such that $f_i=\pi_{i,j}\circ f_j$ for all $i\leq j$, there exists a unique measurable map $f : Y \rightarrow X_I$ such that $\pi_i \circ f =f_i$ for all $i\in I$, i.e., the diagram in Figure~\ref{Fig2} commutes. Moreover, $\{(X_I, \Sigma_I),\pi_i, I\}$ is unique up to isomorphisms, i.e., given any other projective limit $\{(\tilde{X}_I, \tilde{\Sigma}_I),\tilde{\pi}_i, I\}$ of the same projective system there exists a unique bijective map between $(\tilde{X}_I, \tilde{\Sigma}_I)$ and $(X_I, \Sigma_I)$ such that the map and its inverse are both measurable. 
\end{dfn}
\vspace{-0.3cm}
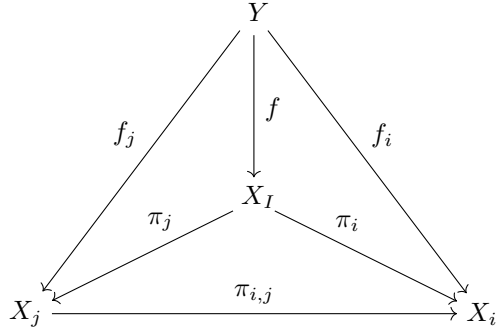
\begin{figure}[htb!]
\begin{tikzpicture}[scale=1]
\draw (-3,0.5)  node[centered=0.5pt] (A) {$X_j$};
\draw (3,0.5) node (B) {$X_i$};
\draw (0,2) node[minimum size = 0.6cm, circle]  (C) {};
\draw (0.06,2.05) node {$X_I$};
\draw (0,4.5) node[minimum size = 0.6cm, circle]  (D) {$$};
\draw (0.06,4.5) node {$Y$};
\draw[->] (A) -- (B) node[ above, pos =0.5, very thick]{$\pi_{i,j}$};
\draw[->] (C) -- (A)  node[ above=2pt , pos =0.4, very thick]{$\pi_{j}$} ;
\draw[->] (C) -- (B) node[ above=2pt , pos =0.4, very thick]{$\pi_{i}$}  ;
\draw[->] (D) -- (C) node[ right=0.2pt , pos =0.52, very thick]{$f$};
\draw[->] (D) -- (A) node[ left=5pt , pos =0.4, very thick]{$f_j$}; 
\draw[->] (D) -- (B) node[ right=5pt , pos =0.4, very thick]{$f_i$}; 
\end{tikzpicture}
\caption{Universal property of a projective limit $\{(X_I, \Sigma_I), \pi_i, I\}$ (here $i\leq j$ in $I$).}
\label{Fig2}
\end{figure}

A projective limit of a projective system of measurable spaces always exists. Indeed, if we take $X_I:=\left\{ (x_i)_{i\in I}\in\prod_{i\in I} X_i : x_i=\pi_{i,j}(x_j),\ \forall\, i\leq j\ \textrm{ in } I\right\}$, $\Sigma_I$ to be the restriction of the product $\sigma-$algebra to $X_I$ and, for any $i\in I$ we define
$\pi_i : X_I \rightarrow X_i$ to be the restriction to $X_I$ of the canonical projection $\prod_{j\in I} X_j\to X_i$, then it is easy to show that $\{(X_I, \Sigma_I), \pi_i, I\}$ fulfills Definition~\ref{def:projlim}. \par\medskip

\begin{rem}\label{rem::cofinal-proj-limit}
We can analogously define a projective system of topological spaces and its projective limit by replacing measurability with continuity in the previous definitions. Therefore, the results in the remainder of this subsection also hold for projective limits of topological spaces.
\end{rem}

It is often more convenient to take the projective limit over a cofinal subset of the index set in a given projective system. This does not change the projective limit as shown in the following proposition (which is a version of \cite[III., \S 7.2, Proposition 3]{BouST} for measurable spaces).

\begin{prop}\label{prop::cofinal-proj-limit}
Let $\{(X_I,\Sigma_I),\pi_i,I\}$ be the projective limit of the projective system $\{(X_i,\Sigma_i),\pi_{i,j},I\}$ and let $J$ be a cofinal subset of $I$, i.e., for every $i\in I$, there exists some $j\in J$ such that $j\geq i$. Then $\{(X_I,\Sigma_I),\pi_i,J\}$ is the projective limit of the projective system $\{(X_i,\Sigma_i),\pi_{i,j},J\}$.
\end{prop}
\begin{proof}
Since $J\subseteq I$, clearly $\Sigma_J\subseteq\Sigma_I$. The converse inclusion follows by the cofinality of $J$ in $I$ and the compatibility of the $\pi_i$'s. To show the universal property of $\{(X_I, \Sigma_I), \pi_i, J\}$, for each $j\in J$, we consider a measurable map $f_j:Y\to X_j$ such that $f_j=\pi_{j,k}\circ f_k$ whenever $j\leq k$ in $J$. By the cofinality of $J$ in $I$, for each $i\in I$ there is some $j\in J$ such that $\pi_i=\pi_{i,j}\circ\pi_j$ and so defining $f_i:=\pi_{i,j}\circ f_j$ we can deduce the universal property of $\{(X_I,\Sigma_I), \pi_i, J\}$ from the one of $\{(X_I,\Sigma_I),\pi_i,I\}$.
\end{proof}

The following definitions are well-known. Below we sightly adapt the terminology of \cite[Part I, Chapter I, \S 10, Part II, Chapter II, \S 1]{S73} to our purposes.

\begin{dfn}\label{def-cyl}
Let $\{(X_i, \Sigma_i),\pi_{i,j}, I\}$ be a projective system of measurable spaces and $\{(X_I, \Sigma_I), \pi_i, I\}$ its projective limit.
\begin{enumerate}[(a)]
\item A \emph{cylinder set }in $X_I$ is a set of the form $\pi_i^{-1}(E)$ for some $i\in I$ and $E\in\Sigma_i$. In other words, $\pi_i^{-1}(E)$ belongs to the $\sigma-$algebra $\pi_i^{-1}(\Sigma_i)$ generated by $\pi_i$.
\item The collection $\C_I$ of all the cylinders sets in $X_I$ forms an algebra, which is usually called the \emph{cylinder algebra} on $X_I$. In general, $\C_I$ is not a $\sigma-$algebra.
\item We call \emph{cylinder $\sigma-$algebra} on $X_I$ the smallest $\sigma-$algebra containing all the cylinders sets in $X_I$, which clearly coincides with $\Sigma_I$.
\end{enumerate}
\end{dfn}

We introduce a two-stage construction of the projective limit $\{(X_I,\Sigma_I),\pi_i,I\}$ of a projective system $\P_I:=\{(X_i,\Sigma_i),\pi_{i,j},I\}$ of measurable spaces, which will be particularly useful when $I$ does not contain a countable cofinal subset (see Proposition~\ref{lem::projective-limit}, Lemma \ref{cor::double-proj-lim-to-char-space} and Section \ref{main-results-gen}). This alternative construction can be obtained by adapting \cite[III., Exercises for \S 7, 1.\! (p.\! 251)]{BouST} to projective limits of measurable spaces. As this construction will be used later on, we provide all its steps and a diagram (see Figure~\ref{Fig3}) to facilitate the reading and fix the notation.
Consider 
$$
\Lambda:=\{\lambda\subseteq I: \lambda\text{ contains a countable directed subset cofinal in }\lambda\}
$$
partially ordered by inclusion. Note that $I$ contains a countable cofinal subset if and only if $\Lambda$ has a maximal element.
\begin{prop}\label{prop::ctlb-Lambda}
For $J\subseteq I$ countable, there exists $\lambda\in\Lambda$ such that $J\subseteq\lambda$.
\end{prop}
\begin{proof}
Let $J=(i_n)_{n\in\NN}$ and $l_1:=i_1$. As $I$ is directed, there exists $l_2\in I$ such that $i_2\leq l_2$ and $l_1\leq l_2$. Proceeding in this way, we get $(l_n)_{n\in\NN}\subseteq I$ such that $i_n\leq l_n$ and $l_{n-1}\leq l_n$ for all $n\geq 2$. This shows that $(l_n)_{n\in\NN}$ is a countable directed cofinal subset of $\lambda:=J\cup (l_n)_{n\in\NN} \subseteq I$. Hence, $\lambda\in\Lambda$ and $J\subseteq\lambda$.
\end{proof}

\begin{prop}\label{prop::Lambda-directed}
If $(\lambda_n)_{n\in\NN}\subseteq \Lambda$, then there exists $\lambda\in\Lambda$ such that $\lambda_n\subseteq\lambda$ for all $n\in\NN$. In particular, $\Lambda$ is directed.
\end{prop}
\begin{proof}
For each $n\in\NN$, let $\theta_n$ be a countable cofinal subset of $\lambda_n$. As $\bigcup_{n\in\NN}\theta_n$ is a countable subset of $I$, by Proposition \ref{prop::ctlb-Lambda} there exists $\theta\in\Lambda$ such that $\bigcup_{n\in\NN}\theta_n\subseteq\theta$. Then $\lambda:=\bigcup_{n\in\NN}\lambda_n\cup\theta\subseteq I$ belongs to $\Lambda$ and clearly contains each $\lambda_n$. 
\end{proof}

For any $\lambda\in\Lambda$, let $\{(X_\lambda,\Sigma_\lambda),\pi_i^\lambda,\lambda\}$ denote the projective limit of the projective system $\P_\lambda:=\{(X_i,\Sigma_i),\pi_{i,j},\lambda\}$. For any $\lambda\subseteq\kappa$ in $\Lambda$ and all $i\leq j$ in $\lambda$, we have
$$
\pi_i^\kappa=\pi_{i,j}\circ\pi_j^\kappa
$$
and so by the universal property of the projective limit $\{(X_\lambda,\Sigma_\lambda),\pi_i^\lambda,\lambda\}$ of $\P_\lambda$, there exists a unique measurable map
\begin{equation}\label{eq::2}
\pi^{\lambda,\kappa}:X_\kappa\to X_\lambda\text{ such that }\pi_i^\lambda\circ\pi^{\lambda,\kappa}=\pi_i^\kappa\text{ for all }i\in\lambda.
\end{equation}
Using the uniqueness of this map, we easily get the following proposition.

\begin{prop}\label{prop::projective-system}
The family $\P_\Lambda:=\{(X_\lambda,\Sigma_\lambda),\pi^{\lambda,\kappa},\Lambda\}$ is a projective system of measurable spaces.
\end{prop}

For any $\lambda\in\Lambda$ and all $i\leq j$ in $\lambda$, we have $\pi_i=\pi_{i,j}\circ\pi_j$ and so by the universal property of the projective limit $\{(X_\lambda,\Sigma_\lambda),\pi_i^\lambda,\lambda\}$ of $\P_\lambda$, there exists a unique measurable map
\begin{equation}\label{eq::3}
\pi^\lambda: X_I\to X_\lambda\text{ such that }\pi_i^\lambda\circ\pi^\lambda=\pi_i\text{ for all }i\in\lambda.
\end{equation}

\begin{prop}\label{lem::projective-limit}
$\{(X_I,\Sigma_I),\pi^\lambda,\Lambda\}$ is the projective limit of $\P_\Lambda$.
\end{prop}

This result shows that the two-stage construction, in which for each $\lambda\in\Lambda$ we take a projective limit $\{(X_\lambda,\Sigma_\lambda),\pi_i^\lambda,\lambda\}$ of the projective system $\{(X_i,\Sigma_i),\pi_{i,j},\lambda\}$ and then again take a projective limit of the projective limits $\{(X_\lambda,\Sigma_\lambda),\pi^{\lambda,\kappa},\Lambda\}$ obtained in the first stage, leads to the same space as directly taking a projective limit of the projective system $\{(X_i,\Sigma_i),\pi_{i,j}, I\}$.  
\enlargethispage{0.4cm}
\vspace{-0.2cm}
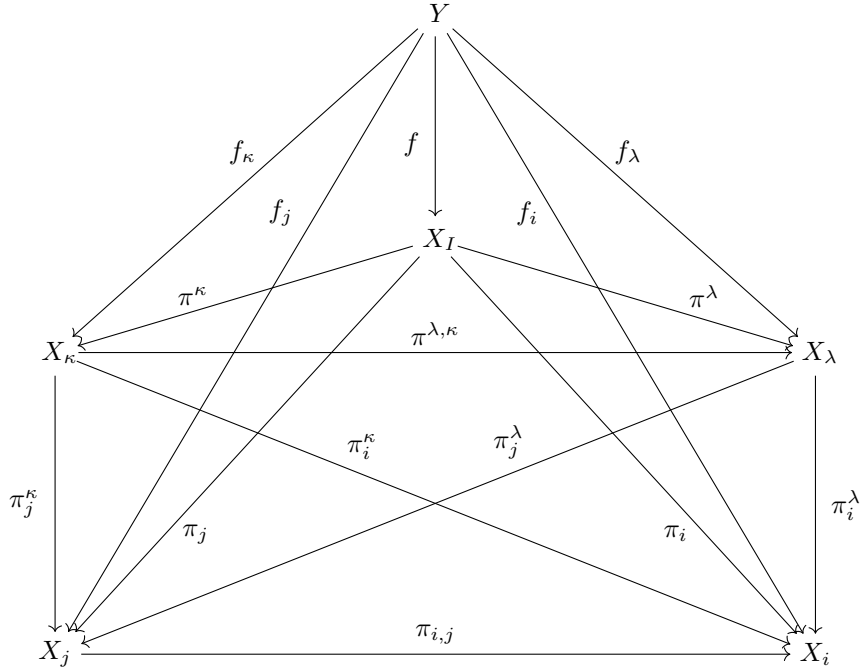
\begin{figure}[h]
\begin{tikzpicture}[scale=1]
\draw (-5,0.5)  node[centered=0.5pt] (A) {$X_j$};
\draw (5,0.5) node (B) {$X_i$};
\draw (0,6) node[minimum size = 0.6cm, circle]  (C) {};
\draw (C)++(0.06,0)  node {$X_I$};
\draw (-5,4.5) node[minimum size = 0.6cm, circle]  (D) {$$};
\draw (D)++(0.06,0) node {$X_\kappa$};
\draw (+5,4.5) node[minimum size = 0.6cm, circle]  (E) {$$};
\draw (E)++(0.06,0) node {$X_\lambda$};
\draw (0,9) node[minimum size = 0.6cm, circle]  (F) {$$};
\draw (F)++(0.06,0) node {$Y$};
\draw[->] (A) -- (B) node[ above, pos =0.5, very thick]{$\pi_{i,j}$};
\draw[->] (C) -- (A)  node[ below=5pt , pos =0.65, very thick]{$\pi_{j}$} ;
\draw[->] (C) -- (B) node[ below=5pt , pos =0.65, very thick]{$\pi_{i}$}  ;
\draw[->] (C) -- (D) node[ left=10pt , pos =0.5, very thick]{$\pi^\kappa$};
\draw[->] (C) -- (E) node[ right=20pt , pos =0.5, very thick]{$\pi^\lambda$};
\draw[->] (D) -- (A) node[ left=2pt , pos =0.5, very thick]{$\pi^\kappa_j$}; 
\draw[->] (E) -- (B) node[ right=2pt , pos =0.5, very thick]{$\pi^\lambda_i$}; 
\draw[->] (D) -- (E) node[ above=-1pt , pos =0.5, very thick]{$\pi^{\lambda,\kappa}$}; 
\draw[->] (D) --(B) node[ above=2pt , pos =0.4, very thick]{$\pi^\kappa_{i}$};
\draw[->] (E) --(A) node[ above=2pt , pos =0.4, very thick]{$\pi^\lambda_{j}$};
\draw[->] (F) --(A) node[ left=5pt , pos =0.3, very thick]{$f_{j}$};
\draw[->] (F) --(B) node[ left=2pt , pos =0.3, very thick]{$f_i$};
\draw[->] (F) --(C) node[ left=2pt , pos =0.6, very thick]{$f$};
\draw[->] (F) --(D) node[ left=5pt , pos =0.4, very thick]{$f_\kappa$};
\draw[->] (F) --(E) node[ right=5pt , pos =0.4, very thick]{$f_\lambda$};
\end{tikzpicture}\vspace{-0.3cm}
\caption{Two-stage construction of $(X_I, \Sigma_I)$ (here $\lambda\subseteq\kappa$ in $\Lambda$ and $i\leq j$ in $\lambda$).}
\label{Fig3}
\end{figure}

\begin{proof}[Proof of Proposition \ref{lem::projective-limit}]\ \\
Denote by $\Sigma_\Lambda$ the cylinder $\sigma-$algebra on $X_\indexset$ w.r.t.\! $\P_\superindexset$, i.e., the smallest $\sigma-$algebra on $X_I$ such that $\pi^\lambda$ is measurable for all $\lambda\in\Lambda$. We first prove that $\Sigma_\Lambda=\Sigma_I$. As $\pi^\lambda:(X_I,\Sigma_I)\to (X_\lambda,\Sigma_\lambda)$ is measurable for all $\lambda\in\Lambda$, we clearly have $\Sigma_\Lambda\subseteq\Sigma_I$. Conversely, for each $i\in I$ there exists $\lambda\in\Lambda$ such that $i\in\lambda$ and so for any $E_i\in\Sigma_i$ we have
$$
\pi_i^{-1}(E_i)=(\pi^\lambda)^{-1}((\pi_i^\lambda)^{-1}(E_i))\in (\pi^\lambda)^{-1}(\Sigma_\lambda)\subseteq\Sigma_\Lambda.
$$
Hence, $\Sigma_I\subseteq\Sigma_\Lambda$ since $\Sigma_I$ is generated by sets of the form $\pi_i^{-1}(E_i)$ for some $i\in I$ and $E_i\in\Sigma_i$.

It remains to show that the universal property holds for $\{(X_\indexset, \Sigma_\indexset), \pi^\lambda, \Lambda\}$. Let $(Y,\Sigma_Y)$ be a measurable space and for any $\lambda\in\superindexset$ let $f^\lambda:Y\to X_\lambda$ be a measurable map such that $f^\lambda=\pi^{\lambda,\kappa}\circ f^\kappa$ for all $\lambda\subseteq\kappa$ in $\Lambda$. Let $i\in I$. Then there is $\lambda\in\Lambda$ such that $i\in\lambda$ and the map $f_i: Y\to X_i$ given by $
f_i:=\pi_i^\lambda\circ f^\lambda
$
is well-defined and measurable.
The universal property of the projective limit $\{(X_I,\Sigma_I),\pi_i,I\}$ of $\P_I$ yields that there exists a unique measurable map
\begin{equation}\label{eq::1}
f:Y\to X_I\text{ such that }\pi_i\circ f=f_i \text{ for all }i\in I.
\end{equation}
For any $\lambda\in\superindexset$ the universal property of the projective limit $\{(X_\lambda,\B_\lambda),\pi_i^\lambda,\lambda\}$ of $\P_\lambda$ implies that there exists a unique measurable map
\begin{equation}\label{eq::4}
h: Y\to X_\lambda\text{ such that }\pi_i^\lambda\circ h=f_i\text{ for all }i\in\lambda.
\end{equation}
Since for all $i\in\lambda$, we have that
$$
\pi_i^\lambda\circ f^\lambda = f_i = \pi_i^\lambda\circ h \quad\text{and}\quad \pi_i^\lambda\circ(\pi^\lambda\circ f)=\pi_i\circ f =f_i = \pi_i^\lambda\circ h,
$$
the uniqueness of $h$ in \eqref{eq::4} yields $f^\lambda=h=\pi^\lambda\circ f$.

The uniqueness of $f$ in \eqref{eq::1} allows to easily show that $f$ is the unique map from $Y$ to $X_I$ such that $f^\lambda=\pi^\lambda\circ f$ holds for all $\lambda\in\Lambda$. 
\end{proof}

\begin{rem}\label{rem::Lambda}
By Proposition \ref{prop::cofinal-proj-limit}, Proposition \ref{lem::projective-limit} also holds if $\Lambda$ is replaced with a cofinal subset of $\Lambda$. 
\end{rem}

\begin{prop}\label{prop::C-Lambda}
The relation $\C_I\subseteq\Sigma_I=\Sigma_\Lambda=\C_\Lambda$ holds, where $\C_\Lambda$ denotes the cylinder algebra on $X_\indexset$ w.r.t.\! the projective system $\P_\Lambda$. Note that the inclusion is strict in general.
\end{prop}
\begin{proof}
To show that $\Sigma_\Lambda=\C_\Lambda$, it is enough to prove that the algebra $\C_\Lambda$ is in fact a $\sigma-$algebra as $\C_\Lambda$ generates $\Sigma_\Lambda$. For each $n\in\NN$, let $\lambda_n\in\Lambda$ and $E_n\in\Sigma_{\lambda_n}$. Then Proposition \ref{prop::Lambda-directed} ensures there exists $\kappa\in\Lambda$ such that $\lambda_n\subseteq\kappa$ for all $n\in\NN$. Hence,
$$
\bigcup_{n\in\NN} (\pi^{\lambda_n})^{-1}(E_n) =\!\bigcup_{n\in\NN} (\pi^\kappa)^{-1}((\pi^{\lambda_n,\kappa})^{-1}(E_n))= (\pi^\kappa)^{-1}\!\left(\bigcup_{n\in\NN}(\pi^{\lambda_n,\kappa})^{-1}(E_{n})\right)\!\in\C_\Lambda,
$$
since $\bigcup_{n\in\NN}(\pi^{{\lambda_n},\kappa})^{-1}(E_n)\in\Sigma_\kappa$.
\end{proof}

\begin{dfn} Let $\P:=\{(X_i,\Sigma_i), \pi_{i,j}, I\}$ be a projective system of measurable spaces and $\{(X_I,\Sigma_I), \pi_{i}, I\}$ its projective limit.
\begin{enumerate}[(a)]
\item An \emph{exact projective system of measures on $\P$} is a family $\{\mu_i: i\in I\}$ such that each $\mu_i$ is a measure on $\Sigma_i$ and ${\pi_{i,j}}_{\#}\mu_j=\mu_i$ for all $i\leq j$ in $I$, where ${\pi_{i,j}}_{\#}\mu_j$ denotes the pushforward of $\mu_j$ w.r.t.\! $\pi_{i,j}$, i.e., ${\pi_{i,j}}_{\#}\mu_j(E_i)=\mu_j(\pi_{i,j}^{-1}(E_i))$ for all $E_i\in\Sigma_i$. 
\item An exact projective system $\{\mu_i: i\in I\}$ of measures on $\P$ is called \emph{thick} if for each $i\in I$ the set $\pi_i(X_I)$ is thick w.r.t. $\mu_i$, i.e. $\mu_i(M)=0$ for all $M\in\Sigma_i$ such that $\pi_i(X_I)\cap M=\emptyset$.
\item A \emph{cylindrical quasi-measure $\mu$ w.r.t.\!~$\P$} is a set function $\mu$ on $\C_I$ such that ${\pi_i}_{\#}\mu$ is a measure on $\Sigma_i$ for all $i\in I$. 
\end{enumerate}
\end{dfn}

\begin{rem}\label{prop::cyl-quasi-measure}\
\begin{enumerate}[(i)]
\item If for all $i\in I$ the map $\pi_i:X_I\to X_i$ is surjective, then any exact projective system of measures on $\P$ is thick.
\item Cylindrical quasi-measures are \emph{not} measures, since they are finitely additive set functions which are not necessarily $\sigma-$additive for countable disjoint unions of sets in $\C_I$ belonging to $\C_I$. In other words, a cylindrical quasi-measure is in general not even a pre-measure. Note that cylindrical quasi-measures are often addressed as cylinder measures (see, e.g., \cite{GV64}, \cite{Schmu-new}, \cite{S73}). 
\end{enumerate}
\end{rem}

\begin{lem}\label{lem::corres-cylqm-thickexact}
Let $\P:=\{(X_i,\Sigma_i), \pi_{i,j}, I\}$ be a projective system of measurable spaces. There is a one-to-one correspondence between cylindrical quasi-measures $\mu$ and thick exact projective systems $\{\mu_i:i\in I\}$ of measures w.r.t.\! $\P$, given by $\mu(\pi_i^{-1}(E_i))=\mu_i(E_i)$ for all $i\in I$ and all $E_i\in\Sigma_i$. We call $\mu$ the \emph{cylindrical quasi-measure corresponding to} $\{\mu_i:i\in I\}$.
\end{lem}
\begin{proof}
Let $\mu$ be a cylindrical quasi-measure w.r.t.\! $\P$. Then $\{{\pi_i}_{\#}\mu:i\in I\}$ is an exact projective system of measures on~$\P$ since ${\pi_{i,j}}_{\#}({\pi_j}_{\#}\mu)={\pi_i}_{\#}\mu$ for all $i\leq j$ in~$I$. Further, for any $i\in I$ and any $E_i\in \Sigma_i$ with $E_i\cap \pi_i(X_I)=\emptyset$, we have $\pi_i^{-1}(E_i)=\emptyset$. Thus, ${\pi_i}_{\#}\mu(E_i)=\mu(\emptyset)={\pi_i}_{\#}\mu(\emptyset)=0$, i.e., $\{ {\pi_i}_{\#}\mu:i\in I\}$ is thick.

Conversely, let $\{\mu_i: i\in I\}$ be a thick exact projective system of measures on $\P$, the map $\mu:\C_I\to\RR$ given by $\mu(\pi_i^{-1}(E_i)):=\mu_i(E_i)$ for all $i\in I$ and $E_i\in\Sigma_i$ is well-defined. Indeed, let $i,j\in I$ and $E_i\in\Sigma_i, E_j\in\Sigma_j$ be such that $\pi_i^{-1}(E_i)=\pi_j^{-1}(E_j)$. Since $I$ is directed, there exists $k\in I$ such that $i\leq k$ and $j\leq k$ and so $\pi_k^{-1}(\pi_{i,k}^{-1}(E_i))=\pi_k^{-1}(\pi_{j,k}^{-1}(E_j))$. This implies that $\pi_{i,k}^{-1}(E_i)\cap \pi_k(X_I)=\pi_{j,k}^{-1}(E_j)\cap \pi_k(X_I)$, i.e., $(\pi_{i,k}^{-1}(E_i) \triangle \pi_{j,k}^{-1}(E_j))\cap\pi_k(X_I)=\emptyset$. Thus, we get that $\mu_k(\pi_{i,k}^{-1}(E_i))=\mu_k(\pi_{j,k}^{-1}(E_j))$ as $\pi_k(X_I)$ is thick w.r.t.\ $\mu_k$. The exactness of the system $\{\mu_i: i\in I\}$ then yields that
$$
\mu_i(E_i)=\mu_k(\pi_{i,k}^{-1}(E_i))=\mu_k(\pi_{j,k}^{-1}(E_j))=\mu_j(E_j).
$$
By definition, $\mu$ satisfies ${\pi_i}_{\#}\mu=\mu_i$ for all $i\in I$.
\end{proof}

In the following, we will discuss when cylindrical quasi-measures can be extended to measures on $\Sigma_I$. For this purpose, we will restrict our attention to projective systems of measurable spaces associated to projective systems of topological spaces.

\subsection{Prokhorov's condition}\label{sec:prokh-cond}\

Let $(I, \leq)$ be a directed partially ordered set and $\T:=\{(X_i, \tau_i), \pi_{i,j}, I\}$ a projective system of Hausdorff topological spaces. Denote by $\{(X_I, \tau_I), \pi_i, I\}$ a projective limit of $\T$. Then $(X_I, \tau_I)$ is Hausdorff and we can either equip it with the Borel $\sigma-$algebra $\B_I$ w.r.t.\! $\tau_I$ or with the cylinder $\sigma-$algebra $\Sigma_I$ obtained by taking the projective limit of the projective system $\P_{\T}:=\{(X_i, \B_i), \pi_{i,j},  I\}$, where $\B_i$ is the Borel $\sigma-$algebra on $X_i$ w.r.t.\!~$\tau_i$. Note that $\Sigma_I\subseteq\B_I$. 

\begin{prop}\label{prop::Tobias}
If $I$ contains a countable cofinal subset $J$, then $\Sigma_I=\B_I$.
\end{prop}
\begin{proof}
Since $\B_I$ is generated by $\tau_I$, it suffices to show $\tau_I\subseteq\Sigma_I$. Let $U\in\tau_I$. As the open cylinder sets form a basis of $\tau_I$, for each $i\in I$ there exists $U_i\in\tau_i$ such that $U=\bigcup_{i\in I}\pi_i^{-1}(U_i)$. As $J$ is cofinal in $I$, we have $I=\bigcup_{j\in J}\{i\in I:i\leq j\}$. Therefore,
$$
U=\bigcup_{i\in I}\pi_i^{-1}(U_i)=\bigcup_{j\in J}\bigcup_{\{i\in I:i\leq j\}}\pi_j^{-1}(\pi_{i,j}^{-1}(U_i))=\bigcup_{j\in J}\pi_j^{-1}\left(\bigcup_{\{i\in I:i\leq j\}} \pi_{i,j}^{-1}(U_i)\right)
$$
and so $U$ is a countable union of sets in $\Sigma_I$ as $\bigcup_{\{i\in I:i\leq j\}} \pi_{i,j}^{-1}(U_i)\in\tau_j$ for each $j\in J$. Hence, $U\in\Sigma_I$.
\end{proof}

To introduce Prokhorov's condition (see~\cite{Prokh}), we recall that a \emph{Radon measure} on an arbitrary Hausdorff topological space $(X, \tau)$ is a measure $\mu$ defined on the Borel $\sigma-$algebra $\B_{\tau}$ on $X$ w.r.t.\! $\tau$ such that $\mu$ is locally finite (i.e., every point in $X$ has a neighborhood of finite measure) and inner regular w.r.t.\! compact subsets of $X$ (i.e., for any $M\in\B_{\tau}$, $\mu(M)=\sup\{\mu(K): K\subseteq M \text{ compact}\}$).

\begin{dfn}\ 
An \emph{exact projective system of Radon probability measures w.r.t.\!~$\T$ (or on $\P_\T$)} is an {exact projective system $\{\mu_i: i\in I\}$ of measures on~$\P_\T$} such that $\mu_i$ is Radon measure on $\B_i$ and $\mu_i(X_i)=1$ for all $i\in I$.
\end{dfn}

We say that an exact projective system $\{\mu_i: i\in I\}$ of Radon probability measures w.r.t.\! $\T$ fulfills \textbf{Prokhorov's condition} if the following holds
\be\label{epsilon-K}
\forall \varepsilon>0\ \exists\ C_{\varepsilon}\subset X_I \text{ compact s.t. } \forall i\in I: \ \mu_i(\pi_i(C_{\varepsilon}))\geq 1-\varepsilon.
\ee 

\begin{rem}\label{epsilon-K->thick}
Any exact projective system $\{\mu_i: i\in I\}$ of Radon probability measures fulfilling Prokhorov's condition~\eqref{epsilon-K} is thick. Indeed, for $i\in I$ and $M\in\Sigma_i$ with $M\cap\pi_i(X_I)=\emptyset$, we have that for every $\varepsilon>0$ also $M\cap\pi_i(C_{\varepsilon})=\emptyset$ holds and so $\mu_i(M)\leq 1-\mu_i(\pi_i(C_{\varepsilon}))\leq \varepsilon$, which implies in turn that $\mu_i(M)=0$.
\end{rem}

In the following we provide sufficient conditions for Prokhorov's condition \eqref{epsilon-K} to be fulfilled. To this purpose, let us show a preparatory lemma (which is essentially \cite[I., \S 9.6, Proposition~8]{BouGT}).
\begin{lem}\label{lem::intersection-compact}
For $i\in I$, let $K_i\subseteq X_i$ be compact such that $\pi_{i,j}(K_j)\subseteq K_i$ for all $j\in I$ with ${j\geq i}$. Then $K:=\bigcap_{i\in I}\pi_i^{-1}(K_i)$ is compact and, for each $i\in I$,
$
\pi_i(K)=\bigcap_{j\geq i}\pi_{i,j}(K_j).
$
\end{lem}
\begin{proof}
By assumption $\{(K_i,\tau_i\cap K_i),\pi_{i,j}\!\restriction_{K_i},I\}$ is a projective system of compact spaces and its projective limit (seen as a subset of $\prod_{i\in I}K_i\subseteq\prod_{i\in I}X_i$) is identified with $\bigcap_{i\in I}\pi_i^{-1}(K_i)$. Then the conclusion follows from \cite[I., \S 9.6, Proposition 8]{BouGT}.
\end{proof}

\begin{prop}\label{prop::compact-support-epsilon-K}
An exact projective system $\{\mu_i:i\in I\}$ of Radon probability measures w.r.t.\! $\T$ fulfills Prokhorov's condition~\eqref{epsilon-K} if and only if for any $\varepsilon>0$ and each $i\in I$ there exists $K_i\subseteq X_i$ compact such that $\mu_i(K_i)\geq 1-\varepsilon$ and $\pi_{i,j}(K_j)\subseteq K_i$ for all $j\in I$ with $j\geq i$.
\end{prop}

\begin{proof}
Let $\varepsilon>0$. Suppose that there exists $C_\varepsilon\subseteq X_I$ compact such that $\mu_i(\pi_i(C_\varepsilon))\geq 1-\varepsilon$ for all $i\in I$. Then setting $K_i:=\pi_i(C_\varepsilon)$ yields the conclusion.

Conversely, suppose that for each $i\in I$ there exists $K_i\subseteq X_i$ compact such that $\mu_i(K_i)\geq 1-\varepsilon$ and $\pi_{i,j}(K_j)\subseteq K_i$ for all $j\in I$ with $j\geq i$. Then Lemma \ref{lem::intersection-compact} guarantees that $K:=\bigcap_{i\in I}\pi_i^{-1}(K_i)$ is compact. 

Fix $i\in I$. By the inner regularity of $\mu_i$ w.r.t.\! compact subsets of $X_i$, there exists $C\subseteq X_i\setminus\pi_i(K)$ compact such that $\mu_i(X_i\setminus\pi_i(K))\leq \mu_i(C)+\varepsilon$. Using again Lemma \ref{lem::intersection-compact}, we obtain that $\bigcup_{j\in I, j\geq i}\left(X_i\setminus\pi_{i,j}(K_j)\right)$ is an open cover of $C$. Hence, using that $C$ is compact and that $I$ is directed, we get that there exists $k\in I$ such that $k\geq i$ and $C\subseteq X_i\backslash \pi_{i,k}(K_k)$. This implies that
$$
\mu_i(\pi_i(K))\geq \mu_i(\pi_{i,k}(K_k))-\varepsilon =\mu_{k}(\pi_{i,k}^{-1}(\pi_{i,k}(K_k)))-\varepsilon \geq \mu_k(K_k)-\varepsilon \geq 1-2\varepsilon,
$$
which proves that \eqref{epsilon-K} holds with $C_{2\varepsilon}=K$.
\end{proof}

A sketch of the proof of Proposition \ref{prop-eps-K} below can be found in \cite[Corollary, p.\! 81]{S73} (cf.\! \cite[IX., \S 4.3, Theorem 2]{BouInt}), but we include here the details for the convenience of the reader as this result is the backbone of the next section.

\begin{prop}\label{prop-eps-K}
If $I$ contains a countable cofinal subset, then any exact projective system of Radon probability measures $\{\mu_i: i\in I\}$ w.r.t.\!~$\T$ fulfills Prokhorov's condition~\eqref{epsilon-K}.
\end{prop}

\proof
Let us first consider the case when $I=\NN$. 

Since $\mu_1$ is inner regular w.r.t.\! compact subsets of $X_1$, there exists $K_1\subset X_1$ compact such that 
$\mu_1(K_1) \geq 1- \frac{\varepsilon}{2}.$
As $\pi_{1,2}^{-1}(K_1)\in\B_2$, using the inner regularity of $\mu_2$ w.r.t.\! compact subsets of $X_2$, we can choose $K_2\subset \pi_{1,2}^{-1}(K_1)$ compact such that $\mu_2(K_2) \geq \mu_2(\pi_{1,2}^{-1}(K_1))- \frac{\varepsilon}{2^2}=\mu_1(K_1)-\frac{\varepsilon}{2^2}.$
Hence, we can inductively construct a sequence $(K_n)_{n\in\NN}$ such that for each $n\in\NN$: $K_n\subset X_n$ is compact, $K_{n} \subseteq \pi_{n-1,n}^{-1}(K_{n-1})$ and 
$$
\mu_{n+1}(K_{n+1})\geq \mu_{n}(K_{n}) - \frac{\varepsilon}{2^{n+1}}.
$$
By iterating, we get that for all $n\in\NN$
\begin{equation}\label{eq::rec-ind}
\mu_n(K_n)\geq 1-\varepsilon\sum_{k=1}^n2^{-k}\geq 1-\varepsilon.
\end{equation}

Fixed $m\in\NN$, the sequence $(\pi_{m,n}(K_n))_{n\in\NN}$ is a decreasing sequence of compact subsets of $X_m$, as for all $n\geq m$ we get
$\pi_{m, n+1}(K_{n+1})\subseteq \pi_{m, n+1}(\pi_{n,{n+1}}^{-1}(K_{n}))\subseteq\pi_{m, n}(K_n).$
Then, Lemma \ref{lem::intersection-compact} guarantees that $C_\varepsilon:= \bigcap_{n\in\NN} \pi_n^{-1}(K_n)\subset X_I$ is compact and 
\be\label{limite2}
\pi_m(C_\varepsilon)=\bigcap_{n\geq m}\pi_{m,n}(K_n).
\ee
Combining \eqref{limite2} with the continuity of $\mu_m$, we get that
$
\lim\limits_{n\to\infty}\mu_m(\pi_{m,n}(K_n))=\mu_m(\pi_m(C_\varepsilon)).
$
This together with the fact that for any $n\geq m$
$$\mu_m(\pi_{m,n}(K_n))= \mu_n(\pi_{m,n}^{-1}(\pi_{m,n}(K_n)))\geq  \mu_n(K_n)\stackrel{\eqref{eq::rec-ind}}{\geq }1-\varepsilon$$
implies that $\mu_m(\pi_m(C_\varepsilon))\geq 1-\varepsilon$. Hence, \eqref{epsilon-K} holds for $I=\NN$.

Consider a countable cofinal subset $J$ of $I$. Then $X_I=X_J$ as topological spaces (see Proposition \ref{prop::cofinal-proj-limit} and Remark \ref{rem::cofinal-proj-limit}) and the result easily follows.
\endproof

\subsection{Extension to the cylinder $\sigma-$algebra $\Sigma_I$}\label{sec:ext-cyl}\

The problem of extending cylindrical quasi-measures to actual measures has been extensively studied since Kolmogorov's result \cite{K-33} (see, e.g., \cite{B55, M63, F72}). In this subsection we gather in Theorem~\ref{ext-cylinderSigmaAlg} the most important result needed for our purposes. We follow the exposition in \cite[Theorems 8.2~and~9.3]{Y85}.
We will first consider the case of a projective system of Hausdorff topological spaces whose index set contains a countable cofinal subset and then proceed with the general case.

\begin{lem}\label{lem::content}
Let $(I, \leq)$ be a directed partially ordered set, $\T:=\{(X_i, \tau_i), \pi_{i,j}, I\}$ a projective system of Hausdorff topological spaces, and $\{(X_I, \tau_I), \pi_i, I\}$ its projective limit.
If $I$ contains a countable cofinal subset, then any exact projective system $\{\mu_i: i\in I\}$ of Radon probability measures w.r.t.\!~$\T$ is thick and the corresponding cylindrical quasi-measure uniquely extends to a probability measure on~$\B_I$.
\end{lem}
%\Patrick{Do we want to prove here that the extension of the corresponding cylindrical quasi-measure is, in fact, a Radon measure?}
\begin{proof}
Let $\{\mu_i: i\in I\}$ be an exact projective system of Radon probability measures w.r.t.\!~$\T$. By combining Proposition \ref{prop-eps-K} with Remark \ref{epsilon-K->thick}, we easily get that $\{\mu_i: i\in I\}$ is thick. Let $\mu$ be the corresponding cylindrical quasi-measure, which is defined by
$\mu(\pi_i^{-1}(E_i)):=\mu_i(E_i)$ for all $i\in I$ and all $E_i\in\B_I$ and is finite since $\mu(X_I)=\mu(\pi_i^{-1}(X_i))=\mu_i(X_i)=1$.
Showing the $\sigma-$additivity of $\mu$ is equivalent to show that any decreasing sequence $(\tilde{E}_n)_{n\in\NN}$ in $\C_I$ with $\inf_{n\in\NN}\mu( \tilde{E}_n) >0$ satisfies $\bigcap_{n\in\NN} \tilde{E}_n \neq \emptyset$. 

Let $0<\varepsilon<\inf_{n\in\NN} \mu(\tilde{E}_n)$. For any $n\in\NN$, there exists $i_n\in I$ and $E_n\in\B_{i_n}$ such that $\tilde{E}_n=\pi_{i_n}^{-1}(E_n)$. Since $I$ is directed, we can always assume that $(i_n)_{n\in\NN}$ is increasing. Then, by the inner regularity of $\mu_{i_n}$ w.r.t.\! compact subsets of $X_{i_n}$, there exists $K_n\subset E_n$ compact such that 
$\mu_{i_n}(K_n) >\varepsilon$ and $\mu_{i_n}(K_n) > \mu_{i_n}(E_n)-\frac{\varepsilon}{4^n}$. 
For each $n\in\NN$, consider the closed set $\tilde{K}_n := \pi_{i_n}^{-1}(K_n)$ and note that 
\be\label{inters}
\tilde{K}_j\cap \tilde{K}_n =\tilde{K}_n \setminus (\tilde{E}_j\setminus \tilde{K}_j), \quad\forall j\leq n.
\ee
Then, for any $n\in\NN$, we get that
$$
\tilde{K}_n \setminus \bigcup_{j=1}^{n-1} (\tilde{E}_j\setminus \tilde{K}_j)=\bigcap_{j=1}^{n-1}\left( \tilde{K}_n \setminus (\tilde{E}_j\setminus \tilde{K}_j)\right) \stackrel{\eqref{inters}}{=}\bigcap_{j=1}^{n-1} \left(\tilde{K}_j\cap \tilde{K}_n\right)= \bigcap_{j=1}^{n} \tilde{K}_j.
$$
Considering now the closed cylinder set $\tilde{K}'_n:=\bigcap_{j=1}^n \tilde{K}_j\in\C_I$ and setting $F_n:=\bigcap_{j=1}^n \pi_{{i_j},{i_n}}^{-1}(K_j)\in X_{i_n}$ yields the following $$\tilde{K}'_n=\bigcap_{j=1}^n \pi_{i_j}^{-1}(K_j)=\pi_{i_n}^{-1}\left(\bigcap_{j=1}^n \pi_{{i_j},{i_n}}^{-1}(K_j)\right)=\pi_{i_n}^{-1}(F_n).$$ This implies $
\tilde{K}'_n\subseteq \pi_{i_n}^{-1}(\pi_{i_n}(\tilde{K}'_n))=\pi_{i_n}^{-1}(\pi_{i_n}(\pi_{i_n}^{-1}(F_n)))\subseteq \pi_{i_n}^{-1}(F_n)=\tilde{K}'_n
$ and so $\tilde{K}'_n=\pi_{i_n}^{-1}(\pi_{i_n}(\tilde{K}'_n))$. We obtain
\begin{eqnarray*}
\mu(\tilde{K}'_n)&=&\mu\left(\tilde{K}_n \setminus \bigcup_{j=1}^{n-1} (\tilde{E}_j\setminus \tilde{K}_j)\right)
\geq\mu(\tilde{K}_n)- \mu\left( \bigcup_{j=1}^{n-1} (\tilde{E}_j\setminus \tilde{K}_j)\right)\\
&\geq&\mu(\tilde{K}_n)- \sum_{j=1}^{n-1}\left(\mu(\tilde{E}_j)-\mu(\tilde{K}_j)\right)\\
&=&\mu(\pi_{i_n}^{-1}(K_n))- \sum_{j=1}^{n-1}\left(\mu(\pi_{i_j}^{-1}(E_j))-\mu(\pi_{i_j}^{-1}(K_j))\right)\\
&=&\mu_{i_n}(K_n)- \sum_{j=1}^{n-1}\left(\mu_{i_j}(E_j)-\mu_{i_j}(K_j)\right)\\
&>& \varepsilon-\sum_{j=1}^{n-1}\frac{\varepsilon}{4^j}>\varepsilon-\sum_{j=1}^{\infty}\frac{\varepsilon}{4^j}=\frac{2}{3}\varepsilon>\frac{1}{2}\varepsilon
\end{eqnarray*}
and so
\be\label{pre}
\mu_{i_n}(\pi_{i_n}(\tilde{K}'_n))=\mu(\pi_{i_n}^{-1}(\pi_{i_n}(\tilde{K}'_n)))=\mu(\tilde{K}'_n)>\frac{1}{2}\varepsilon.
\ee
As $I$ contains a countable cofinal subset, by Proposition~\ref{prop-eps-K}, there exists a compact subset $C$ of $X_I$ such that $\mu_{i_n}(\pi_{i_n}(C))>1-\tfrac{1}{2}\varepsilon$ for all $n\in\NN$. This together with~\eqref{pre} implies that $\pi_{i_n}(\tilde{K}'_n)\cap \pi_{i_n}(C)\neq\emptyset$. Hence, there exist $x\in \tilde{K}'_n$ and $y\in C$ such that $\pi_{i_n}(x)=\pi_{i_n}(y)$ and so $y\in \pi_{i_n}^{-1}(\pi_{i_n}(\tilde{K}'_n))= \tilde{K}'_n$, i.e., $y\in \tilde{K}'_n\cap C$. As for each $n\in\NN$ the set $\tilde{K}'_n\cap C$ is a non-empty compact subset of $X_I$ and $(\tilde{K}'_{n+1}\cap C)\subset(\tilde{K}'_n\cap C)$, we have that 
$$\emptyset\neq\bigcap_{n\in\NN}\left(\tilde{K}'_n\cap C\right)\subseteq\bigcap_{n\in\NN}\tilde{K}'_n \subseteq\bigcap_{n\in\NN}\tilde{E}_n,$$
which yields $\bigcap_{n\in\NN}\tilde{E}_n\neq\emptyset$. Therefore, $\mu$ is a finite pre-measure on $\C_I$ and so, by \cite[Theorem 5.6]{Bau}, uniquely extends to $\Sigma_I$. Moreover, $\Sigma_I=\B_I$ by Proposition \ref{prop::Tobias}.
\end{proof}

In the general case, when the index set $I$ does not contain a countable cofinal subset, we are going to combine the two-stage construction of $(X_I, \Sigma_I)$ introduced earlier (see Proposition \ref{lem::projective-limit}) with Lemma \ref{lem::content}.

\begin{rem}\label{doub-exact}
Consider
$$
\Lambda:=\{\lambda\subseteq I: \lambda\text{ contains a countable directed subset cofinal in }\lambda\}
$$
ordered by inclusion. For $\lambda\in\Lambda$, let $\{(X_\lambda,\Sigma_\lambda),\pi_i^\lambda,\lambda\}$ denote the projective limit of the projective system $\{(X_i,\Sigma_i),\pi_{i,j},\lambda\}$. 
Also, for any $\lambda\subseteq\kappa$ in $\Lambda$ consider the map $\pi^{\lambda,\kappa}$ defined in \eqref{eq::2}. Then $\{(X_\lambda,\B_\lambda),\pi^{\lambda,\kappa},\Lambda\}$ is a projective system of measurable spaces by Proposition \ref{prop::projective-system} (recall that $\Sigma_\lambda=\B_\lambda$ by Proposition~\ref{prop::Tobias}). Furthermore, given  an exact projective system $\{\mu_i: i\in I\}$ of Radon probability measures w.r.t.\!~$\{(X_i,\Sigma_i),\pi_{i,j},I\}$,
 we get:
\begin{enumerate}[(i)]
\item For any $\lambda\in\Lambda$, the exact projective system $\{\mu_i: i\in\lambda\}$ is thick and the corresponding cylindrical quasi-measure uniquely extends to a probability measure $\mu_\lambda$ on $(X_\lambda,\B_\lambda)$. 
\item The family $\{\mu_\lambda:\lambda\in\Lambda\}$ is an exact projective system of probability measures w.r.t.\! $\{(X_\lambda,\B_\lambda),\pi^{\lambda,\kappa},\Lambda\}$. 
\end{enumerate}
Indeed, (i) simply follows by Lemma \ref{lem::content} as $\lambda$ contains a countable cofinal subset. To prove (ii), let us observe that for all $\lambda\subseteq\kappa$ in $\Lambda$ we have 
$$
{\pi^{\lambda,\kappa}}_{\#}\mu_\kappa((\pi_i^\lambda)^{-1}(E_i))=\mu_\kappa((\pi^{\lambda,\kappa})^{-1}((\pi_i^\lambda)^{-1}(E_i)))=\mu_\kappa((\pi_i^\kappa)^{-1}(E_i))=\mu_i(E_i),
$$
for all $i\in\lambda$ and all $E_i\in\B_i$. This shows that ${\pi^{\lambda,\kappa}}_{\#}\mu_\kappa$ is an extension of the cylindrical quasi-measure corresponding to $\{\mu_i:i\in\lambda\}$ and so must coincide with $\mu_\lambda$ by the uniqueness of $\mu_\lambda$.
\end{rem}

\begin{thm}\label{ext-cylinderSigmaAlg}
Let $(I, \leq)$ be a directed partially ordered set, $\T:=\{(X_i, \tau_i), \pi_{i,j}, I\}$ a projective system of Hausdorff topological spaces, $\{(X_I, \tau_I), \pi_i, I\}$ its projective limit and $\{\mu_i: i\in I\}$ an exact projective system of Radon probability measures w.r.t.\!~$\T$.
For any $\lambda\in \Lambda$, as in Remark \ref{doub-exact}-(i), let $\mu_\lambda$  be the probability measure on $\B_\lambda$ extending the cylindrical quasi-measure associated with $\{\mu_i: i\in \lambda\}$.

Then there exists a unique probability measure $\nu$ on the cylinder $\sigma-$algebra $\Sigma_I$ on $X_I$ such that ${\pi_{i}}_{\#}\nu=\mu_i$ for all $i\in I$ if and only if the exact projective system $\{\mu_\lambda: \lambda\in \Lambda\}$ of Remark \ref{doub-exact}-(ii) is thick.
\end{thm}
\begin{proof}
Suppose that the exact projective system {$\{\mu_\lambda: \lambda\in \Lambda\}$} is also thick. Then the cylindrical quasi-measure $\mu_\Lambda$ corresponding to $\{\mu_\lambda:\lambda\in\Lambda\}$ is well-defined on $\C_\Lambda$ by Lemma~\ref{lem::corres-cylqm-thickexact}, and $\C_\Lambda=\Sigma_\Lambda=\Sigma_I$ by Proposition~\ref{prop::C-Lambda}. We claim that $\mu_\Lambda$ is actually itself a probability measure extending the cylindrical quasi-measure $\mu_I$ corresponding to $\{\mu_i:i\in I\}$ on $\C_I$. Indeed, for $i\in I$, $\lambda\in\Lambda$ such that $i\in\lambda$ and $E_i\in\B_i$, we have 
$$
\mu_\Lambda(\pi_i^{-1}(E_i))\!=%\mu_\Lambda((\pi^\lambda)^{-1}((\pi_i^\lambda)^{-1}(E_i)))\!=
\mu_\lambda((\pi_i^\lambda)^{-1}(E_i))\!=\mu_i(E_i)\!=\mu_I(\pi_i^{-1}(E_i)).
$$
Hence, $\mu_\Lambda\!\restriction_{\C_I}=\mu_I$. In particular, $X_\indexset=\pi_i^{-1}(X_i)\in\C_\indexset$ (for any $i\in\indexset$) and so, $\mu_\Lambda(X_I)=\mu_\indexset(\pi_i^{-1}(X_i))=\mu_i(X_i)=1$.

Let us show that $\mu_\Lambda$ is $\sigma-$additive. For any $(\lambda_n)_{n\in\NN}\subseteq\Lambda$, there is $\kappa\in\Lambda$ such that $\lambda_n\subseteq\kappa$ for all $n\in\NN$ by Proposition \ref{prop::Lambda-directed}. For any $n\in\NN$, let $E_n\in\B_{\lambda_n}$ be such that $\{(\pi^{\lambda_n})^{-1}(E_n):n\in\NN\}$ consists of pairwise disjoint cylinder sets in~$\C_\Lambda$. Then
$$
\mu_\Lambda\left(\bigcup_{n\in\NN}(\pi^{\lambda_n})^{-1}(E_n)\right)=\mu_\Lambda\left((\pi^\kappa)^{-1}\left(\bigcup_{n\in\NN}(\pi^{\lambda_n,\kappa})^{-1}(E_n)\right)\right)=\mu_\kappa\left(\bigcup_{n\in\NN}F_n\right),
$$
where for convenience we set $F_n:=(\pi^{\lambda_n,\kappa})^{-1}(E_n)$. For any $m\neq n$ in $\NN$, we have
$$
\emptyset = (\pi^{\lambda_m})^{-1}(E_m)\cap (\pi^{\lambda_n})^{-1}(E_n)=(\pi^\kappa)^{-1}(F_m\cap F_n)
$$
and hence, $\mu_\kappa(F_m\cap F_n)=\mu_\Lambda((\pi^\kappa)^{-1}(F_m\cap F_n))=\mu_\Lambda(\emptyset)=0$. For a fixed $n\in\NN$, define $G_n:=\bigcup_{m\neq n} (F_m\cap F_n)\in\B_\kappa$. Then $\mu_\kappa(G_n)=0$ as $G_n$ is a countable union of null sets and so
$\mu_\kappa\left(\bigcup_{n\in\NN}F_n\right)= \mu_\kappa\left(\bigcup_{n\in\NN}(F_n\backslash G_n)\right)$.
By construction, for any $m\neq n$ in $\NN$ we have $(F_m\backslash G_m)\cap (F_n\backslash G_n)=\emptyset$. 
Thus,
$$
\mu_\kappa\left(\bigcup_{n\in\NN}(F_n\backslash G_n)\right) = \sum_{n\in\NN}\mu_\kappa(F_n\backslash G_n)=\sum_{n\in\NN}\mu_\kappa(F_n)= 
%\sum_{n\in\NN}\mu_\kappa((\pi^{\lambda_n,\kappa})^{-1}(E_n))
\mu_\Lambda((\pi^{\lambda_n})^{-1}(E_n)),
$$
which combined with the previous steps yields the $\sigma-$additivity of $\mu_\Lambda$.
Hence, $\mu_\Lambda$ is a probability measure on $\Sigma_\Lambda=\Sigma_I$. The uniqueness of $\mu_\Lambda$ follows from \cite[Theorem~5.4]{Bau} as $\mu_\Lambda$ is finite and $\C_I$ is an $\cap-$stable family generating $\Sigma_I$.

Conversely, suppose there exists a unique probability measure $\nu$ on the cylinder $\sigma-$algebra $\Sigma_I=\C_\Lambda$ on $X_I$ such that ${\pi_{i}}_{\#}\nu=\mu_i$ for all $i\in I$. Let $\lambda\in\Lambda$. Then ${\pi_i^\lambda}_{\#}({\pi^\lambda}_{\#}\nu)={\pi_i}_{\#}\nu=\mu_i$ for all $i\in\lambda$, which shows that ${\pi^\lambda}_{\#}\nu$ is an extension of the cylindrical quasi-measure corresponding to $\{\mu_i:i\in\lambda\}$. Thus, ${\pi^\lambda}_{\#}\nu=\mu_\lambda$ by the uniqueness of $\mu_\lambda$, i.e., $\nu$ is the cylindrical quasi-measure on $\C_\Lambda$ corresponding to $\{\mu_\lambda:\lambda\in\Lambda\}$, and so $\{\mu_\lambda:\lambda\in\Lambda\}$ is thick by Lemma~\ref{lem::corres-cylqm-thickexact}.
\end{proof}

Using Remark~\ref{prop::cyl-quasi-measure}-(i) we can easily derive the following corollary. 

\begin{cor}\label{cor::yamasaki-surjective}
Let $(I, \leq)$ be a directed partially ordered set, $\T:=\{(X_i, \tau_i), \pi_{i,j}, I\}$ a projective system of Hausdorff topological spaces, $\{(X_I, \tau_I), \pi_i, I\}$ its projective limit and $\{\mu_i: i\in I\}$ an exact projective system of Radon probability measures w.r.t.\!~$\T$. Assume that for each $\lambda\in\Lambda$ the map $\pi^\lambda$ is surjective (see \eqref{eq::3} and Remark~\ref{doub-exact}).

Then there exists a unique probability measure $\nu$ on the cylinder $\sigma-$algebra $\Sigma_I$ on $X_I$ such that ${\pi_{i}}_{\#}\nu=\mu_i$ for all $i\in I$.
\end{cor}

\begin{rem}\label{ext-cylinderSigmaAlg-rem-support}\
\begin{enumerate}[(i)]
\item If in Theorem \ref{ext-cylinderSigmaAlg} each $\mu_i$ with $i\in\indexset$ is supported in a closed subset $K^{(i)}$ of $X_i$, then $\mu(\pi_i^{-1}(K^{(i)}))=\mu_i(K^{(i)})=1$ and hence, $\mu$ is supported in each $\pi_i^{-1}(K^{(i)})$. In particular, the measure $\mu$ is supported in $\bigcap_{i\in J}\pi_i^{-1}(K^{(i)})$ for each countable subset $J$ of $\indexset$.
\item If $K^{(i)}=\overline{\pi_i(K)}$ for some $K\subseteq X_\indexset$ such that 
\begin{equation}\label{eq::preimage-closed}
\overline{K}=(\pi^\lambda)^{-1}(E)\text{ for some }\lambda\in\superindexset\text{ and }E\subseteq X_\lambda\text{ closed},
\end{equation}
then there exists a countable $J\subseteq\indexset$ such that $\overline{K}=\bigcap_{i\in \subindexset}\pi_i^{-1}(\overline{\pi_i(K)})$ and so $\mu$ is supported in $\overline{K}$ by (i). Note that \eqref{eq::preimage-closed} holds in the following two cases:
\begin{enumerate}[(a)]
\item\label{item-open} $K=(\pi^\lambda)^{-1}(F)$ with $\lambda\in\Lambda$, $F\subseteq X_\lambda$ and $\pi_\lambda$ open,
\item\label{item-closed} $\overline{K}=(\pi^\lambda)^{-1}(F)$ with $\lambda\in\Lambda$, $F\subseteq X_\lambda$ and $\pi_\lambda$ closed,
\end{enumerate}
since taking $E:=\overline{F}$ in (a) and $E:=\overline{\pi^\lambda(K)}$ in (b) yields\ $\overline{K}=(\pi^\lambda)^{-1}(E)$.
\end{enumerate}
\end{rem}

\subsection{Extension to Radon measures}\label{sec:ext-Rad}\

In this subsection we are going to give a detailed proof of a result due to Prokhorov, Theorem~\ref{Prokh}, which provides a necessary and sufficient condition to extend cylindrical quasi-measures to Radon measures (see \cite{Prokh}, \cite[IX., \S 4.2, Theorem~1]{BouInt}, \cite[Part I, Chapter I, \S 10, Theorem 21]{S73}). Before doing so, let us introduce some notation and an extension theorem by Choksi \cite[Theorem 1]{Ch58}, which is fundamental in the proof of Prokhorov's theorem.

\begin{dfn}
Let $\mu$ be a measure on a $\sigma-$ring $\mathcal{R}$ on a set $X$ and $\mathcal{F}\subseteq\mathcal{R}$. 
\begin{enumerate}[(a)]
\item $\mu$ is inner regular w.r.t.\! $\mathcal{F}$ if $\mu(M)=\sup\{\mu(U):U\in \mathcal{F}, U\subseteq M\}, \forall M\in\mathcal{F}$.
\item $\mu$ is outer regular w.r.t.\! $\mathcal{F}$ if $\mu(M)=\inf\{\mu(U):U\in \mathcal{F}, M\subseteq U\}, \forall M\in\mathcal{F}$.
\end{enumerate}
\end{dfn}

We restate the Choksi result in a form which is less general than the original one in \cite[Theorem 1]{Ch58} but sufficient for our purposes.

\begin{thm}\label{choksi-thm}
Let $X$ be a Hausdorff topological space. Denote by $\mathcal{F}$ a collection of compact subsets of $X$ which contains the empty set and is closed under finite unions and countable intersections.
Let $\mu$ be a finite, non-negative and real-valued set function on $\mathcal{F}$ such that
\begin{enumerate}[(a)]
\item $\mu$ is monotone, 
\item $\mu$ is finitely additive and sub-additive,
\item\label{cond3} If $K_1 ,K_2\in\mathcal{F}$ with $K_1\subseteq K_2$, then for any $\varepsilon >0$ there exists $K_3\in\mathcal{F}$ such that $ K_3 \subset K_2 \setminus K_1$ and $\mu(K_1)+ \mu(K_3 ) \geq \mu(K_2) - \varepsilon$.
\end{enumerate}
Then $\mu$ can be uniquely extended to a finite measure on the $\sigma-$ring $S(\F)$ generated by $\mathcal{F}$, which is inner regular w.r.t.\! $\mathcal{F}$.
\end{thm}

The following is the main theorem about extending cylindrical quasi-measures to Radon measures on the Borel $\sigma-$algebra.

\begin{thm}\label{Prokh}
Let $(I, \leq)$ be a directed partially ordered set, $\T:=\{(X_i, \tau_i), \pi_{i,j}, I\}$ a projective system of Hausdorff topological spaces, $\{(X_I, \tau_I), \pi_i, I\}$ its projective limit and $\{\mu_i: i\in I\}$ an exact projective system of Radon probability measures w.r.t.\!~$\T$. Then there exists a unique Radon probability measure $\nu$ on $X_I$ such that ${\pi_{i}}_{\#}\nu=\mu_i$ for all $i\in I$ if and only if Prokhorov's condition~\eqref{epsilon-K} holds.
\end{thm}

\begin{rem}\label{Prokh-rem-support}
If in Theorem \ref{Prokh} each $\mu_i$ with $i\in\indexset$ is supported in a closed subset $K^{(i)}$ of $X_i$, then $\nu(\pi_i^{-1}(K^{(i)}))=\mu_i(K^{(i)})=1$ and hence, the Radon probability measure $\nu$ is supported in $\bigcap_{i\in \indexset}\pi_i^{-1}(K^{(i)})$ by \cite[Part I, Chapter I, 6.(a)]{S73}.
\end{rem}

Note that since Prokhorov's condition~\eqref{epsilon-K} holds when $I$ contains a countable cofinal subset (see Proposition \ref{prop-eps-K}), we immediately have the following corollary of Theorem \ref{Prokh}.

\begin{cor}\label{Prokh-ctbl}
Let $(I, \leq)$ be a directed partially ordered set, $\T:=\{(X_i, \tau_i), \pi_{i,j}, I\}$ a projective system of Hausdorff topological spaces, $\{(X_I, \tau_I), \pi_i, I\}$ its projective limit and $\{\mu_i: i\in I\}$ an exact projective system of Radon probability measures w.r.t.\!~$\T$. If $I$ contains a countable cofinal subset, then there exists a unique Radon probability measure $\nu$ on $X_I$ such that ${\pi_{i}}_{\#}\nu=\mu_i$ for all $i\in I$.
\end{cor}

We now present the proof of Theorem \ref{Prokh} by illustrating the most important steps and postponing some parts to lemmas, so to not interrupt the main proof scheme.

\begin{proof}[Proof of Theorem \ref{Prokh}]\ \\
We show only the sufficiency of \eqref{epsilon-K} for the existence and the uniqueness of a Radon probability measure extending $\mu$, as the converse direction is easy.

Suppose that \eqref{epsilon-K} holds. By Remark \ref{epsilon-K->thick}, $\{\mu_i: i\in I\}$ is thick. Denote by $\mu$ the cylindrical quasi-measure corresponding to $\{\mu_i: i\in I\}$ and by $\mu^\ast$ the outer measure associated to $\mu$. Let $\varepsilon>0$. Then there is $C_\varepsilon\subseteq X_I$ compact such that $\mu_i(\pi_i(C_{\varepsilon}))\geq 1-\varepsilon$ for all $i\in\indexset$. Now for any $U\in\C_I$ open and such that $C_\varepsilon\subseteq U$, we have that $U=\pi_i^{-1}(U_i)$ for some $i\in I$ and $U_i\in \B_i$,  so we have
$
\mu(U)= \mu_i(U_i)\geq \mu_i(\pi_i(C_\varepsilon))\geq 1-\varepsilon.
$ 
This yields
$
\mu^\ast(C_\varepsilon)\geq 1-\varepsilon
$
since $\mu^\ast$ is outer regular w.r.t.\! open sets in $\C_I$ (cf.\! Lemma~\ref{lem::outer-regularity}).

We show in Lemma \ref{lem::Choksi} that Theorem \ref{choksi-thm} can be applied to $\mu^\ast$ restricted to the family $\K$ of all compact subsets of $(X_I, \tau_I)$, which yields the existence of an inner regular measure $\eta$ which extends $\mu^\ast\!\restriction_\K$ to the $\sigma-$ring $S(\K)$ generated by $\K$.

Let us consider the function $$
\nu(B) := \sup\{\eta(B\cap K): K\in\K\},\quad\forall B\in\B_I.
$$
Lemma \ref{lem::measure-candidate} ensures that $\nu$ is a finite Radon measure on $\B_I$. We need to show that ${\pi_{i}}_{\#}\nu=\mu_i$ for all $i\in I$, i.e., $\mu(A)=\nu(A)$ for all $A\in\C_I$, which entails that $\nu$ extends $\mu$ to $\B_I$. 

Let $A\in\C_I$ and $\varepsilon>0$. Using Lemma \ref{lem::outer-regularity}, we get by \eqref{eq::inner-regularity-closed} that there exists $C\in\C_I$ closed such that $C\subseteq A$ and $\mu(A)\leq \mu(C)+\varepsilon$ and by \eqref{eq::outer-regularity} that there exists $U\in\C_I$ open such that $ C\cap C_\varepsilon\subseteq U$ and $ \mu(U)\leq \mu^\ast (C\cap C_\varepsilon)+\varepsilon$. Note that $C_\varepsilon= (C\cap C_\varepsilon)\cup ((X_I\backslash C)\cap C_\varepsilon)\subseteq U\cup (X_I\backslash C)$ and so 
$$1-\varepsilon\leq \mu^\ast(C_\varepsilon)\leq \mu(U\cup (X_I\backslash C))\leq\mu(U)+\mu(X_I\backslash C)\leq \mu(U)+1-\mu(C),$$ 
which implies
$\mu(C)\leq \mu(U)+\varepsilon$.
Hence, we have
\begin{eqnarray*}
\mu(A) & \leq & \mu(C)+\varepsilon \leq \mu(U)+2\varepsilon\leq \mu^\ast(C\cap C_\varepsilon)+ 3\varepsilon \\
&=& \eta(C\cap C_\varepsilon)+ 3\varepsilon\leq \nu(C) +3\varepsilon \leq \nu(A)+3\varepsilon,
\end{eqnarray*}
which yields $\mu(A)\leq \nu(A)$. Conversely, by the inner regularity of $\nu$ w.r.t.\! compact sets, for any $\varepsilon>0$ there exists $K\subseteq A$ compact such that 
$$
\nu(A)\leq \nu(K)+\varepsilon=\eta(K)+\varepsilon=\mu^\ast(K)+\varepsilon\leq \mu(A)+\varepsilon.
$$
Hence, $\nu$ is an extension of $\mu$ to $\B_I$ and in particular, $\nu(X_I)=\mu(X_I)=1$. Since $\C_I$ contains a basis of $\tau_I$, it is easy to show that such an extension is unique using that $\nu$ is inner regular w.r.t. compact sets.
%Suppose that $\nu^\prime$ is another Radon probability measure extending $\mu$ to $\B_I$. Hence, $\nu(A)=\nu^\prime(A)$ for all $A\in\C_I$. Since $\C_I$ contains a basis of $\tau_I$, by Lemma \ref{uniqueness-general} we have that $\nu$ and $\nu^\prime$ coincide on $\B_I$.
\end{proof}

\begin{lem}\label{lem::outer-regularity}
Let $\K$ be the family of all compact subsets of $X_I$. Then
\begin{equation}\label{eq::outer-regularity}
\mu^\ast(A)=\inf\{\mu(U):U\in \C_I \text{ open}, \ A\subseteq U\}, \quad\forall A\in \K \cup \C_I
\end{equation}
and
\begin{equation}\label{eq::inner-regularity-closed}
\mu^\ast(A)=\sup\{\mu(V):V\in \C_I \text{ closed}, \ V\subseteq A\}, \quad\forall A\in \C_I.
\end{equation}
\end{lem}
\begin{proof}
Let $A\in\C_I$, say $A=\pi_i^{-1}(A_i)$ for some $i\in I$ and $A_i\in\B_i$. As $\mu_i$ is finite and inner regular w.r.t.\! compact subsets of $X_i$, it is outer regular w.r.t.\! to open sets in $X_i$. Hence, for any $\varepsilon>0$, there is some $U_i\in\B_i$ open such that $A_i\subseteq U_i$ and $\mu_i(A_i)\geq \mu_i(U_i)-\varepsilon$. Setting $U:=\pi_i^{-1}(U_i)\in\C_I$, we get that $U$ is open, $A\subseteq U$, and
$
\mu^\ast(A) = \mu(A) \geq \mu(U) - \varepsilon.
$
Hence, $A$ satisfies \eqref{eq::outer-regularity} since $\varepsilon$ was arbitrary.

Now let $K\in\K$ and $\varepsilon>0$. Then, by the definition of $\mu^*$, there exists a sequence $(A_n)_{n\in\NN}$ such that $A_n\in\C_I$ for each $n\in\NN$, $K\subseteq \bigcup_{n\in\NN}A_n$ and $\mu^\ast(K)\geq \sum_{n\in\NN}\mu(A_n)-\varepsilon$. By the first part of this proof, we have that for any $n\in\NN$ there is $U_n\in\C_I$ open such that $A_n\subseteq U_n$ and $\mu(A_n)\geq \mu(U_n)-2^{-n}\varepsilon$. Then $K\subseteq \bigcup_{n\in\NN} A_n \subseteq \bigcup_{n\in\NN} U_n$ and by the compactness of $K$ there exists $N\in \NN$ such that $K\subseteq \bigcup_{n=1}^N U_n=:U\in\C_I$. This yields
$$
\mu^\ast(K) \geq \sum_{n\in\NN} \mu(A_n) - \varepsilon \geq \sum_{n\in\NN} \mu(U_n) - 2\varepsilon\geq \sum_{n=1}^N \mu(U_n) - 2\varepsilon \geq \mu(U) -2\varepsilon,
$$
i.e., $K$ satisfies \eqref{eq::outer-regularity} since $\varepsilon$ was arbitrary. For any $A\in\C_I$, by using \eqref{eq::outer-regularity} applied to $X_I\setminus A\in\C_I$ and the finiteness of $\mu$, we get \eqref{eq::inner-regularity-closed}.
\end{proof}

\begin{lem}\label{lem::Choksi}
Let $\K$ be the family of all compact subsets of $X_I$.
The set function $\mu^\ast\!\restriction_\K$ extends to a finite measure $\eta$ on the $\sigma-$ring $S(\K)$ generated by $\K$ and $\eta$ is inner regular w.r.t.\! $\K$.
\end{lem}
\begin{proof}
Recall that $\mu^\ast\!\restriction_\K$ is monotone and sub-additive.

Let $K_1,K_2\in\K$ be disjoint and $\varepsilon>0$. Then, by \eqref{eq::outer-regularity}, there exists $U\in\C_I$ open and containing $K_1\cup K_2$ such that $\mu(U)\leq\mu^\ast(K_1\cup K_2)+\varepsilon$. Furthermore, by using that the space $(X_I, \tau_I)$ is Hausdorff and $\C_I$ contains a basis of $\tau_I$, we can easily show that there exist $U_1,U_2\in\C_I$ disjoint and open such that $K_1\subseteq U_1\subseteq U$ and $K_2\subseteq U_2\subseteq U$. Therefore, we get 
$$
\mu^\ast(K_1)+\mu^\ast(K_2)\leq \mu(U_1)+\mu(U_2)=\mu(U_1\cup U_2)\leq \mu(U)\leq \mu^\ast(K_1\cup K_2)+\varepsilon
$$
and so the finite additivity of $\mu^\ast\!\restriction_\K$ as $\varepsilon$ was arbitrary. 

Let $K_1,K_2\in\K$ such that $K_1\subseteq K_2$ and let $\varepsilon>0$. By \eqref{eq::outer-regularity}, there exists $U_1\in\C_I$ open and containing $K_1$ such that $\mu(U_1)\leq\mu^\ast(K_1)+\varepsilon$. Then $K_3:=K_2\backslash U_1$ is compact as a closed subset of a compact and $K_3\subseteq K_2\backslash K_1$. Hence, by \eqref{eq::outer-regularity}, there exists $U_2\in\C_I$ open and containing $K_3$ such that $\mu(U_2)\leq\mu^\ast(K_3)+\varepsilon$. By construction, $K_2\subseteq U_1\cup (K_2\backslash U_1)\subseteq U_1\cup U_2$, so that
$$
\mu^\ast(K_2)\leq \mu(U_1)+\mu(U_2)\leq \mu^\ast(K_1)+\mu^\ast(K_3)+2\varepsilon.
$$
This shows that $\mu^\ast\!\restriction_\K$ satisfies the assumptions of Theorem \ref{choksi-thm}, which yields the assertion.
\end{proof}

\begin{lem}\label{lem::measure-candidate}
Let $\K$ be the family of all compact subsets of $X_I$.
The function 
$$
\nu(B) := \sup\{\eta(B\cap K): K\in\K\},\quad\forall B\in\B_I
$$
is a finite measure on $\B_I$, which extends $\eta$ and is inner regular w.r.t.\! $\K$.
\end{lem}
\begin{proof}
For any $B\in\B_I$ and any $K\in\K$, the set $K\cap B\in S(\K)$ since the Borel $\sigma-$algebra is the smallest $\sigma-$algebra generated by the collection $\F$ of all closed subsets in $(X_I, \tau_I)$ and $K\cap S(\F)=S(K\cap \F)$ (see, e.g.,~\cite[Theorem E, p.\!~25]{Hal}).

Let $(B_n)_{n\in\NN}$ be a sequence of pairwise disjoint subsets in $\B_I$. Then for all $K\in\K$ 
$$
\eta\left(\bigcup_{n\in\NN}B_n\cap K\right)=\sum_{n\in\NN}\eta(B_n\cap K)\leq \sum_{n\in\NN}\nu(B_n),
$$
which shows that $\nu$ is sub-additive. Let $\varepsilon>0$ and $n\in\NN$. Then there exists $K_n\in\K$ such that $\nu(B_n)\leq \eta(B_n\cap K_n)+2^{-n}\varepsilon$. For any fixed $N\in\NN$, set $K:=\bigcup_{n=1}^NK_n\in\K$. Then
\begin{eqnarray*}
\nu\left(\bigcup_{n\in\NN}B_n\right) &\geq & \eta\left(\left(\bigcup_{n\in\NN}B_n\right)\cap K\right)\geq\eta\left(\bigcup_{n=1}^N(B_n\cap K_n)\right)\\
&=&\sum_{n=1}^N\eta(B_n\cap K_n)\geq\sum_{n=1}^N(\nu(B_n)-2^{-n}\varepsilon)\\&\geq& \sum_{n=1}^N\nu(B_n)-\varepsilon,
\end{eqnarray*}
which implies $\nu(\bigcup_{n\in\NN}B_n)\geq \sum_{n\in\NN}\nu(B_n)$. Hence, $\nu$ is a measure on $\B_I$. 

Let us now show the inner regularity of $\nu$ w.r.t.\! $\K$. Let $B\in\B_I$ and $\varepsilon>0$. Then, by definition of $\nu$, there exists $K\in\K$ such that $\nu(B)\leq \eta(B\cap K)+\varepsilon$. Since $\eta$ is inner regular w.r.t.\! $\K$, there exists $K^\prime\subseteq B\cap K$ compact and such that $\eta(B\cap K)\leq \eta(K^\prime)+\varepsilon$ and so
$$
\nu(B)\leq \eta(B\cap K)+\varepsilon \leq \eta(K^\prime) +2\varepsilon\leq \nu(K^\prime)+2\varepsilon,
$$
which shows the inner regularity of $\nu$ w.r.t.\! $\K$.

Since by definition $\nu$ and $\eta$ coincide on $\K$ and since they are both inner regular w.r.t.\! $\K$, we obtain that for any $M\in S(\K)$ the following holds
$$\nu(M)=\sup\{\nu(K): K\subseteq M, K\in\K\}=\sup\{\eta(K): K\subseteq M, K\in\K\}=\eta(M), $$
i.e., $\nu$ extends $\eta$ to $\B_I$.
\end{proof}

\newpage
\section{Character space as projective limit}
\label{sec:Characters}
In the following all algebras in consideration are assumed to be unital commutative real algebras.
\subsection{The construction}\label{sec:constr}\ 

Let $A$ be a unital commutative real algebra and denote by $X(A)$ its character space, i.e., the set of all (unitary) homomorphisms $\alpha: A\to\RR$. In the following we always assume that $X(A)$ is non-empty. The aim of this section is to identify $X(A)$ as a projective limit of an appropriate projective system of measurable (resp.\! topological) spaces. To this purpose let us start with some fundamental definitions.

For any subset $G\subseteq A$, we denote by $\langle G\rangle$ the subalgebra of $A$ generated by $G$. For any subalgebra $S\subseteq A$ and any $a\in S$, we define $\hat{a}_S:X(S)\to\RR$ by $\hat{a}_S(\alpha):=\alpha(a)$, where $X(S)$ is the character space of $S$. When the subalgebra is clear from the context we just write $\hat{a}$ instead of $\hat{a}_S$. For any $S, T$ subalgebras of $A$ such that $S\subseteq T$, we set $\pi_{S,T}:X(T)\to X(S),\alpha\mapsto\alpha\!\restriction_S$ to be the canonical restriction map. Then, for any $S, T, R$ subalgebras of $A$ such that $S\subseteq T\subseteq R$, we clearly have
\be\label{comp-projections}
\pi_{S,R}=\pi_{S,T}\circ\pi_{T,R}
\ee
and also for all $a\in S$
\be\label{comp-Gelfand-transform}
\hat{a}_T=\hat{a}_S\circ\pi_{S,T}.
\ee

For any subalgebra $S\subseteq A$, we endow $X(S)$ with the weakest topology $\topchar{S}$ such that all the functions $\hat{a}_S$ with $a\in S$ are continuous, i.e., the family of all sets 
$$
\{\alpha\in X(S): \hat{a}(\alpha)>0\},\quad \forall \ a\in S
$$
forms a basis of $\topchar{S}$. Then for any $K\subseteq X(S)$ closed, there exists $G\subseteq S$ such that 
\be\label{eq::repr-closed-X(A)}
K=\{\alpha\in X(S): \hat{g}(\alpha)\geq 0\text{ for all }g\in G\}.
\ee
The topology $\topchar{S}$ coincides with the topology induced on $X(S)$ by the embedding 
$$
\begin{array}{llll}
\ & X(S)&\to& \RR^S\\
\ & \alpha &\mapsto& \left(\alpha(a)\right)_{a\in S}
\end{array}
$$
where $\RR^S$ is equipped with the product topology. Hence, $(X(S), \topchar{S})$ is Hausdorff. We consider on $X(S)$ the Borel $\sigma-$algebra $\Borchar{S}$ w.r.t.\!\! $\topchar{S}$. Then, for any $S,T$ subalgebras of $A$ such that $S\subseteq T$, the restriction map $\pi_{S,T}: X(T)\to X(S)$ is both continuous and measurable. Let us consider the following set
\begin{equation*}
\indexset:=\{S\subseteq A: S\text{ finitely generated subalgebra of }A\}.
\end{equation*}
ordered by the inclusion, which makes it into a directed partially ordered set. 

For each $S\in \indexset$, we set $\pi_S:=\pi_{S,A}$. The family $\{(X(S),\Borchar{S}),\pi_{S,T},\indexset\}$ (resp.\! $\{(X(S),\topchar{S}),\pi_{S,T},\indexset\}$) is a projective system of measurable (resp.\! topological) spaces. We denote by $\Sigma_\indexset$ the smallest $\sigma-$algebra on $X(A)$ such that all the maps $\pi_S$ with $S\in \indexset$ are measurable (resp.\! $\tau_\indexset$ the weakest topology on $X(A)$ such that all the maps $\pi_S$ with $S\in \indexset$ are continuous).

\begin{rem}\label{fin-dim-char}
Let $S\in \indexset$ and let $a_1,\ldots, a_d$ be generators of $S$ with $d\in\NN$. Then the homomorphism $\varphi: \RR[X_1,\ldots, X_d]\to S$ given by $\varphi(X_i):=a_i$ for $i=1,\ldots, d$ is surjective, and therefore, $S$ is isomorphic to $\RR[X_1,\ldots, X_d]/\ker(\varphi)$.
Since characters on $\RR[X_1,\ldots, X_d]$ are in one-to-one correspondence with points of $\RR^d$ (see, e.g., \cite[Proposition 5.4.5]{MarshBook}), the map $\psi:\Z(\ker(\varphi))\to X(S)$ given by $\psi(x)(\varphi(p)):=p(x)$ is a topological isomorphism, where $\Z(\ker(\varphi))=\{x\in\RR^d: p(x)=0\text{ for all }p\in\ker(\varphi)\}\subseteq\RR^d$ is equipped with the topology induced by the euclidean topology. Hence, the space $(X(S), \topchar{S})$ is locally compact and Polish.
\end{rem}

\begin{thm}\label{prop::X(A)-proj-lim}
The measurable space $(X(A),\Sigma_\indexset)$ together with the maps $\{\pi_S:S\in \indexset\}$ is the projective limit of the projective system $\{(X(S),\Borchar{S}),\pi_{S,T},\indexset\}$ (resp.\! the topological space $(X(A),\topchar{A})$ is the projective limit of the projective system $\{(X(S),\topchar{S}),\pi_{S,T},\indexset\}$). The $\sigma-$algebra $\Sigma_\indexset$ is the smallest $\sigma-$algebra on $X(A)$ making $\hat{a}$ measurable for all $a\in A$.
\end{thm}
\begin{proof}
By taking $R=A$ in \eqref{comp-projections}, we get $\pi_{S}=\pi_{S,T}\circ\pi_{T}$ for any $S, T\in \indexset$ such that $S\subseteq T$. Hence, it remains to show that $(X(A),\Sigma_\indexset)$ satisfies the universal property of the projective limit.

Let $(Y,\Sigma_Y)$ be a measurable space and for any $S\in \indexset$ let $f_S:Y\to X(S)$ be a measurable map which satisfies $f_S=\pi_{S,T}\circ f_T$ for all $T\in \indexset$ such that $S\subseteq T$. Define $f:Y\to X(A),y\mapsto f(y)$ by $f(y)(a):=f_{\langle a\rangle}(y)(a)$. 
Then, using the compatibility between the maps $\{f_S: S\in \indexset\}$, it is easy to show that $f$ is the unique measurable map satisfying $\pi_S\circ f=f_S$ for all $S\in \indexset$. 

Since for $S\in I$ the $\sigma-$algebra $\Borchar{S}$ is generated by all $\hat{a}_S$ with $a\in S$, the observation that for any $a\in A$ we have $\hat{a}_A=\hat{a}_S\circ\pi_S$ for all $S\in\indexset$ containing $a$ implies that $\Sigma_\indexset$ coincides with the $\sigma-$algebra generated by all $\hat{a}_A$ with $a\in A$.

This proof shows mutatis mutandis that $\{(X(A),\tau_\indexset), \pi_S, \indexset\}$ is the projective limit of the projective system $\{(X(S),\topchar{S}),\pi_{S,T},\indexset\}$ and that $\topchar{A}=\tau_\indexset$. 
\end{proof}

\subsection{The $\sigma-$algebras}\label{sec:sigma-alg}\

In this subsection we employ the two-stage construction of projective limits introduced in Section~\ref{sec:quasi-meas} (see Proposition~\ref{lem::projective-limit} and Proposition \ref{prop::C-Lambda}) to identify $(X(A),\Sigma_\indexset)$ as the projective limit of a different projective system over an index set $\superindexset$ larger than $\indexset$ (see Lemma~\ref{cor::double-proj-lim-to-char-space} below). This will allow us to relate the cylinder $\sigma-$algebra to the constructibly Borel $\sigma-$algebra introduced in \cite{GKM16} (see Proposition~\ref{prop::constr-Borel-is-cylinder} below) and to compare our results on the moment problem to the corresponding ones in \cite{GKM16} (see Section~\ref{sec:PLMP}).

\begin{lem}\label{cor::double-proj-lim-to-char-space}
Consider the set
\begin{equation}\label{superindexset}
\superindexset:=\{T\subseteq A: T\text{ countably generated subalgebra of }A\}
\end{equation}
ordered by the inclusion and let $\Sigma_\superindexset$ be the smallest $\sigma-$algebra on $X(A)$ such that all the $\pi_T$ with $T\in \superindexset$ are measurable. Then $\{(X(A),\Sigma_\indexset),\pi_T,\superindexset\}$ is the projective limit of the projective system $\{(X(T),\Borchar{T}),\pi_{T,R},\superindexset\}$, where $\indexset$ is as above. Moreover, 
\be\label{eq::sigma-lambda}
\C_\indexset\subseteq \Sigma_\indexset=\Sigma_\superindexset=\C_\superindexset,
\ee
where the inclusion is strict in general.
\end{lem}
 
\begin{proof}
For any $T\in\superindexset$, define $I_T:=\{S\subseteq T:S\text{ finitely generated subalgebra } \text{of }T\}$. Then the map $e(T):=I_T$ defines an embedding of $\superindexset$ into the set
\begin{equation}\label{eq::set-delta}
\{\lambda\subseteq \indexset: \lambda\text{ contains a countable directed subset cofinal in }\lambda\}.
\end{equation}
Indeed, let $T$ be a subalgebra of $A$ generated by $(a_n)_{n\in\NN}$ and for each $d\in\NN$ set $S_d:=\langle a_1,\ldots,a_d\rangle\in I_T$, then $(S_d)_{d\in\NN}$ is a countable directed subset of $I_T$ cofinal in $I_T$. Clearly, $e$ is order preserving and injective. Moreover, $e(\superindexset)$ is cofinal in the set defined in \eqref{eq::set-delta}, because if $\lambda\subseteq \indexset$ contains a countable directed subset $\kappa$ cofinal in $\lambda$, then $T:=\langle S: S\in\kappa\rangle$ is an element of $\superindexset$ and $e(T)= I_T\supseteq \lambda$. Indeed, by the cofinality of $\kappa$ in $\lambda$, for any $R\in\lambda$ there exists $S\in\kappa$ such that $R\subseteq S$ but $S\subset T$ and so $R\in I_T$.

Therefore, by Proposition \ref{lem::projective-limit} and Remark \ref{rem::Lambda}, $\{(X_\indexset,\Sigma_\indexset),\pi^{I_T}, e(\superindexset)\}$ is the projective limit of the projective system $\{(X_{I_T},\Sigma_{I_T}),\pi^{I_T,I_R}, e(\superindexset)\}$. Now Theorem \ref{prop::X(A)-proj-lim} guarantees that $X_\indexset=X(A)$ and also that for any $T\in\superindexset$ we have $X_{I_T}=X(T)$ and $\pi_S^{I_T}=\pi_{S,T}$ for all $S\in I_T$, which in turn implies $\Sigma_{I_T}=\Sigma_T$. This together with Proposition \ref{prop::Tobias} gives that $\Sigma_{I_T}=\Borchar{T}$. Therefore, for any $T\subseteq R$ in $\superindexset$, the uniqueness of the map $\pi^{I_T,I_R}$ (cf.\! \eqref{eq::2}) ensures that $ \pi^{I_T,I_R}=\pi_{T,R}$. Similarly, we obtain that $\pi^{I_T}=\pi_T$ for any $T\in\superindexset$ (cf.\! \eqref{eq::3}). Hence, we can replace $e(\superindexset)$ with $\superindexset$ and get that $\{(X(A),\Sigma_\indexset),\pi_T,\superindexset\}$ is the projective limit of $\{(X(T),\Borchar{T}),\pi_{T,R},\superindexset\}$. 

The relation \eqref{eq::sigma-lambda} directly follows by Proposition \ref{prop::C-Lambda} using that $\superindexset$ is embedded in the set in \eqref{eq::set-delta}.
\end{proof}

Let us recall that in \cite[Section 2]{{GKM16}} the $\sigma-$algebra $\B_c$ of constructibly Borel subsets of $X(A)$ is defined as the $\sigma-$algebra generated by all subsets $\{\alpha\in X(A): \alpha(a)>0\}$ with $a\in A$. In \cite[Proposition 2.3]{{GKM16}} the authors provide a characterization of $\B_c$, which in our setting translates in the following result.

\begin{prop}\label{prop::constr-Borel-is-cylinder}
The $\sigma-$algebra $\B_c$ of constructibly Borel subsets of $X(A)$ coincides with the cylinder $\sigma-$algebra $\Sigma_\superindexset$ for $\superindexset$ as in \eqref{superindexset}.
\end{prop}

We can then rewrite the definition of constructibly Radon measure \cite[Definition~2.4]{GKM16} as follows.

\begin{dfn}\label{def::constr-Radon}
A \emph{constructibly Radon} measure on $X(A)$ is a measure $\mu$ on the cylinder $\sigma-$algebra $\Sigma_\superindexset$ such that ${\pi_S}_{\#}\mu$ is a Radon measure on $\Borchar{S}$ for all $S\in \superindexset$.
\end{dfn}

\begin{rem}\label{rem::constr-Borel}\
\begin{enumerate}[(i)]
\item The restriction of any Radon measure on $X(A)$ to $\Sigma_\superindexset$ is a constructibly Radon measure, but  a constructibly Radon measure is \emph{not} a Radon measure in general; neither it is defined on the Borel $\sigma-$algebra $\Borchar{A}$ w.r.t.\! $\topchar{A}$ nor it is required to be inner regular. However, it is always a locally finite measure. 
\item If the algebra $A$ is countably generated, then the notions of constructibly Radon measure and Radon measure coincide. Indeed, by Proposition \ref{prop::Tobias} and \eqref{eq::sigma-lambda}, we get $\B_c=\C_\Lambda=\Sigma_\superindexset=\Sigma_\indexset=\Borchar{A}$ (cf.\! \cite[Proposition 2.2]{GKM16}) and so by Theorem~\ref{Prokh} any constructibly Radon measure is a Radon measure.
\item Any constructibly Radon measure on $X(A)$ is also a cylindrical quasi-measure w.r.t.\! the projective system $\{(X(S),\Borchar{S}),\pi_{S,T},\superindexset\}$.
\end{enumerate}
\end{rem}

We now show that  the set of finite constructibly Radon measures on $X(A)$ actually coincides with the set of finite measure on the cylinder $\sigma-$algebra $\Sigma_\indexset$.

\begin{lem}\label{lem::constrRadonIJ}
Any finite measure $\mu$ on the cylinder $\sigma-$algebra $\Sigma_\indexset$ is a finite constructibly Radon measure on $X(A)$.\end{lem}
\begin{proof}
Let $\mu$ be a finite measure on the cylinder $\sigma-$algebra $\Sigma_\indexset$. Then $\mu$ is a finite measure on $\Sigma_\superindexset$ by \eqref{eq::sigma-lambda}. W.l.o.g.\! let us assume that $\mu(X(A))=1$. Then for any $T\in \superindexset$ we have that ${\pi_T}_{\#}\mu$ is a probability measure on $\Borchar{T}$ on $X(T)$, since $\pi_T: (X(A), \Sigma_\superindexset)\to(X(T), \Borchar{T})$ is measurable by Theorem \ref{prop::X(A)-proj-lim}. Hence, for any $S\in\indexset$, we have that ${\pi_S}_{\#}\mu$ is a finite Borel measure on the Polish space $X(S)$ (see Remark \ref{fin-dim-char}) and so a Radon measure. It remains to show that ${\pi_R}_{\#}\mu$ is Radon for any $R\in \superindexset\backslash \indexset$. To this purpose, for any $R\in \superindexset\backslash \indexset$, set $\indexset_R:=\{S\in \indexset:S\subseteq R\}$. Then $\{(X(R),\Borchar{R}), \pi_{S, R}, \indexset_R\}$ is the projective limit of the projective system $\P:=\{(X(S),\Borchar{S}),\pi_{S,T},\indexset_R\}$ and the family $\{{\pi_S}_{\#}\mu: S\in \indexset_R\}$ is an exact projective system of Radon probability measures on~$\P$ since $\indexset_R\subset \indexset$ and $\mu(X(A))=1$. Moreover, for any $S\in \indexset_R$ we have ${\pi_{S,R}}_{\#}({\pi_R}_{\#}\mu)={\pi_S}_{\#}\mu$. By Proposition~\ref{prop-eps-K} the assumptions of Theorem \ref{Prokh} are satisfied since $\indexset_R$ contains a countable cofinal subset. Hence, ${\pi_R}_{\#}\mu$ is in fact a Radon measure, which completes the proof.
\end{proof}

\newpage
\section{The projective limit approach to the moment problem}\label{sec:PLMP}

Let $A$ be a unital commutative real algebra and $X(A)$ its character space endowed with the topology $\topchar{A}$. In this section, we deal with Problem \ref{GenKMP} for $\Sigma$ being the Borel $\sigma-$algebra $\Borchar{A}$ on $(X(A),\topchar{A})$ and more precisely we are interested in finding necessary and sufficient conditions for the existence of a \emph{$K-$representing Radon measure on the Borel $\sigma-$algebra $\Borchar{A}$ on $X(A)$}. Our general strategy is to combine the projective limit structure of $X(A)$ studied in Section \ref{sec:Characters} with the results in Section~\ref{sec:ext-cyl}  to establish a general criterion for the existence of a $K-$representing measure on the cylinder $\sigma-$algebra $\Sigma_\indexset$ on $X(A)$ and then the results in Section \ref{sec:ext-Rad} to extend such a representing measure on $\Sigma_\indexset$ to a $K-$representing Radon measure on $\Borchar{A}$.
As in Section \ref{sec:Characters}, we set 
\be\label{main-indexset}
\indexset:=\{S\subseteq A: S\text{ finitely generated subalgebra of }A\},
\ee
$\pi_{S,T}:X(T)\to X(S)$ to be the canonical restriction and $\pi_S:=\pi_{S,A}$ for any $S\subseteq T$ in $\indexset$. Recall that all algebras in consideration are unital commutative and real.

We immediately notice that when $A$ is countably generated, say by $(a_n)_{n\in\NN}$, the index set $I$ in \eqref{main-indexset} contains the countable cofinal subset $J:=\{\langle a_1, \ldots, a_d \rangle : d\in\NN\}$. Therefore, our strategy for solving Problem \ref{GenKMP} simplifies when $A$ is countably generated and we will analyze this case separately in Section \ref{main-results-ctbl}, while in Section~\ref{main-results-gen} we tackle the general case. We start Section~\ref{main-results} with some general results holding in both cases. In Sections \ref{RH-Thm}, \ref{Nuss-Thm}, \ref{appl-Arch} and \ref{appl-loc} we will focus on applications of the main results in Section \ref{main-results}. In Section \ref{RH-Thm}, respectively in Section \ref{Nuss-Thm}, we will derive a generalized version of the classical Riesz-Haviland theorem, respectively Nussbaum's theorem. In Section \ref{appl-Arch}, we will both retrieve the solution to Problem \ref{GenKMP} for Archimedean quadratic modules and prove a new representation theorem for linear functionals non-negative on a ``partially Archimedean'' quadratic module of $A$. In Section \ref{appl-loc}, we will retrieve some recent results about applications of localization to the moment problem.\vspace{-0.2cm}

\subsection{Main results on Problem \ref{GenKMP}}\label{main-results}\ 

Let $L$ be a linear functional on an algebra $A$. Assume that each restriction $L\!\!\restriction_S$ to each finitely generated subalgebra $S$ of $A$ admits a representing measure. We investigate the question how to build out of them an exact projective system. Only in the case when $L\!\!\restriction_S$ admits a unique representing measure this is straightforward, see Lemma~\ref{lem:uniq-S}. Whereas in Lemma \ref{lem::ex-ex-proj-sys} we show how to build such an exact projective system even when uniqueness is lacking.

 \begin{lem}\label{lem:uniq-S} Let $A$ be an algebra, $L:A \rightarrow \mathbb{R}$ linear, $L(1)=1$ and 
$\indexset$ as in~\eqref{main-indexset}. For each $S\in I$, let $K^{(S)}$ be a closed subset of $X(S)$ such that $\pi_{S,T}\left(K^{(T)}\right)\subseteq K^{(S)}$ for any $T\in \indexset$ containing~$S$. 

If for each $S\in \indexset$ there exists a unique $K^{(S)}-$representing Radon measure $\nu_S$ for~$L\!\!\restriction_S$, then $\{\nu_S:S\in \indexset\}$ is an exact projective system.
 \end{lem}
 \proof
For each $S\in \indexset$, we have $1\in S$ and so $L\restriction_S(1)=L(1)=1$, which ensures $\nu_S(K^{(S)})=1$. 
Moreover, for all $S,T\in  \indexset$ with $S\subseteq T$ and all $a\in S$ we have that
\begin{eqnarray}\label{eq:repr}
L\!\restriction_S\!\!(a)&=&L\!\restriction_T\!(a)=\int_{K^{(T)}}\hat{a}_T(\beta)\dd\nu_T(\beta)\nonumber\\&\stackrel{\eqref{comp-Gelfand-transform}}{=}&\int_{K^{(T)}}\hat{a}_S(\pi_{S, T}(\beta))\dd\nu_T(\beta)=\int_{\pi_{S,T}(K^{(T)})}\hat{a}_S(\alpha)\dd{\pi_{S,T}}_{\#}\nu_T(\alpha)
\end{eqnarray}
and
\be\label{eq:repr2}
{\pi_{S,T}}_{\#}\nu_T(K^{(S)})= \nu_T(\pi_{S,T}^{-1}(K^{(S)}))\geq \nu_T(K^{(T)})=1.
\ee Hence,
 ${\pi_{S,T}}_{\#}\nu_T$ is a $K^{(S)}-$representing Radon probability for~$L\!\!\restriction_S$. Then the uniqueness assumption ensures that
$\nu_S\equiv{\pi_{S,T}}_{\#}\nu_T$ for all $ S\subseteq T$ and so $\{\nu_S, S\in \indexset\}$ is an exact projective system of Radon probabilities on $\{(X_S, \B_{X_S}), \pi_{S,T}, \indexset\}$.
 \endproof
 
In general, the representing measure $\nu_S$ need not to be unique and so the system $\{\mu_S: S \in I \}$ is not necessarily exact. However, we will show in Lemma \ref{lem::ex-ex-proj-sys} that it is always possible to find an exact projective system of representing measures. To this end let us introduce some notations and a general topological result, namely Proposition~\ref{lem::compactness}, which can be seen as a refinement of \cite[Theorem 1.19]{Schm17}.\\

For any $S\in \indexset$, where $\indexset$ is as in~\eqref{main-indexset}, we set $X_S:=X(S)$. We endow the set $\PP(X_S)$ of all Radon probability measures on $X_S$ with the vague topology $\omega_S$, i.e., the coarsest topology on $\PP(X_S)$ such that for all $f\in\C_\mathrm{b}(X_S)$ the map $\PP(X_S)\to\RR, \mu\mapsto \int f\dd\mu$ is continuous, where $\C_\mathrm{b}(X_S)$ is the set of all continuous and bounded functions on $X_S$. Since $(X_S,\tau_{X_S})$ is Polish by Remark \ref{fin-dim-char}, the topology $\omega_S$ is metrizable (see, e.g., \cite[Theorem 6.8]{B99}). For any $K^{(S)}\subseteq X_S$ closed and any $L_S:S\to\RR$ linear with $L_S(1)=1$, we denote by $\R(L_S,K^{(S)})$ the set of all $K^{(S)}-$representing Radon measures for $L_S$ on $\B_{X_S}$. The condition $L_S(1)=1$ ensures that all representing measures for $L_S$ are in fact probability measures, and so $\R(L_S,K^{(S)})\subseteq \PP(X_S)$. Moreover, for any $\varepsilon>0$ and any $C_\varepsilon\subseteq X(S)$ closed, we set $\R(L_S,K^{(S)},C_\varepsilon):=\{\mu\in \R(L_S,K^{(S)}): \mu(C_\varepsilon)\geq 1-\varepsilon\}$. Clearly $\R(L_S,K^{(S)},K^{(S)})=\R(L_S,K^{(S)})$.

\begin{prop}\label{lem::compactness}
For any $S\in I$, $\varepsilon>0$, and $C_\varepsilon\subseteq X(S)$ closed, the set \\$\R(L_S,K^{(S)},C_\varepsilon)$ is compact in $(\PP(X_S), \omega_S)$. In particular, $\R(L_S,K^{(S)})$ is compact in $(\PP(X_S), \omega_S)$.
\end{prop}

\begin{proof}
Let us first observe that $\R_S:=\R(L_S,K^{(S)},C_\varepsilon)$ is tight, i.e., $$
\forall\delta>0 ,\;\exists\; K\text{ compact  in }(X_S,\tau_{X_S})\text{ s.t. }\forall \mu\in\R_S,\ \mu(K)\geq 1-\delta.
$$
Let $a_1,\ldots,a_d$ be generators of $S$ with $d\in\NN$ and set $p:=\sum_{i=1}^da_i^2$. Then, for any $m\in\NN$, the set $K_m:=\{\alpha\in X_S:\hat{p}(\alpha)\leq m\}$ is compact w.r.t.\! $\tau_{X_S}$ and for all $\mu\in\R_S$ we have
$$
\mu\left(X_S\backslash K_m\right)\leq \tfrac{1}{m}\int_{X_S} \hat{p}(\alpha) \dd\mu(\alpha) = \tfrac{1}{m}L_S(p).
$$
Since $L_S(p)<\infty$, this shows that $\R_S$ is tight. Hence, $\overline{\R_S}$ is compact {in $(\PP(X_S), \omega_S)$} by a well-know theorem due to Prokhorov (see, e.g., \cite[Theorems~5.1,~5.2]{B99}).
The conclusion will follow from the fact that $\R_S$ is closed {in $(\PP(X_S), \omega_S)$.}
To show this, {we consider $\nu\in\overline{\R_S}$} and prove that $\nu$ is a $K^{(S)}-$representing probability measure for $L_S$ such that $\nu(C_\varepsilon)\geq 1-\varepsilon$. {Fix $a\in S$ and ${\delta}>0$. Then there exists $m\in\NN$ such that $L_S(p)L_S(a^4)\leq m\delta^2$. As $X_S$ is a closed subset of $\RR^d$ (cf.\! Remark \ref{fin-dim-char}), $K_m$ is contained in the interior of $K_{m+1}$ in the metric space $(X_S,\tau_S)$. Urysohn's lemma guarantees that there exists $h_m\in\C_\mathrm{b}(X_S)$ such that $\mathbbm{1}_{K_m}\leq h_m\leq\mathbbm{1}_{K_{m+1}}$. Since $\nu\in\overline{\R_S}$, there exists $(\mu_n)_{n\in\NN}\subseteq\R_S$ converging to $\nu$ w.r.t. $\omega_S$ and so, there exists $n\in\NN$ such that $\left|\int h_m\hat{a}^2 \dd\nu - \int h_m\hat{a}^2 \dd\mu_n\right|<\delta$ as $h_m\hat{a}^2\in\C_\mathrm{b}(X_S)$. Moreover, $\left|\int (1-h_m)\hat{a}^2\dd\mu_n\right|^2\leq \mu_n(X_S\backslash K_m)L_S(a^4)\leq {\delta}^2$.} Hence, we obtain
$$
\left|\int h_m\hat{a}^2\dd\nu- L_S(a^2)\right|\leq \left|\int h_m\hat{a}^2 \dd\nu - \int h_m\hat{a}^2 \dd\mu_n\right|+\left|\int (1-h_m)\hat{a}^2\dd\mu_n\right|\leq 2\delta.
$$
As $(h_m\hat{a}^2)_{n\in\NN}$ is monotone increasing and converging to $\hat{a}^2$, the previous inequality and the monotone convergence theorem imply that
$$
\int \hat{a}^2\dd\nu =\lim_{n\to\infty} \int h_n\hat{a}^2 \dd\nu = L_S(a^2).
$$
Then $L_S(a)=\int \hat{a}\dd\nu$ follows immediately from the identity $4a=(a+1)^2-(a-1)^2$. Hence, $\nu$ is an $X_S-$representing measure for $L_S$.

As a direct consequence of the Portmanteau theorem \cite[Theorem 2.1]{B99}, which  can be applied as $K^{(S)}$ is closed, we obtain
$\nu(K^{(S)})\geq \limsup_{n\to\infty}\mu_n(K^{(S)})=1$. Similarly, $\nu(C_\varepsilon)\geq \limsup_{n\to\infty}\mu_n(C_\varepsilon)=1-\varepsilon$ which proves the assertion.
\end{proof}

The following lemma will be crucial in the proof of our main results.

\begin{lem}\label{lem::ex-ex-proj-sys} Let $A$ be an algebra, $L:A \rightarrow \mathbb{R}$ linear, $L(1)=1$ and 
$\indexset$ as in~\eqref{main-indexset}. For each $S\in I$, let $K^{(S)}$ be a closed subset of $X(S)$ such that $\pi_{S,T}\left(K^{(T)}\right)\subseteq K^{(S)}$ for any $T\in \indexset$ containing~$S$. 

\begin{enumerate}[(i)]
\item If for each $S\in \indexset$ there exists $\nu_S\in\R(L\!\!\restriction_S,K^{(S)})$, then there exists an \emph{exact} projective system $\{\mu_S:S\in \indexset\}$ such that each $\mu_S\in \R(L\!\!\restriction_S,K^{(S)})$.
\item If in addition, the family $\{\nu_S :S\in I\} $ fulfills Prokhorov's condition \eqref{epsilon-K-char}, i.e.,
\be\label{epsilon-K-char}
\forall \varepsilon>0\ \exists\ C_\varepsilon\subset X(A) \text{ compact s.t. } \forall S\in \indexset: \ \nu_S(\pi_S(C_\varepsilon))\geq 1-\varepsilon,
\ee
then also the exact projective system $\{\mu_S:S\in \indexset\}$ fulfills Prokhorov's condition~\eqref{epsilon-K-char}. 
\end{enumerate}
\end{lem}

\begin{proof}\
\begin{enumerate}[(i)]
\item For any $S\in\indexset$, we set $\R_S:=\R(L\!\!\restriction_S,K^{(S)})$, which is non-empty by as\-sumption and compact by Proposition \ref{lem::compactness}. For any $S\subseteq T$ in $\indexset$, the map ${\pi_{S,T}}_{\#}: \R_T\to\R_S,\nu_T\mapsto \nu_T\circ \pi_{S,T}^{-1}$ is well-defined. Indeed, for any $\nu_T\in\R_T$ and any $a\in S$, 
\eqref{eq:repr} and \eqref{eq:repr2} both hold, ensuring that ${\pi_{S,T}}_{\#}\nu_T\in\R_S$. Furthermore, it is easy to show that ${\pi_{S,T}}_{\#}$ is continuous using the definition of the vague topology.
Altogether, this shows that $\{(\R_S,\omega_S\cap \R_S),{\pi_{S,T}}_{\#},\indexset\}$ is a projective system of non-empty compact topological spaces and so, by \cite[I., \S 9.6, Proposition 8]{BouGT}, its projective limit $\R$ is non-empty and compact. Considering the standard model of the projective limit in the product space, we identify $\R$ as a subset of $\prod_{S\in \indexset}\R_S$.  Let $(\mu_S)_{S\in \indexset}\in\R$. Then $\{\mu_S:S\in \indexset\}$ is an exact projective system of Radon probability measures on $\{(X_S, \B_{X_S}), \pi_{S,T}, I\}$, as ${\pi_{S,T}}_{\#}\mu_T=\mu_S$ for all $S\subseteq T$ in $\indexset$ and $\mu_S\in\R_S$ for all $S\in \indexset$.
\item By replacing $\R_S$ with $\widetilde{\R_S}:=\bigcap_{\varepsilon>0}\R(L\!\!\restriction_S,K^{(S)},\pi_S(C_\varepsilon))$ in the proof of (i) above, %and observing that for any $\mu_T\in \widetilde{\R_T} $ we have ${\pi_{S,T}}_{\#}\mu_T(\pi_S(C_\varepsilon))= \mu_T(\pi_{S,T}^{-1}(\pi_{S,T}(\pi_T(C_\varepsilon))))\geq \mu_T(\pi_T(C_\varepsilon))\geq 1-\varepsilon$ for all $\varepsilon>0$ , 
we easily get that there exists an exact projective system $\{\mu_S:S\in \indexset\}$ such that $\mu_S\in\widetilde{\R_S}$ for all $S\in \indexset$. In particular, $\mu_S(\pi_S(C_\varepsilon))\geq 1-\varepsilon$ for all $S\in\indexset$ and all $\varepsilon>0$, i.e., $\{\mu_S:S\in \indexset\}$ fulfils Prokhorov's condition \eqref{epsilon-K-char}.\end{enumerate}\vspace{-0.3cm}
\end{proof}

Recall that a quadratic module $Q$ in a unital commutative real algebra $A$ is a subset $Q$ of $A$ such that $1 \in Q, \ Q+Q \subseteq Q$ and $a^{2}Q\subseteq Q$ for each $a \in A$. We say that $Q$ is generated by a subset $G$ of $A$ if $Q$ is the smallest quadratic module in $A$ containing $G$. For a quadratic module $Q$ in $A$, we have
\be\label{semi-alg-Q}
K_Q=\{\alpha\in X(A): \hat{a}(\alpha)\geq 0, \forall a\in Q\}.
\ee
If $Q$ is generated by $G$, then $K_Q=\{\alpha\in X(A): \hat{g}(\alpha)\geq 0 , \forall g\in G\}$.

\enlargethispage{0.5cm}
\begin{rem}\label{rem-count}\
\begin{enumerate}[(i)]
\item Let $K\subseteq X(A)$ be closed. Then there exists $G$ subset of $A$ such that $K=\{\alpha\in X(A): \hat{g}(\alpha)\geq 0\text{ for all }g\in G\}$ (cf.\! \eqref{eq::repr-closed-X(A)}) and thus, $K=K_Q$ by \eqref{semi-alg-Q}, where $Q$ is the quadratic module in $A$ generated by $G$. In particular,
$$
K=K_Q = \bigcap_{S\in\indexset}\pi_S^{-1}(K_{Q\cap S}),
$$
where 
\be\label{bcsas-S}
K_{Q\cap S}:=\{\alpha\in X(S):\hat{a}(\alpha)\geq 0 \text{ for all } a\in Q\cap S\}.
\ee
Tough $S$ is finitely generated, $K_{Q\cap S}$ may be not finitely generated.
\item If a quadratic module $Q$ in $A$ is generated by a countable set $G$, then
$$
K_Q = \bigcap_{S\in J}\pi_S^{-1}(K_{Q\cap S}),
$$
where $J=\{\langle g\rangle:g\in G\}$ is a countable subset of $\indexset$.
\item A quadratic module $Q$ in $A$ is generated by a countable set $G$ if and only if $K_Q=\pi_R^{-1}(E)\in\C_\Lambda=\Sigma_\indexset$ for some $R\in\Lambda$ and $E$ closed in $X(R)$ (e.g., take
$R:=\langle G\rangle\in\Lambda$ and $E:=\{\alpha\in X(R): \hat{g}(\alpha)\geq 0, \forall \ g\in G\}$).
\item When $A$ is countably generated, every closed $K\subseteq X(A)$ can be written as $K=K_Q$ with $Q$ countably generated, since $K\in\B_{\tau_{X(A)}}=\Sigma_\Lambda$ by \eqref{eq::sigma-lambda} and so we can apply Remark 3.5-(iii).
\end{enumerate}
\end{rem}

\subsubsection{\bf $A$ countably generated}\label{main-results-ctbl}\ 

Let us first consider Problem \ref{GenKMP} for $K=X(A)$. 
\begin{thm}\label{cylinder-MP-ctbl}
Let $A$ be a countably generated algebra, $L:A \to \mathbb{R}$ linear, $L(1)=1$ and 
$\indexset$ as in~\eqref{main-indexset}. Then there exists an $X(A)-$representing Radon measure $\nu$ for $L$ if and only if for every $S\in \indexset$ there exists an $X(S)-$representing Radon measure for $L\!\restriction_S$.
\end{thm}

The following more general result Theorem~\ref{cylinder-MP-Support-ctbl} also deals with the support of the representing measure. Theorem \ref{cylinder-MP-ctbl} can be derived from Theorem~\ref{cylinder-MP-Support-ctbl} by taking $K^{(S)}=X(S)$ for each $S\in \indexset$. 

\begin{thm}\label{cylinder-MP-Support-ctbl}
Let $A$ be a countably generated algebra, $L:A \to \mathbb{R}$ linear, $L(1)=1$ and 
$\indexset$ as in~\eqref{main-indexset}. For any $S\in \indexset$, let $K^{(S)}$ be a closed subset of $X(S)$ such that $\pi_{S,T}\left(K^{(T)}\right)\subseteq K^{(S)}$ for any $T\in \indexset$ containing~$S$.\\ Then there exists a $\left(\bigcap_{S\in \indexset}\pi_S^{-1}\!\left(K^{(S)}\right)\!\right)$--repre\-senting Radon measure $\nu$ for $L$ if and only if, for every $S\in \indexset$ there exists a $K^{(S)}-$representing Radon measure for $L\!\restriction_S$.
\end{thm}

\begin{proof}\ \\
As observed in Theorem \ref{prop::X(A)-proj-lim}, $\{(X(A),\Sigma_\indexset), \pi_S, \indexset\}$ is the projective limit of the projective system $\P:=\{(X(S),\Borchar{S}),\pi_{S,T},\indexset\}$, where $\Borchar{S}$ denotes the Borel $\sigma-$algebra on $X(S)$ w.r.t.\! $\topchar{S}$. Suppose that for any $S\in I$ there exists a $K^{(S)}-$representing Radon measure for $L\!\restriction_S$. By Lemma \ref{lem::ex-ex-proj-sys}-(i), there exists an exact projective system $\{\mu_S:S\in \indexset\}$ of Radon probability measures on $\P$ such that $\mu_S$ is a $K^{(S)}-$representing Radon measure for $L\!\restriction_S$ for all $S\in \indexset$. Since $A$ is countably generated, we have that $I$ contains a countable cofinal subset and so Corollary \ref{Prokh-ctbl} and Remark~\ref{Prokh-rem-support} ensure that there exists a Radon measure $\nu$ such that 
\be\label{compatib}
{\pi_{S}}_{\#}\nu=\mu_S,\quad\forall S\in \indexset.
\ee 
and $\nu$ is supported in $(\bigcap_{S\in \indexset}\pi_S^{-1}(K^{(S)}))$. 
%Recall that this follows from the fact that, by Proposition \ref{prop-eps-K}, $\{\mu_S:S\in \indexset\}$ fulfills the \eqref{epsilon-K}--condition, i.e.,
%\be\label{epsilon-K-char}
%\forall \varepsilon>0\ \exists\ K\subset X(A) \text{ compact s.t. } \forall S\in \indexset: \ \mu_S(\pi_S(K))\geq 1-\varepsilon.
%\ee 
 For any $a\in A$, there exists  $S\in \indexset$ such that $a\in S$ and so
$$
L(a)=L\!\restriction_S(a)=\int_{X(S)}\!\!\hat{a}(\beta)\dd\mu_S(\beta)\stackrel{\eqref{compatib}}{=}\int_{X(A)}\!\!\hat{a}(\pi_{S}(\beta))\dd\nu(\beta)\stackrel{\eqref{comp-Gelfand-transform}}{=}\int_{X(A)}\!\!\hat{a}(\alpha) \dd\nu(\alpha). 
$$
Hence, $\nu$ is a $(\bigcap_{S\in \indexset}\pi_S^{-1}(K^{(S)}))-$representing Radon measure for $L$. 

Conversely, if there exists a $(\bigcap_{S\in \indexset}\pi_S^{-1}(K^{(S)}))-$representing Radon measure $\nu$ for $L$ then for every $S\in I$ the measure ${\pi_S}_{\#}\nu$ is a $K^{(S)}-$representing Radon measure for $L\restriction_S$, since $1=\nu(\bigcap_{T\in \indexset}\pi_T^{-1}(K^{(T)}))\leq \mu(\pi_S^{-1}(K^{(S)}))={\pi_S}_{\#}\nu( K^{(S)})\leq 1$.
\end{proof}

\begin{cor}\label{main-KMP-ctbl}
Let $A$ be a countably generated algebra, $L:A \to \mathbb{R}$ linear, $L(1)=1$, 
$\indexset$ as in~\eqref{main-indexset}, and $K\subseteq X(A)$ closed, i.e., $K=K_Q$ for some quadratic module $Q$ in $A$ (cf.\! Re\-mark~\ref{rem-count}-(i)). Then there exists a $K-$repre\-senting Radon measure $\nu$ for $L$ if and only if for every $S\in \indexset$ there exists a $K_{Q\cap S}-$representing Radon measure for $L\!\restriction_S$, where $K_{Q\cap S}$ is as in \eqref{bcsas-S} \end{cor}
\begin{proof}
Let us observe that $\pi_{S,T}(K_{Q\cap T})\subseteq K_{Q\cap S}$ for any $S\subseteq T$ in $\indexset$ and, by Remark~\ref{rem-count}-(i), we have that $K=K_Q = \bigcap_{S\in\indexset}\pi_S^{-1}(K_{Q\cap S})$. Thus, the main conclusion follows by applying Theorem \ref{cylinder-MP-Support-ctbl} with $K^{(S)}=K_{Q\cap S}$ for $S\in\indexset$.
\end{proof}

\subsubsection{\bf $A$ not necessarily countably generated}\label{main-results-gen}\ \\
Since in this case the index set $I$ in~\eqref{main-indexset} may not contain a countable cofinal subset, to get Theorem \ref{cylinder-MP}, resp. Theorem \ref{cylinder-MP-Support}, below (the analogue of Theorem \ref{cylinder-MP-ctbl}, resp. Theorem \ref{cylinder-MP-Support-ctbl}, in this general case) we will apply Theorem \ref{ext-cylinderSigmaAlg} and Theorem \ref{Prokh}. To this end, we need to exploit the two-stage construction of $(X(A), \Sigma_I)$ introduced in Lemma \ref{cor::double-proj-lim-to-char-space}. Beside $I$ as in~\eqref{main-indexset} we need to consider the set
\begin{equation}\label{main-superindexset}
\superindexset:=\{T\subseteq A: T\text{ countably generated subalgebra of }A\}
\end{equation}
ordered by the inclusion. Recall that $\{(X(A),\Sigma_\indexset),\pi_T,\superindexset\}$ is also the projective limit of the projective system $\{(X(T),\Borchar{T}),\pi_{T,R},\superindexset\}$.

By combining Lemma~\ref{cor::double-proj-lim-to-char-space} with Remark \ref{doub-exact}, we easily obtain:

\begin{lem}\label{doub-exact-char}
Let $\{\mu_S: S\in I\}$ be an exact projective system of Radon probability measures w.r.t.\ $\{(X(S),\Borchar{S}),\pi_{S,T},\indexset\}$. Let $\Lambda$ be as in \eqref{main-superindexset}. Then we get:
\begin{enumerate}[(i)]
\item For any $T\in\Lambda$, there exists a unique Radon probability measure $\mu_T$ on $(X(T),\Borchar{T})$ such that ${\pi_{S, T}}_{\#}\mu_T=\mu_S$ for all $S\in I$ with $S\subseteq T$. 
\item The family $\{\mu_T:T\in\Lambda\}$ is an exact projective system of Radon probability measures w.r.t.\! $\{(X(T),\Borchar{T}),\pi_{T,R},\superindexset\}$.
\end{enumerate}
\end{lem}

Since Theorem \ref{cylinder-MP} is a special case of Theorem \ref{cylinder-MP-Support} (taking $K^{(S)}=X(S)$ for each $S\in \indexset$), we shall focus on Theorem \ref{cylinder-MP-Support}. To apply Theorem \ref{Prokh} in establishing Theorem \ref{cylinder-MP-Support}-(ii), we will make use of Lemma \ref{lem::ex-ex-proj-sys}-(ii).  However, to apply Theorem~\ref{ext-cylinderSigmaAlg} in establishing Theorem \ref{cylinder-MP-Support}-(i), Lemma \ref{lem::ex-ex-proj-sys}-(i) will not be sufficient. Indeed, Lemma \ref{lem::ex-ex-proj-sys}-(i) ensures the existence of an exact projective system $\{\nu_S : S\in I \}$ of representing Radon measures for $L\!\restriction_S$, but the corresponding exact projective system $\{\nu_T: T\in\Lambda\}$ given in Lemma \ref{doub-exact-char}-(ii) need not be thick (though we are not aware of an example).

\begin{thm}\label{cylinder-MP}
Let $A$ be an algebra, $L:A \to \RR$ linear, $L(1)=1$ and 
$\indexset$ as in~\eqref{main-indexset}.

\begin{enumerate}[(i)]
\item There exists an $X(A)-$representing measure $\mu$ for $L$ on the cylinder $\sigma-$algebra $\Sigma_\indexset$ if and only if for every $S\in \indexset$ there exists an $X(S)-$representing Radon measure $\mu_S$ for $L\!\restriction_S$ such that $\{\mu_S: S\in I\}$ is an exact projective system and the corresponding exact projective system $\{\mu_T:  T\in \Lambda\}$ is thick (here $\superindexset$ is as in \eqref{main-superindexset} and $\mu_T$ as in Lemma~\ref{doub-exact-char}).

\item There exists an $X(A)-$representing Radon measure $\nu$ for $L$ if and only if for every $S\in \indexset$ there exists an $X(S)-$representing Radon measure $\mu_S$ for $L\!\restriction_S$ such that $\{\mu_S: S\in I\}$ fulfills Prokhorov's condition \eqref{epsilon-K-char}.
\end{enumerate}
 \end{thm}

More generally:
\begin{thm}\label{cylinder-MP-Support}
Let $A$ be an algebra, $L:A \to\RR$ linear, $L(1)=1$ and 
$\indexset$ as in~\eqref{main-indexset}.  For any $S\in \indexset$, let $K^{(S)}$ be a closed subset of $X(S)$ such that $\pi_{S,T}\left(K^{(T)}\right)\subseteq K^{(S)}$ for any $T\in \indexset$ containing~$S$.
\begin{enumerate}[(i)]
\item 
There exists an $X(A)-$representing measure $\mu$ for $L$ on the cylinder $\sigma-$algebra~$\Sigma_\indexset$ supported in each $\pi_S^{-1}\left(K^{(S)}\right)$ with $S\in \indexset$ if and only if for every $S\in \indexset$ there exists a $K^{(S)}-$representing Radon measure $\mu_S$ for $L\!\restriction_S$ such that $\{\mu_S: S\in I\}$ is an exact projective system and the corresponding exact projective system $\{\mu_T:  T\in \Lambda\}$ is thick (here $\superindexset$ is as in \eqref{main-superindexset} and $\mu_T$ as in Lemma~\ref{doub-exact-char}).
\item
There exists a $\left(\bigcap_{S\in \indexset}\pi_S^{-1}\!\left(K^{(S)}\right)\!\right)$--repre\-senting Radon measure $\nu$ for $L$ if and only if for every $S\in \indexset$ there exists a $K^{(S)}-$representing Radon measure $\mu_S$ for $L\!\restriction_S$ such that $\{\mu_S: S\in I\}$ fulfills Prokhorov's  condition \eqref{epsilon-K-char}.
\end{enumerate}
\end{thm}

\newpage
\begin{proof}\
\begin{enumerate}[(i)] 
\item Suppose that for each $S\in \indexset$ there exists a $K^{(S)}-$representing Radon measure for $L\!\restriction_S$ such that $\{\mu_S:S\in \indexset\}$ is an exact projective system and the corresponding exact projective system $\{\mu_T:  T\in \Lambda\}$ is thick. Then Theorem~\ref{ext-cylinderSigmaAlg} ensures that there exists a unique measure $\mu$ on the cylinder $\sigma-$algebra $\Sigma_\indexset$ on $X(A)$ such that 
\be\label{compatib2}
{\pi_{S}}_{\#}\mu=\mu_S,\quad\forall S\in \indexset.
\ee 
By Remark \ref{ext-cylinderSigmaAlg-rem-support}, $\mu$ is supported in each $\pi_S^{-1}(K^{(S)})$ with $S\in \indexset$. For any $a\in A$, there exists  $S\in \indexset$ such that $a\in S$ and so
$$
L(a)=\int_{X(S)}\!\!\hat{a}(\beta)\dd\mu_S(\beta)\stackrel{\eqref{compatib2}}{=}\int_{X(A)}\!\!\hat{a}(\pi_{S}(\beta))\dd\mu(\beta)\stackrel{\eqref{comp-Gelfand-transform}}{=}\int_{X(A)}\!\!\hat{a}(\alpha) \dd\mu(\alpha). 
$$

Conversely, if there exists an $X(A)-$representing measure $\mu$ for $L$ on the cylinder $\sigma-$algebra~$\Sigma_\indexset$ supported in each $\pi_S^{-1}\left(K^{(S)}\right)$ with $S\in \indexset$, then for every $S\in \indexset$ the measure ${\pi_S}_{\#}\mu$ is a $K^{(S)}-$representing Radon measure for $L\restriction_S$ and Theorem~\ref{ext-cylinderSigmaAlg} ensures that the exact projective system $\{{\pi_T}_{\#}\mu: T\in\Lambda\}$, which corresponds to $\{{\pi_S}_{\#}\mu: S\in\indexset\}$, is thick.

\item Suppose that for each $S\in \indexset$ there exists a $K^{(S)}-$representing Radon measure for $L\!\restriction_S$ such that $\{\mu_S:S\in \indexset\}$ is an exact projective system which
fulfills~\eqref{epsilon-K-char}. Then, replacing Theorem~\ref{ext-cylinderSigmaAlg} by Theorem~\ref{Prokh} and Remark~\ref{ext-cylinderSigmaAlg-rem-support} by Remark~\ref{Prokh-rem-support} in the proof of (i), we get that there exists a $(\bigcap_{S\in \indexset}\pi_S^{-1}(K^{(S)}))-$re\-presenting Radon measure $\nu$ for $L$.

Conversely, if there exists a $\left(\bigcap_{S\in \indexset}\pi_S^{-1}\!\left(K^{(S)}\right)\!\right)$--repre\-senting Radon measure $\nu$ for $L$, then for every $S\in \indexset$ the measure ${\pi_S}_{\#}\mu$ is a $K^{(S)}-$representing Radon measure for $L\restriction_S$ as $\bigcap_{T\in \indexset}\pi_T^{-1}\!\left(K^{(T)}\right)\subseteq \pi_S^{-1}(K^{(S)})$ and Theorem~\ref{Prokh} ensures that exact projective system $\{{\pi_S}_{\#}\mu: S\in\indexset\}$ fulfills \eqref{epsilon-K-char}.
\end{enumerate}
\vspace{-0.3cm}
\end{proof}

\begin{cor}\label{main-KMP}
Let $A$ be an algebra, $L:A \to\RR$ linear, $L(1)=1$,
$\indexset$ as in~\eqref{main-indexset}, and $K\subseteq X(A)$ closed. Write $K=K_Q$ for some quadratic module $Q$ in $A$ (cf.\! Re\-mark~\ref{rem-count}-(i)). 
\begin{enumerate}[(i)]
\item There exists an $X(A)-$representing measure $\mu$ for $L$ on the cylinder $\sigma-$algebra $\Sigma_\indexset$ supported on each $\pi_S^{-1}(K_{Q\cap S})$ with $S\in \indexset$, where $K_{Q\cap S}$ is as in \eqref{bcsas-S}, if and only if for every $S\in \indexset$ there exists a $K_{Q\cap S}-$representing Radon measure $\mu_S$ for $L\!\restriction_S$ such that $\{\mu_S: S\in I\}$ is an exact projective system and the corresponding exact projective system $\{\mu_T:  T\in \Lambda\}$ is thick (here $\superindexset$ is as in \eqref{main-superindexset} and $\mu_T$ as in Lemma~\ref{doub-exact-char}). Moreover, if $Q$ is countably generated, then $\mu$ is a $K-$representing measure for $L$ on~$\Sigma_\indexset$.
\item There exists a $K-$repre\-senting Radon measure $\nu$ for $L$  if and only if for every $S\in \indexset$ there exists a $K_{Q\cap S}-$representing Radon measure $\mu_S$ for $L\!\restriction_S$ such that $\{\mu_S: S\in I\}$ fulfills Prokhrov's condition~\ref{epsilon-K-char}.
\end{enumerate}
\end{cor}

\begin{proof}
Let us observe that $\pi_{S,T}(K_{Q\cap T})\subseteq K_{Q\cap S}$ for any $S\subseteq T$ in $\indexset$ and, by Remark~\ref{rem-count}-(i), we have that $K=K_Q = \bigcap_{S\in\indexset}\pi_S^{-1}(K_{Q\cap S})$. Thus, (i) and (ii) follow by applying Theorem \ref{cylinder-MP-Support} with $K^{(S)}=K_{Q\cap S}$ for $S\in\indexset$.

In particular in (i), when $Q$ is countably generated, then by Remark \ref{rem-count}-(ii) there exists a countable subset $J$ of $\indexset$ such that $K=K_Q = \bigcap_{S\in J}\pi_S^{-1}(K_{Q\cap S})$ and so $\mu$ is supported in $K$ by Remark \ref{ext-cylinderSigmaAlg-rem-support}-(i).
\end{proof}
\newpage
\begin{rem}\label{rem::cylinder-MP-cofinal}\
\begin{enumerate}[(i)]
\item If we assume that the algebra $A$ is such that for each $T\in\Lambda$ the map $\pi_T$ is surjective, then the thickness assumption in Theorem \ref{cylinder-MP-Support} and Corollary \ref{main-KMP} is always verified (cf.\! Remark~\ref{prop::cyl-quasi-measure}-(i)), e.g., for $A=\RR[X_i: i\in \Omega]$ where $\Omega$ is an arbitrary index set.
\item If in Theorem \ref{cylinder-MP-Support} we also assume the \emph{uniqueness} of the $K^{(S)}-$representing Radon measure for $L\!\restriction_S$ for each $S\!\in\! \indexset$, then there exists at most one $X\!(A)-$re\-presenting measure $\mu$ for $L$ on $\Sigma_\indexset$ supported in each $\pi_S^{-1}\left(K^{(S)}\right)$ with $S\in \indexset$ (the same holds for Corollary \ref{main-KMP} replacing $K^{(S)}$ by $K_{Q\cap S}$). Indeed, assume there exists a unique $K^{(S)}-$representing Radon measure for $L\!\restriction_S$ for each $S\in \indexset$. If $\mu^\prime$ is another $X(A)-$representing measure of $L$ on $\Sigma_\indexset$ supported in each $\pi_S^{-1}\left(K^{(S)}\right)$ with $S\in \indexset$, then ${\pi_S}_{\#}\mu^\prime={\pi_S}_{\#}\mu$ as both pushforward measures are $K^{(S)}-$representing for $L\!\restriction_S$ for all $S\in \indexset$ and so $\mu^\prime=\mu$ by Theorem~\ref{ext-cylinderSigmaAlg}. Replacing Theorem~\ref{ext-cylinderSigmaAlg} with Theorem~\ref{Prokh}, we can show the uniqueness of the $(\bigcap_{S\in \indexset} \pi_S^{-1}(K^{(S)}))-$representing measure $\nu$ for $L$ on $\Borchar{A}$. 
\item Theorem \ref{cylinder-MP-Support} remains true if the existence of a $K^{(S)}-$representing measure for $L\!\restriction_S$ is assumed only for all $S$ in a \emph{cofinal} subset $\subindexset$ of $\indexset$ (the same holds for Corollary \ref{main-KMP} replacing $K^{(S)}$ by $K_{Q\cap S}$). In fact, by Proposition~\ref{prop::cofinal-proj-limit}, the projective limits of the corresponding projective systems give rise to the same measurable space $(X(A),\Sigma_\indexset)$. Furthermore, for $S\in \indexset$ there exists $T\in \subindexset$ such that $S\subseteq T$ and so $\pi_{S,T}(K^{(T)})\subseteq K^{(S)}$ implies $\pi_T^{-1}(K^{(T)})\subseteq \pi_S^{-1}(K^{(S)})$. Therefore, the representing measure $\mu$ with $\indexset$ replaced by $\subindexset$ is supported in each $\pi_S^{-1}(K^{(S)})$ with $S\in \indexset$ and $\nu$ on $\bigcap_{S\in \indexset} \pi_S^{-1}(K^{(S)})=\bigcap_{S\in \subindexset} \pi_S^{-1}(K^{(S)})$.

\item  Theorem \ref{cylinder-MP-Support} and Corollary \ref{main-KMP} remain true if we fix a set of generators $G$ of $A$ and replace $I$ with its cofinal subset $J:=\{\langle F \rangle : F\subseteq G \text{ finite}\}$.
\end{enumerate}
\end{rem}

Theorem \ref{cylinder-MP-Support} and Corollary \ref{main-KMP}  build a bridge between the moment theory in finite and infinite dimensions. Indeed, they allow us to lift many of the results for the finite dimensional moment problem up to the general setting of Problem~\ref{GenKMP}.

\subsection{Riesz-Haviland theorem}\label{RH-Thm}\ 

In this subsection we are going to generalize the classical Riesz-Haviland theorem \cite{Riesz}, \cite{Hav} for the polynomial algebra to any unital commutative real algebra. To this purpose, we state in Theorem \ref{thm::fin-dim-Haviland} a version of the Riesz-Haviland theorem for finitely generated algebras which is an easy generalization of the classical result (for a proof see \cite[Theorem 1.14]{Schm17}).

Given a unital commutative real algebra $A$ and $K\subseteq X(A)$, we define
$$
\mathrm{Pos}_A(K):=\{a\in A: \hat{a}_A(\alpha)\geq 0\text{ for all }\alpha\in K\}.
$$

\begin{thm}\label{thm::fin-dim-Haviland}
Let $A$ be a finitely generated algebra, $L: A \rightarrow \mathbb{R}$ linear, $L(1)=1$ and $K$ a closed subset of $X(A)$. If $L(\mathrm{Pos}_A(K))\subseteq [0,\infty)$, then there exists a $K-$representing Radon measure for $L$ on $\Borchar{A}$.
\end{thm}

As a direct consequence of Theorem \ref{cylinder-MP-Support} and Corollary \ref{main-KMP}, we get the following generalized version of the Riesz-Haviland theorem.

\begin{thm}\label{cor::Haviland}
Let $A$ be an algebra, $L:A \rightarrow \mathbb{R}$ linear, $L(1)=1$, 
$\indexset$ as in~\eqref{main-indexset}, and $K\subseteq X(A)$ closed, i.e., $K=K_Q$ for some quadratic module $Q$ in $A$ (cf.\! Re\-mark~\ref{bcsas-S}-(i)).

If $L(\mathrm{Pos}_A(K))\subseteq [0,\infty)$, then 
\begin{enumerate}[(i)] 
\item  For each $S\in I$ there exists a  $\overline{\pi_S(K)}-$representing Radon measure~$\mu_S$ for~$L\restriction_S$.
 \item There exists an $X(A)-$representing measure $\mu$ for $L$ on the cylinder $\sigma-$algebra $\Sigma_\indexset$ supported on each $\pi_S^{-1}(\overline{\pi_S(K)})$ with $S\in \indexset$ if and only if $\{\mu_S: S\in I\}$ is an exact projective system and the corresponding exact projective system $\{\mu_T:  T\in \Lambda\}$ is thick. Moreover, if $Q$ is countably generated, then $\mu$ is a $K-$representing measure for $L$ on~$\Sigma_\indexset$.
\item There exists a $K-$repre\-senting Radon measure $\nu$ for $L$  if and only if $\{\mu_S: S\in I\}$ fulfills Prokhorov's condition \eqref{epsilon-K-char}.
\end{enumerate}
\end{thm}
\proof
Let $S\in \indexset$ and suppose $L(\mathrm{Pos}_A(K))\subseteq [0,\infty)$. Since $\mathrm{Pos}_S(\overline{\pi_S(K)})\subseteq\mathrm{Pos}_A(K)\cap S$, we have $L\!\restriction_S(\mathrm{Pos}_S(\overline{\pi_S(K)}))\subseteq [0,\infty)$. Hence, by Theorem \ref{thm::fin-dim-Haviland}, there exists a $\overline{\pi_S(K)}-$representing Radon measure $\mu_S$ for $L\!\restriction_S$, i.e. (i) holds. Each $\mu_S$ is also $K_{Q\cap S}-$representing for $L\!\restriction_S$ as $\overline{\pi_S(K_Q)}\subseteq K_{Q\cap S}$, where $K_{Q\cap S}$ is as in \eqref{bcsas-S}. Then, by applying  Corollary~\ref{main-KMP}, we get the conclusions (ii) and (iii). 
\endproof

When $A=\RR[X_i : i\in\Omega]$ with $\Omega$ arbitrary index set, the map $\pi_T$ is surjective for any $T\in \Lambda$. Hence, by combining Remark \ref{rem::cylinder-MP-cofinal}-(i), Theorem \ref{cor::Haviland} and Lemma \ref{lem::constrRadonIJ} for $A=\RR[X_i : i\in\Omega]$ with $\Omega$ arbitrary index set, we retrieve \cite[Theorem 5.1]{GKM16}. If $\Omega$ is countable, we similarly derive \cite[Theorem 5.1]{AJK15} from Theorem~\ref{cor::Haviland}. Another special case of Theorem \ref{cor::Haviland} is the following result \cite[Theorem 3.2.2]{MarshBook}.

\begin{thm}
Let $A$ be an algebra and $K\subseteq X(A)$ closed. Assume there exists $p\in A$ such that $\hat{p}\geq 0$ on $K$ and $K_n:=\{\alpha\in K: \hat{p}(\alpha) \leq n\}$ is compact for each $n\in\NN$. Then for each $L:A\to\RR$ linear, $L(1)=1$ and $L(\mathrm{Pos}_A(K))\subseteq [0,\infty)$ there exists a $K-$representing Radon measure $\nu$ for $L$.
\end{thm}
\begin{proof}
Define $J:=\{S\in I: p\in S\}\subseteq \indexset$, where $\indexset$ is as in \eqref{main-indexset}. By Theorem \ref{cor::Haviland}-(i), for each $S\in J$ there exists a  $\overline{\pi_S(K)}-$representing Radon measure $\mu_S$ for $L\restriction_S$. As $J$ is cofinal in $\indexset$, by Theorem \ref{cor::Haviland}-(iii) and Remark \ref{rem::cylinder-MP-cofinal}-(iii), we just need to show that $\{\mu_S:S\in J\}$ fulfills Prokhorov's condition \eqref{epsilon-K-char}. 

Let us first prove that  
\begin{equation}\label{eq::closure}
\overline{\pi_S(K)}=\bigcup_{n\in\NN}\pi_S(K_n)=\pi_S(K), \quad\forall S\in J.
\end{equation}
Let $S\in J$ and w.l.o.g.\! assume that $p$ is one of the generators $a_1,\ldots,a_d$ of $S$. Let $\alpha\in\overline{\pi_S(K)}$ and note that the sets 
$$
\{\beta\in X(S):\left|\hat{a}_i(\beta)-\hat{a}_i(\alpha) \right|\leq r\text{ for all }i=1,\ldots,d \},\quad\text{ where }0<r\leq 1,
$$ 
form a basis of neighborhoods of $\alpha$ in $(X(S),\topchar{S})$. For any such basic neighborhood $U$ of $\alpha$ there exists some $\beta\in \pi_S(K)\cap U$ as $\alpha\in\overline{\pi_S(K)}$. Thus, $\left|\hat{p}(\beta)-\hat{p}(\alpha)  \right|\leq 1$ and hence, there exists $n\in\NN$ (independent of $\beta$) such that $\hat{p}(\beta)\leq \hat{p}(\alpha)+1\leq n$. Since $\beta\in\pi_S(K)$, this shows $\beta\in \pi_S(K_n)\cap U$. As $U$ was an arbitrary basic neighborhood of $\alpha$, we conclude $\alpha\in\overline{\pi_S(K_n)}=\pi_S(K_n)$ and so $\alpha\in\bigcup_{n\in\NN}\pi_S(K_n)$. The other inclusion in \eqref{eq::closure} easily follows from $K=\bigcup_{n\in\NN}K_n$.

Now let $\varepsilon>0$ and set $S:=\langle p\rangle\in J$. Then \eqref{eq::closure} and the continuity from below of $\mu_S$ imply that $\lim_{n\to\infty} \mu_S(\pi_S(K_n))= \mu_S(\overline{\pi_S(K)})=1$
and so there exists $n\in\NN$ such that $\mu_S(\pi_S(K_n))\geq 1-\varepsilon$. Moreover, for any $T\in J$ we clearly have $S\subseteq T$ and it is easy to show that 
\begin{equation}\label{eq::repr2}
\pi_T(K_n)=\pi_{S,T}^{-1}(\pi_S(K_n))\cap\pi_T(K).
\end{equation}
Thus, using \eqref{eq::closure} and $\mu_T(\pi_T(K))=1$, we get
\begin{eqnarray*}
\mu_T(\pi_T(K_n)) &\stackrel{\eqref{eq::repr2}}{=}&\mu_T(\pi_{S,T}^{-1}(\pi_S(K_n))\cap\pi_T(K))\\
&=&\mu_T(\pi_{S,T}^{-1}(\pi_S(K_n)))\\
&=&\mu_S(\pi_S(K_n))\geq 1-\varepsilon,
\end{eqnarray*}
which proves that $\{\mu_T: T\in J\}$ fulfills \eqref{epsilon-K-char}.

It remains to show \eqref{eq::repr2}. Clearly, $\pi_T(K_n)\subseteq\pi_T(K)$ and 
$$
\pi_T(K_n)\subseteq\pi_{S,T}^{-1}(\pi_{S,T}(\pi_T(K_n)))=\pi_{S,T}^{-1}(\pi_S(K_n)).
$$
Conversely, let $\alpha\in\pi_{S,T}^{-1}(\pi_S(K_n))\cap\pi_T(K)$, then there exists $\beta\in K_n$ such that $\pi_{S,T}(\alpha)=\pi_S(\beta)$ and $\gamma\in K$ such that $\alpha=\pi_T(\gamma)$. Since $p\in S\subseteq T$, we get $\gamma(p)=\alpha(p)=\beta(p)\leq n$. Hence, $\gamma\in K_n$ and so $\alpha=\pi_T(\gamma)\in\pi_T(K_n)$.
\end{proof}

\subsection{Nussbaum theorem}\label{Nuss-Thm}\

In this subsection we are going to extend the classical Nussbaum theorem \cite[Theorem~10]{N65} (see also \cite[Theorem 14.19]{Schm17}) to any unital commutative real algebra. Note that the classical Nussbaum theorem has been recently generalized to the case of representing Radon measures supported in subsets of $\RR^d$ determined by polynomial inequalities (see \cite[Theorem 3.2]{Lass} and \cite[Theorem 14.25]{Schm17} for the case of finitely many polynomial inequalities, and \cite[Theorem 5.1]{IKR} for the case of arbitrarily many polynomial inequalities). Therefore, we will apply a quite standard procedure to first generalize this result from the polynomial algebra in finitely many variables to any finitely generated algebra.

\begin{thm}\label{Nussbaum-support}
Let $A$ be finitely generated by $a_1,\ldots,a_d\in A$ with $d\in\NN$, $L:A\to\RR$ linear and $L(1)=1$. If 
\begin{enumerate}[(a)]
\item $L(Q)\subseteq [0,\infty)$ for some quadratic module $Q$ in $A$,   
\item for each $i\in\{1,\ldots, d\}$, $\sum_{n=1}^\infty\frac{1}{\sqrt[2n]{L(a_i^{2n})}}=\infty$ holds, 
\end{enumerate}  
then there exists a unique $K_Q-$representing Radon measure $\nu$ for $L$, where $K_Q$ is as in \eqref{semi-alg-Q}.
\end{thm}
\begin{proof} Set $\RR[\underline{X}]:=\RR[X_1,\ldots,X_d]$. By Remark \ref{fin-dim-char}, the map $\varphi:\RR[\underline{X}]\to A$ given by $\varphi(X_i):=a_i$ is surjective and the map $\psi:\Z(\ker(\varphi))\to X(A)$ given by $\psi(x)(\varphi(p)):=p(x)$ is a topological isomorphism. If $\widetilde{Q}:=\varphi^{-1}(Q)$, then $K_{\widetilde{Q}}=\psi^{-1}(K_Q)$, where $K_{\widetilde{Q}}:=\{x\in \Z(\ker(\varphi)): {q}(x)\geq 0\text{ for all }q\in \widetilde{Q}\}$. The linear functional $\widetilde{L}:=L\circ\varphi$ on $\RR[\underline{X}]$ satisfies $\widetilde{L}(\widetilde{Q})\subseteq [0,\infty)$ as $L(Q)\subseteq [0,\infty)$ by assumption~(a). Moreover, (b) ensures that $\widetilde{L}$ fulfills the classical Carleman condition (see \cite{Carl26, Denj21}), i.e. $\sum_{n=1}^\infty\frac{1}{\sqrt[2n]{\widetilde{L}(X_i^{2n})}}=\infty$, as $\widetilde{L}(X_i^{2n})=L(a_i^{2n})$ for all $n \in\NN$ and all $i\in\{1,\ldots,d\}$. Hence, there exists a unique $K_{\widetilde{Q}}-$representing Radon measure $\tilde{\mu}$ for $\widetilde{L}$ on $\RR^d$ (see \cite[Theorem~10]{N65} and \cite[Theorem 3.2]{Lass} or \cite[Theorems~14.19 and~14.25]{Schm17} when $\widetilde{Q}$ is finitely generated, while see \cite[Theorem 5.1]{IKR} when $\widetilde{Q}$ is not finitely generated). It is then easy to show that $\mu:={\psi}_{\#}\tilde{\mu}$ is the unique $K_Q-$representing Radon measure for $L$ on $X(A)$.
\end{proof}

As a direct consequence of Corollary \ref{main-KMP} and Theorem~\ref{Nussbaum-support}, we get the following result, which can be seen as a version of the classical Nussbaum theorem for any unital commutative real algebra. 

\begin{thm}\label{cylinder-MP2}
Let $A$ be generated by $G\subseteq A$, $L:A \rightarrow \mathbb{R}$ linear, $L(1)=1$, and $I$ as in \eqref{main-indexset}. If
\begin{enumerate}[(a)]
\item $L(Q)\subseteq [0,\infty)$ for some quadratic module $Q$ in $A$,
\item for each $g\in G$,  $\sum_{n=1}^\infty\frac{1}{\sqrt[2n]{L(g^{2n})}}=\infty$ holds,
\end{enumerate}
then 
\begin{enumerate}[(i)]
\item There exists an exact projective system $\{\mu_S: S\in I\}$, where each $\mu_S$ is the unique $K_{Q\cap S}-$representing Radon measure $\mu_S$ for $L\!\restriction_S$ and $K_{Q\cap S}$ is as in \eqref{bcsas-S}.
\item There exists a unique $X(A)-$representing measure $\mu$ for $L$ on the cylinder $\sigma-$algebra~$\Sigma_\indexset$ supported on each $\pi_S^{-1}(K_{Q\cap S})$ with $S\in \indexset$ if and only if the exact projective system $\{\mu_T:  T\in \Lambda\}$ corresponding to $\{\mu_S: S\in I\}$ is thick. Moreover, if $Q$ is countably generated, then $\mu$ is a $K_Q-$representing measure for $L$ on~$\Sigma_\indexset$.
\item There exists a unique $K_Q-$repre\-senting Radon measure $\nu$ for $L$ if and only if $\{\mu_S: S\in I\}$ fulfills Prokhorov's condition \eqref{epsilon-K-char}.
\end{enumerate}
\end{thm}

\begin{proof} 
Let $F\subseteq G$ finite and set $S:=\langle F \rangle\in \subindexset$. Then (a) and (b) respectively provide that $L\!\restriction_S(Q\cap S)\subseteq [0,\infty)$ and $\sum_{n=1}^\infty\frac{1}{\sqrt[2n]{L(g^{2n})}}=\infty$ holds for each $g\in F$. Thus, by Theorem~\ref{Nussbaum-support}, there exists a unique $K_{Q\cap S}-$representing Radon measure $\mu_S$ for $L\!\restriction_S$ and so Lemma~\ref{lem:uniq-S} guarantees that $\{\mu_S: S\in I\}$ is exact, i.e., (i) holds. By Since $
\subindexset:=\{\langle F \rangle :F\subseteq G\text{ finite}\}.
$ is cofinal in $I$, Remark \ref{rem::cylinder-MP-cofinal}-(iv) combined with Corollary~\ref{main-KMP} yield conclusions (ii) and (iii).
\end{proof}

For the convenience of the reader, we state here a special case of Theorem \ref{cylinder-MP2} without prescribing a support, i.e., taking $Q=\sum A^2$.

\begin{thm}\label{thm::general-Nussbaum}
Let $A$ be generated by $G\subseteq A$, $L:A \rightarrow \mathbb{R}$ linear, $L(1)=1$, and $I$ as in \eqref{main-indexset}. If
\begin{enumerate}[(a)]
\item $L(a^2) \geq 0$ for all $a \in A$,   
\item for each $g \in G$, $\sum_{n=1}^\infty\frac{1}{\sqrt[2n]{L(g^{2n})}}=\infty$ holds, \end{enumerate}  
then 
\begin{enumerate}[(i)]
\item There exists an exact projective system $\{\mu_S: S\in I\}$, where each $\mu_S$ is the unique $X(S)-$representing Radon measure $\mu_S$ for $L\!\restriction_S$.
\item There exists a unique $X(A)-$representing measure $\mu$ for $L$ on the cylinder $\sigma-$algebra~$\Sigma_\indexset$ if and only if the exact projective system $\{\mu_T:  T\in \Lambda\}$ corresponding to $\{\mu_S: S\in I\}$ is thick. 
\item There exists a unique $X(A)-$repre\-senting Radon measure $\nu$ for $L$ if and only if $\{\mu_S: S\in I\}$ fulfills Prokhorov's condition \eqref{epsilon-K-char}.
\end{enumerate}
\end{thm}

From Theorem \ref{thm::general-Nussbaum} we can derive the following result \cite[Corollary 4.8]{GKM16}.

\begin{cor}\label{Cor-MP-poly}
Let $\RR[X_i: i\in \Omega]$ be the algebra of real polynomials in the variables $\{X_i: i\in \Omega\}$ with $\Omega$ arbitrary index set, $L:\RR[X_i: i\in \Omega] \rightarrow \mathbb{R}$ linear, $L(1)=1$. If
\begin{enumerate}[(a)]
\item\label{Psdness} $L(p^2) \geq 0$ for all $p \in \RR[X_i: i\in \Omega]$,  
\item\label{Carleman-Coordinates} for all $i\in\Omega$, Carleman's condition $\sum_{n=1}^\infty\frac{1}{\sqrt[2n]{L(X_i^{2n})}}=\infty$ is satisfied,
\end{enumerate}  
then there exists a unique $\RR^\Omega-$representing constructibly Radon measure $\mu$ for $L$.
\end{cor}
\proof 
First, we note that the assumptions (a) and (b) in Theorem \ref{thm::general-Nussbaum} are fulfilled for $A=\RR[X_i:i\in\Omega]$. Moreover, in this case, the map $\pi_T$ is surjective for any $T\in \Lambda$ and so the thickness assumption in Theorem \ref{thm::general-Nussbaum}-(ii) is guaranteed (see Remark \ref{rem::cylinder-MP-cofinal}-(i)). Therefore, by applying  Theorem \ref{thm::general-Nussbaum}-(i) and Theorem \ref{thm::general-Nussbaum}-(ii), we get that there exists a unique $\RR^\Omega-$representing measure $\mu$ for $L$ on the cylinder $\sigma-$algebra on $\RR^\Omega$, which is in fact constructibly Radon by Lemma \ref{lem::constrRadonIJ}.
\endproof

If $\Omega$ is countable, then $\RR[X_i: i\in \Omega]$ is a countably generated algebra and so by Remark \ref{rem::constr-Borel}-(ii) the representing measure $\mu$ given by Corollary \ref{Cor-MP-poly} is Radon. When $\Omega$ is finite, Corollary \ref{Cor-MP-poly} reduces to Nussbaum's theorem \cite[Theorem~10]{N65}.\par\medskip

If $A=S(V)$ is the symmetric algebra of a real vector space $V$, then the map $\pi_T$ is surjective for any $T\in\Lambda$ and so we obtain the following result by applying Theorem \ref{cor::Haviland}-(ii) under the hypothesis (a) below and Theorem \ref{cylinder-MP2}-(ii) under the hypothesis (b).

\begin{cor}\label{cor::symmetric-algebra}
Let $S(V)$ be the symmetric algebra of a real vector space $V$, $L:S(V)\to\RR$ linear and $L(1)=1$. Assume that either
\begin{enumerate}[(a)]
\item $L(\mathrm{Pos}(X(S(V))))\subseteq [0,\infty)$, or
\item $L(a^2)\geq 0$ for all $a\in S(V)$ and $\sum_{n=1}^\infty\tfrac{1}{\sqrt[2n]{L(a^{2n})}}=\infty$ for all $a\in V$.
\end{enumerate}
Then there exists an $X(S(V))-$representing measure $\mu$ for $L$ on $\Sigma_\indexset$.
\end{cor}

\begin{rem}
Let us provide an alternative proof of \cite[Theorem~16, equation~(21)]{Schmu-new} by means of Corollary~\ref{cor::symmetric-algebra}. 

Denote by $V^\ast$ the algebraic dual of $V$ and for any finite dimensional subspace $F$ of $V$ let $\xi_{F}: V^*\to F^*$ be the restriction map. Then $\psi_V: (X(S(V)), \Sigma_I) \to V^\ast$, $\alpha\mapsto\alpha\!\restriction_V$ is bi-measurable when $V^*$ is equipped with the cylinder~$\sigma-$algebra $\Sigma_{V^*}$, i.e., the smallest $\sigma-$algebra on $V^\ast$ such that $\xi_{F}$ is measurable for all finite dimensional subspaces $F$ of $V$ and when $F^\ast$ equipped with the Borel $\sigma-$algebra $\B_{F^\ast}$. Thus, by Corollary \ref{cor::symmetric-algebra} we get that for any $a\in S(V)$
$$
L(a)=\int_{X(S(V))}\hat{a}_{S(V)}\dd\mu = \int_{X(S(V))}\hat{a}_{S(V)}\circ \psi_V^{-1}\circ \psi_V\dd\mu =\int_{V^\ast}\hat{a}_V\dd{\psi_V}_{\#}\mu, 
$$
where $\hat{a}_V:=\hat{a}_{S(V)}\circ \psi_V^{-1}$.
Hence, ${\psi_V}_{\#}\mu$ is a $V^\ast-$representing measure on $\Sigma_{V^*}$. In \cite[Theorem 16, equation~(21)]{Schmu-new} the author obtains what we here call a representing cylindrical quasi-measure, namely the measure ${\psi_V}_{\#}\mu$ restricted to the cylinder algebra $\{\xi_F^{-1}(E): \text{$F$ finite dimensional subspace of $V$ and } E\in\B_{F^*}\}$.

Note that \cite[Theorem 16, equation (23)]{Schmu-new} exploits a certain topological structure on~$V$  and so it goes beyond the scope of this paper.
\end{rem}

Condition (b) in Theorem \ref{cylinder-MP2} can be relaxed as shown by the following result, obtained by adapting the proofs of \cite[Corollaries 4.7 and 4.8]{M} to our general setting. The new condition so obtained is strictly weaker than (b) in Theorem \ref{cylinder-MP2} (argue as in \cite[Remark~4.7]{GKM16}).

\begin{thm}\label{thm::weaker-Carleman}
Let $A$ be generated by $G\subseteq A$, $L:A \rightarrow \mathbb{R}$ linear, $L(1)=1$, and $\indexset$ as in \eqref{main-indexset}. If  
\begin{enumerate}[(a)]
\item $L(Q)\subseteq [0,\infty)$ for some quadratic module $Q$ in $A$,
\item $\forall g\in G, \,\exists\, (q_{g,k})_{k\in\NN}\subseteq A\otimes\CC: \lim_{k\to\infty}L\left(\left|1-(1+g^2)q_{g,k}\overline{q_{g,k}}\right|^2\right)=0$,
\end{enumerate}
then 
\begin{enumerate}[(i)]
\item There exists an exact projective system $\{\mu_S: S\in I\}$, where each $\mu_S$ is the unique $X(S)-$representing Radon measure $\mu_S$ for $L\!\restriction_S$.
\item There exists a unique $X(A)-$representing measure $\mu$ for $L$ on the cylinder $\sigma-$algebra~$\Sigma_\indexset$ if and only if the exact projective system $\{\mu_T:  T\in \Lambda\}$ corresponding to $\{\mu_S: S\in I\}$ is thick. 
\item There exists a unique $X(A)-$repre\-senting Radon measure $\nu$ for $L$ if and only if $\{\mu_S: S\in I\}$ fulfills Prokhorov's condition \eqref{epsilon-K-char}.
\end{enumerate}
\end{thm}
\begin{proof}
As in the proof of Theorem \ref{cylinder-MP2}, let
$
\subindexset:=\{\langle F\rangle :F\subseteq G\text{ finite}\}.
$ Let $F\subseteq G$ be finite and $S:=\langle F\rangle\in \subindexset$. Let us verify that three assumptions of \cite[Corollary~3.4]{M03} are fulfilled by the algebra $S\in J$, the element $p:=\prod_{g\in F}(1+g^2)\in S$ and the quadratic module $
(Q\cap S)_p:=\{a\in S: p^{2m}a\in Q\cap S\text{ for some }m\in\NN_0\}
$
in~$S$.
\begin{enumerate}[(1)]
\item Using the  the definition of $p$, we have that $p-1\in (Q\cap S)_p$.
\item Since $p-\sum_{g\in F}b^2\in Q\cap S$, \cite[Proposition 4.2]{M03} guarantees that for any $a\in A$ there exist $n,l\in\NN$ such that $lp^n-a\in (Q\cap S)_p$.
\item Let us show that $L\!\restriction_S((Q\cap S)_p)\subseteq [0,\infty)$. Since for each $a\in S,g\in F$ and $k\in\NN$ the following holds by the Cauchy-Schwarz inequality
$$
L(a-(1+g^2)q_{g,k}\overline{q_{g,k}}a)^2\leq L(a^2)L\left(\left|1-(1+g^2)q_{g,k}\overline{q_{g,k}}\right|^2\right),
$$
we obtain by (b) that
\begin{equation}\label{eq::limit-1}
L(a)=\lim_{k\to\infty}L((1+g^2)q_{g,k}\overline{q_{g,k}}a).
\end{equation}
Now let $a\in (Q\cap S)_p$, i.e., $p^{2m}a\in Q\cap S$ for some $m\in\NN_0$.  Then, by \eqref{eq::limit-1}, for any $g\in F$ we get
$$
\qquad L\left(\tfrac{p^{2m}}{(1+g^2)}a\right) = \lim_{k\to\infty}L\left((1+g^2)q_{g,k}\overline{q_{g,k}}\tfrac{p^{2m}}{(1+g^2)}a\right)=\lim_{k\to\infty}L\left(q_{g,k}\overline{q_{g,k}}p^{2m}a\right)\geq 0
$$
as $q_{g,k}\overline{q_{g,k}}p^{2m}a\in Q$. Iterating this procedure yields $L\!\restriction_S(a)=L(a)\geq 0$.
\end{enumerate}
Thus, by using \cite[Corollary 3.4, Remark 3.5]{M03} and that $K_{Q\cap S}=K_{(Q\cap S)_p}$, we obtain that there exists a unique $K_{Q\cap S}-$representing Radon measure $\mu_S$ for $L\!\restriction_S$. Hence, we can proceed as in the proof of Theorem \ref{cylinder-MP2} and get the conclusions.
\end{proof}

Applying Theorem \ref{thm::weaker-Carleman} to $A=\RR[X_i:i\in\Omega]$ with $\Omega$ an arbitrary index set, we retrieve \cite[Theorems 4.6 and 5.2]{GKM16}.

\subsection{The Archimedean condition}\label{appl-Arch}\

The second parts of Theorem \ref{cylinder-MP-Support} and Corollary \ref{main-KMP} highlight the importance of  the additional assumption \eqref{epsilon-K-char} in getting a representing Radon measure for the starting functional $L$. This leads us to the fundamental problem of understanding under which assumptions on $L$ and $A$ the condition \eqref{epsilon-K-char} holds. We first demonstrate that the Archimedeanity of the quadratic module in Theorem \ref{cylinder-MP2} directly implies that \eqref{epsilon-K-char} holds, by giving an alternative proof of the solution to Problem~\ref{GenKMP} for Archimedean quadratic modules in Theorem~\ref{cylinder-KMP}. Our proof is more involved than the ones obtained via the Archimedean Positivstellensatz \cite[Theorem~4]{Jac} (see \cite[Corollary 3.3]{M03}) or via the GNS construction (along the lines of \cite[Theorem 12.35]{Schm17}). However, our proof of Theorem~\ref{cylinder-KMP} is instructive for our new result in Theorem~\ref{hybrid-quadraticMod}.

Recall that a quadratic module $M$ of $A$ is \emph{Archimedean} if for any $a\in A$ there exists $N\in\NN$ such that $N\pm a\in M$.

\begin{thm}\label{cylinder-KMP}
Let $A$ be an algebra, $L:A \rightarrow \mathbb{R}$ linear and $L(1)=1$.  If $L(Q)\subseteq[0,\infty)$ for some Archimedean quadratic module $Q$ in $A$, 
then there exists a unique $K_Q-$representing Radon measure $\nu$ for $L$, where $K_Q$ is as in \eqref{semi-alg-Q}.
\end{thm}
\begin{proof}
Since $Q$ is Archimedean, for each $a\in A$ there exists a non-negative integer $N$ such that $N^2- a^2\in Q$. Then the non-negativity of $L$ on $Q$ implies that $L(a^{2n})\leq N^{2n}$ for all $n\in\NN$ and so $\sum_{n=1}^\infty\tfrac{1}{\sqrt[2n]{L(a^{2n})}}=\infty$. Hence, by Theorem~\ref{cylinder-MP2}-(i), there exists an exact projective system $\{\mu_S: S\in I\}$ such that each $\mu_S$ is a $K_{Q\cap S}-$representing Radon measure $\mu_S$ for $L\!\restriction_S$.

Let $S\in \indexset$. Since $Q$ is an Archimedean quadratic module of $A$, we have that $Q\cap S$ is an Archimedean quadratic module of $S$. Hence, $K_{Q\cap S}$ is a compact subset of $X(S)$ (see \cite[Theorem 6.1.1]{MarshBook}). As $\mu_S(K_{Q\cap S})=1$ and $\pi_{T,R}(K_{Q\cap R})\subseteq K_{Q\cap T}$ for any $T, R\in \indexset$ such that $S\subseteq T\subseteq R$, the family $\{\mu_S:S\in \indexset\}$ fulfills \eqref{epsilon-K-char} by Proposition \ref{prop::compact-support-epsilon-K}. Theorem \ref{cylinder-MP2}-(iii) ensures that there exists a unique $K_Q-$representing Radon measure $\nu$ for $L$.
\end{proof}

Theorem~\ref{cylinder-KMP} can be applied to the algebra $C:=\RR\left[\frac{1}{1+X_i^2}, \frac{X_i}{1+X_i^2} : i\in\Omega\right]$ considered in \cite{GKM16}. Indeed, every quadratic module in $C$ is Archimedean (cf.\! \cite[Theorem~3.1-(2)]{GKM16}) and so Theorem~\ref{cylinder-KMP} ensures the existence of a $K_Q-$represent\-ing Radon measure for any linear functional on $C$ non-negative on any quadratic module $Q$.

\ \\
The following proposition is of independent interest and can be used to retrieve \cite[Theorem 5.5]{AJK15} from Theorem \ref{cylinder-KMP} applied to $\RR[X_i:i\in\Omega]$ with $\Omega$ countable. Recall that a \emph{preordering} $P$ in $A$ is a quadratic module in $A$ such that $P\cdot P\subseteq P$.

\begin{prop}\label{prop::compact-Archimedean}
Let $P$ be a preordering in an algebra $A$. If for all $S\in \indexset$, with $\indexset$ as in \eqref{main-indexset}, the preordering $P\cap S$ in $S$ is finitely generated and $K_{P\cap S}$ is compact, then $P$ is Archimedean.
\end{prop}
\begin{proof}
Let $S\in \indexset$. Since $S$ is finitely generated, $P\cap S$ is Archimedean by \cite[Theorem 6.1.1]{MarshBook} (which holds for any finitely generated algebra). Thus, if $a\in A$, then $a\in S$ for some $S\in \indexset$ and so there exists $N\in\NN$ such that $N\pm a\in P\cap S \subseteq P$.
\end{proof}

We will now show that it is possible to relax the assumptions in Theorem \ref{cylinder-KMP} and so to solve Problem \ref{GenKMP} for a larger class of linear functionals having a not necessarily compactly supported representing measure.

\begin{thm}\label{hybrid-quadraticMod}
Let $A$ be an algebra, $Q$ a quadratic module in $A$, $L:A\to\RR$ linear and $L(1)=1$. If $L(Q)\subseteq [0,\infty)$ and there exist subalgebras  $G_\mathrm{a},G_\mathrm{c}$ of $A$ such that 
\begin{enumerate}[(a)]
\item $G_\mathrm{a}\cup G_\mathrm{c}$ generates $A$ as a real algebra,
\item $Q\cap G_\mathrm{a}$ is Archimedean in $G_\mathrm{a}$,
\item $G_\mathrm{c}$ is countably generated and $\sum_{n=1}^\infty\tfrac{1}{\sqrt[2n]{L(g^{2n})}}=\infty$ for each $g\in G_\mathrm{c}$,
\end{enumerate}
then there exists a unique $K_Q-$representing Radon measure $\nu$ for $L$, where $K_Q$ is as in \eqref{semi-alg-Q}.\end{thm}
\begin{proof}
By assumption (b) we know that for each $g\in G_\mathrm{a}$ there exists $N\in\NN$ such that $N^2\pm b^2\in Q$. Then the non-negativity of $L$ on $Q$ implies that $L(g^{2n})\leq N^{2n}$ for all $n\in\NN$,  which implies in turn that  $\sum_{n=1}^\infty\tfrac{1}{\sqrt[2n]{L(g^{2n})}}=\infty$. The latter together with~(c) implies that $\sum_{n=1}^\infty\tfrac{1}{\sqrt[2n]{L(g^{2n})}}=\infty$ holds for all $g\in G_\mathrm{a}\cup G_\mathrm{c}$. Since $G_\mathrm{a}\cup G_\mathrm{c}$ generates $A$ by (a) and $L$ is non-negative on $Q$, we can apply Theorem~\ref{cylinder-MP2}-(i). This guarantees the existence of an exact projective system $\{\mu_S: S\in I\}$ such that each $\mu_S$ is a $K_{Q\cap S}-$representing measure for $L\!\restriction_S$. 

By Remark \ref{rem::cylinder-MP-cofinal}-(iv) and Theorem~\ref{cylinder-MP2}-(iii), it remains to show that $\{\mu_S:S\in J\}$ fulfills Prokhorov's condition \eqref{epsilon-K-char} for 
$$
J:=\{\langle F\rangle : F\subseteq G_\mathrm{a}\cup G_\mathrm{c}\text{ finite}\},
$$
as $G_\mathrm{a}\cup G_\mathrm{c}$ generates $A$ by (a). To this aim, we are going to use the characterization given in Proposition \ref{prop::compact-support-epsilon-K}.

Let $\varepsilon>0$. 
By Lemma~\ref{doub-exact-char}-(i), there exists a unique Radon probability measure $\mu_{G_\mathrm{c}}$ such that ${\pi_{S, G_\mathrm{c}}}_\#\mu_{G_\mathrm{c}}=\mu_S$ for all $S\in J$ with $S\subseteq {G_\mathrm{c}}$. Since $K_{Q\cap G_\mathrm{c}}=\bigcap_{S\in J, S\subseteq G_\mathrm{c}} \pi_{S,G_\mathrm{c}}^{-1}(K_{Q\cap S})$ and $\mu_{G_\mathrm{c}}(\pi_{S,G_\mathrm{c}}^{-1}(K_{Q\cap S}))=\mu_S(K_{Q\cap S})=1$ for all $S\in J$ such that $S\subseteq G_\mathrm{c}$, we get $\mu_{G_\mathrm{c}}(K_{Q\cap G_\mathrm{c}})=1$ (see \cite[Part I, Chapter I, 6.(a)]{S73}). As
$
K_{Q\cap G_\mathrm{c}}=\bigcup_{n\in\NN}K_n^{(G_\mathrm{c})}$, where $ K_n^{(G_\mathrm{c})}:=\{\alpha\in X(G_\mathrm{c}): 0\leq \hat{g}(\alpha)\leq n\ \forall g\in Q\cap G_\mathrm{c}\},$
the continuity from below of $\mu_{G_\mathrm{c}}$ implies that there exists $m\in\NN$ such that $\mu_{G_\mathrm{c}}(K_m^{G_\mathrm{c}})\geq 1-\varepsilon$. Let $T=\langle F\rangle\in J$. Define 
$$
K_m^{(T,\mathrm{c})}:=\{\alpha\in X(T):0\leq\hat{g}(\alpha)\leq m\text{ for all } g\in Q\cap(G_\mathrm{c}\cap T)\}
$$
and set $R:=\langle G_\mathrm{c}\cup T\rangle\in\Lambda$. Then we have
$
\pi_{G_\mathrm{c},R}^{-1}(K_m^{(G_\mathrm{c})})\subseteq \pi_{T,R}^{-1}(K_m^{(T,\mathrm{c})})
$
and so 
$$
\mu_T(K_m^{(T,\mathrm{c})})= \mu_R(\pi_{T,R}^{-1}(K_m^{(T,\mathrm{c})}))
\geq  \mu_{R}(\pi_{G_\mathrm{c},R}^{-1}(K_m^{(G_\mathrm{c})}))= \mu_{G_\mathrm{c}}(K_m^{(G_\mathrm{c})}) \geq 1-\varepsilon.
$$

Let us define now
$$
K^{(T,\mathrm{a})}:=\{\alpha\in X(T): \hat{g}(\alpha)\geq 0\text{ for all }g\in Q\cap(G_\mathrm{a}\cap T)\}=\!\!\bigcap_{S\in J, S\subseteq G_\mathrm{a}\cap T} \pi_{S,T}^{-1}(K_{Q\cap S}).
$$
Then $\mu_T(K^{(T,\mathrm{a})})=1$ as the finite measure $\mu_T$ on $X(T)$ is inner regular and $\mu_T(\pi_{S,T}^{-1}(K_{Q\cap S}))=\mu_S(K_{Q\cap S})=1$ for all $S\in J$ such that $S\subseteq G_\mathrm{a}\cap T$.

Taking $K^{(T)}:= K_m^{(T,\mathrm{c})}\cap K^{(T,\mathrm{a})}$ we obtain that
$$
\mu_T(K^{(T)})=\mu_T(K_m^{(T,\mathrm{c})}\cap K^{(T,\mathrm{a})})=\mu_T(K_m^{(T,\mathrm{c})})\geq 1-\varepsilon
$$
and that, for all $R\in J$ with $T\subseteq R$,
$$
\pi_{T,R}(K^{(R)})\subseteq \pi_{T,R}(K_m^{(R,\mathrm{c})})\cap \pi_{T,R}(K^{(R,\mathrm{a})})\subseteq K_m^{(T,\mathrm{c})} \cap K^{(T,\mathrm{a})}= K^{(T)}.
$$

Let us finally prove that $K^{(T)}$ is compact. Let $\alpha\in K^{(T)}$. If $g\in F\cap G_\mathrm{a}$, then there exists $N\in\NN$ such that $N^2-g^2\in Q\cap G_\mathrm{a}$ by (b) and so $\alpha(g)\in[-N,N]$. If $g\in F\cap G_\mathrm{c}$, then $g^2\in Q\cap G_\mathrm{c}$ and so $0\leq \hat{g}^2(\alpha)\leq m \leq m^2$, i.e., $\alpha(g)\in [-m,m]$. This shows that  $K^{(T)}$ is compact since $F$ generates $T$ and so $X(T)$ embeds into~$\RR^F$.
Hence, Proposition \ref{prop::compact-support-epsilon-K} ensures that $\{\mu_T:T\in J\}$ fulfills \eqref{epsilon-K-char}.
\end{proof}

\begin{ques}
Does \eqref{epsilon-K-char} imply (ii) in Theorem~\ref{hybrid-quadraticMod}?
\end{ques}

Theorem \ref{hybrid-quadraticMod} is a generalization of \cite[Theorem 5.4]{GKM16} which we restate here for the convenience of the reader.
\begin{cor} 
Let $\RR[X_i: i\in \Omega]$ be the algebra of real polynomials in the variables $\{X_i: i\in \Omega\}$ with $\Omega$ arbitrary index set, $Q$ a quadratic module in $\RR[X_i: i\in \Omega]$, $L:\RR[X_i: i\in \Omega] \rightarrow \mathbb{R}$ linear and $L(1)=1$. If $L(Q) \subseteq [0,\infty)$ and there exists $\Omega_c\subseteq \Omega$ countable such that
\begin{enumerate}[(a)]
\item\label{PartialArch}  for each $i\in \Omega\setminus\Omega_c$ the quadratic module $Q\cap\RR[X_i]$ is Archimedean,
\item\label{Carleman-RestCoordinates} for each $i \in \Omega_c$, Carleman's condition $\sum_{n=1}^\infty\frac{1}{\sqrt[2n]{L(X_i^{2n})}}=\infty$ is satisfied;
\end{enumerate}  
then there exists a unique $K_Q-$representing Radon measure $\mu$ for $L$, where $K_Q$ is as in \eqref{semi-alg-Q}.\end{cor}

\subsection{Localizations}\label{appl-loc} \ 

In this subsection we deal with the approach to Problem~\ref{GenKMP} through localizations introduced in \cite{M, M-II} and exploited in \cite{GKM16} for $A=\RR[X_i: i\in \Omega]$ with $\Omega$ arbitrary index set. In \cite[Corollary 4.4]{GKM16} it is explained how Problem~\ref{GenKMP} for $A$ reduces to the understanding of Problem~\ref{GenKMP} for the localized algebra $B=\RR\left[X_i, \tfrac{1}{1+X^2_i}: i\in\Omega\right]$. Here we show how our projective limit techniques allow to solve Problem~\ref{GenKMP} for the algebra $B$ using the corresponding results in finite dimensions in \cite{M03}.

\begin{thm}\label{thm::localization}
Let $\indexset:=\{S\subseteq B:S \text{ finitely generated subalgebra of }B\}$, $L:B\to\RR$ linear and $L(1)=1$. If $L(Q)\subseteq [0,\infty)$ for some quadratic module $Q$ in $B$, then: 
\begin{enumerate}[(i)]
\item There exists a unique $X(B)-$representing measure $\mu$ for $L$ on the cylinder $\sigma-$algebra~$\Sigma_\indexset$ supported on each $\pi_S^{-1}(K_{Q\cap S})$ with $S\in \indexset$ and  $K_{Q\cap S}$ is as in~\eqref{bcsas-S}. Moreover, if $Q$ is countably generated, then $\mu$ is a $K_Q-$representing measure for $L$ on~$\Sigma_\indexset$.
\item There exists a unique $K_Q-$repre\-senting Radon measure $\nu$ for $L$ if and only if $\{{\pi_S}_{\#}\mu: S\in I\}$ fulfills Prokhorov's condition~\eqref{epsilon-K-char}.
\end{enumerate}
\end{thm}
\begin{proof}
For any finite $F\subseteq\Omega$ let $S_F:=\RR[X_i:i\in F]$ and set $p_F:=\prod_{i\in F}(1+X_i^2)$. Then $\subindexset:=\{S_F[1/p_F]:F\subseteq\Omega\text{ finite}\}$ is a cofinal subset of $\indexset$, where 
$
S_F[1/p_F]:=\left\{\frac{a}{p_F^k}:a\in S_F,k\in\NN_0\right\}.
$
\enlargethispage{0.8cm}
Let $S:=S_F[1/p_F]\in \subindexset$ for some finite $F\subseteq\Omega$. Then $p_F$ satisfies condition~(1) in \cite[Proposition~4.2]{M03} (see the discussion after the proof of this proposition), where $M=Q\cap S_F$ (and so $M[1/p_F^2]=Q\cap S_F[1/p_F]$). Then the assumptions of \cite[Corollary~3.4]{M03} are fulfilled. Thus, there exists a unique $K_{Q\cap S}-$representing Radon measure $\mu_S$ for $L\!\restriction_S$ (see \cite[Remark 3.5]{M03}) and $\{\mu_S: S\in J\}$ is an exact projective system by Lemma~\ref{lem:uniq-S}. Since the map $\pi_T$ is surjective for any $T\in\Lambda$ (cf.\ Remark \ref{doub-exact-char}) conclusion (i) follows from Corollary~\ref{main-KMP}-(i) and Remark~\ref{rem::cylinder-MP-cofinal}-(i), (ii) and~(iii). Conclusion (ii) follows from Corollary~\ref{main-KMP}-(ii) and Remark~\ref{rem::cylinder-MP-cofinal}-(ii) and~(iii) as ${\pi_S}_{\#}\mu=\mu_S$ for all $S\in J$.
\end{proof}

Applying Theorem \ref{thm::localization} to $Q= \sum B^2$ we retrieve \cite[Theorem 3.8]{GKM16}.

\vspace{-0.1cm}
\section*{Acknowledgments}
We would like to thank Victor Vinnikov for the interesting discussions which led us to Lemmas \ref{lem:uniq-S} and \ref{lem::ex-ex-proj-sys}.We  are also grateful to the anonymous referee for her/his valuable suggestions. We are indebted to the Baden-W\"urttemberg Stiftung for the financial support to this work by the Eliteprogramme for Postdocs. This work was also partially supported by the Ausschuss f\"ur Forschungsfragen (AFF) and Young Scholar Fund (YSF) 2018 awarded by the University of Konstanz. 
\vspace{-0.1cm}
\section*{Data availability statement}
Data sharing is not applicable to this article as no datasets were generated or analysed during the current study.
\enlargethispage{0.8cm}
\vspace{-0.1cm}
{
}

\enlargethispage{0.8cm}

\begin{thebibliography}{99}
\bibitem{AJK15} 
D. Alpay, P. E. T. Jorgensen and D. P. Kimsey,
{\it Moment problems in an infinite number of variables}, 
Infin. Dimens. Anal. Quantum Probab. Relat. Top.  {18} (2015), no. 4, 14 pp.
 
\bibitem{Bau} 
H. Bauer, 
{\it Measure and integration theory (Transl. from the German)}, 
De Gruyter Studies in Mathematics {26}, Walter de Gruyter \& Co., Berlin 2001.

\bibitem{BK}
Yu. M. Berezansky and Yu. G. Kondratiev,
{\it Spectral methods in infinite-dimensional analysis. Vol. II (Transl. from the 1988 Russian original)}, 
Mathematical Physics and Applied Mathematics {12/2}, Kluwer Academic Publishers, Dordrecht 1995. 

\bibitem{BS}
Yu. M. Berezansky and S. N. {\v{S}}ifrin,
{\it A generalized symmetric power moment problem (Russian)}, 
Ukrain. Mat. \v Z. {23} (1971), 291--306.

\bibitem{BCR} 
C. Berg, J. P. R. Christensen and P. Ressel,
{\it Positive definite functions on abelian semigroups},
Math. Ann. {223} (1976), no. 3, 253--274.
 
\bibitem{B99} 
P. Billingsley, 
{\it Convergence of probability measures}, 2nd ed., 
Wiley Series in Probability and Statistics, Wiley, Chichester 1999. 

\bibitem{B55} 
S. Bochner, 
{\it Harmonic analysis and the theory of probability}, 
University of California Press, Berkeley and Los Angeles 1955.  

\bibitem{Bor-Yng75}
H. J. Borchers and J. Yngvason,
{\it Integral representations for {S}chwinger functionals and the moment problem over nuclear spaces},
Comm. Math. Phys. {43} (1975), no. 3, 255--271.

\bibitem{BouGT} 
N. Bourbaki, 
{\it Elements of mathematics. General topology. Part 1}, 
Hermann, Paris; Addison-Wesley Publishing Co., Reading, MA etc.; 1966. 

\bibitem{BouInt}
N. Bourbaki, 
{\it Elements of Mathematics. Integration. II. Chapters 7--9 (Transl. from the 1963 and 1969 French originals)}, 
Springer, Berlin 2004. 

\bibitem{BouST} 
N. Bourbaki, 
{\it Elements of mathematics. Theory of sets}, 
Hermann, Paris; Addison-Wesley Publishing Co., Reading, MA etc.; 1968. 

\bibitem{Carl26}
T. Carleman,
{\it Les fonctions quasi-analytiques},
Collection de monographies sur la th{\'e}orie des fonctions publi{\'e}e {7}, Gauthier-Villars, Paris 1926.

\bibitem{Ch58} 
J. R. Choksi,
{\it On compact contents}, 
J. Lond. Math. Soc. {33} (1958), 387--398. 

\bibitem{Denj21}
{A. Denjoy,
{\it Sur les fonctions quasi-analytiques de variable r\' eelle},
C.R. Acad. Sci. Paris {173} (1921), 1329--1331.}

\bibitem{F72}
{Z. Frol\'ik, {\it Projective limits of measure spaces.} Proceedings of the Sixth Berkeley Symposium on Mathematical Statistics and Probability, Volume 2: Probability Theory, 67--80, University of California Press, Berkeley, Calif., 1972.}

\bibitem{GV64} 
I. M. Gel'fand and N. Ya. Vilenkin, 
{\it Generalized functions. Vol. 4: Applications of harmonic analysis}, 
Academic Press, New York and London, 1964. 

\bibitem{GIKM} 
M. Ghasemi, M. Infusino, S. Kuhlmann and M. Marshall,
{\it Moment problem for symmetric algebras of locally convex spaces},
Integr. Equat. Oper. Th. 90 (2018), no. 3, Art. 29, 19 pp. 

\bibitem{MGHAS} 
M. Ghasemi and S. Kuhlmann,
{\it Closure of the cone of sums of $2d-$powers in real topological algebras},
J. Funct. Anal. {264} (2013), no. 1, 413--427.

\bibitem{MGES} 
M. Ghasemi, S. Kuhlmann and E. Samei,
{\it The moment problem for continuous positive semidefinite linear functionals},
Arch. Math. (Basel) {100} (2013), no. 1, 43--53.

\bibitem{GKM13} 
M. Ghasemi, S. Kuhlmann and M. Marshall,
{\it Application of Jacobi's representation theorem to locally multiplicatively convex topological $\RR-$algebras},
J. Funct. Anal. {266} (2014), no. 2, 1041--1049. 

\bibitem{GKM16} 
M. Ghasemi, S. Kuhlmann and M. Marshall,
{\it Moment problem in infinitely many variables},
Israel J. Math. {212} (2016), no. 2, 989--1012.

\bibitem{GMW} 
M. Ghasemi, M. Marshall and S. Wagner,
{\it Closure of the cone of sums of $2d-$powers in certain weighted $\ell_1-$seminorm topologies},
Canad. Math. Bull. {57} (2014), no. 2, 289--302.

\bibitem{Hal}
P. Halmos, 
{\it Measure theory}, 
Grad. Texts in Math. {18}, Springer, Berlin 1974  

\bibitem{Hav} 
E. K. Haviland,
{\it On the momentum problem for distribution functions in more than one dimension. II},
Amer. J. Math. {58} (1936), no. 1, 164--168.

\bibitem{Heg75}
G. C. Hegerfeldt,
{\it Extremal decomposition of Wightman functions and of states on nuclear $^*-$algebras by Choquet theory}, 
Comm. Math. Phys. {45} (1975), no. 2, 133--135.

\bibitem{I} 
M. Infusino,
{\it Quasi-analyticity and determinacy of the full moment problem from finite to infinite dimensions},
Stochastic and Infinite Dimensional Analysis, {Chap.\!~9}: 161--194, Trends in Mathematics, Birkh\"auser, 2016. 

\bibitem{IK-probl}
M. Infusino and S. Kuhlmann,
{\it Infinite dimensional moment problem: open questions and applications},
Ordered Algebraic Structures and Related Topics, Contemporary Mathematics {697}, 187--201, Amer. Math. Soc., Providence, RI 2017. 

\bibitem{IKM} 
M. Infusino, S. Kuhlmann and M. Marshall,
{\it On the determinacy of the moment problem for symmetric algebras of a locally convex space}, 
Operator theory in different settings and related applications, 
Oper. Theory Adv. Appl. {262}, 243--250, Birkh\"auser/Springer, Cham 2018. 

\bibitem{IKKM} 
M. Infusino, S. Kuhlmann, T. Kuna and P. Michalski,
{\it Topological aspects of the infinite dimensional moment problem}, in preparation.

\bibitem{IK-new} M. Infusino and T. Kuna, {\it The full moment problem on subsets of probabilities and point configurations}, J. Math. Anal. Appl. 483 (2020), no. 1, 123551.

\bibitem{IKR} 
M. Infusino, T. Kuna and A. Rota,
{\it The full infinite dimensional moment problem on semi-algebraic sets of generalized functions},
J. Funct. Anal.  {267} (2014), no. 5, 1382--1418.

 
\bibitem{Jac} T. Jacobi, {\it A representation theorem for certain partially ordered commutative rings}. Math. Z. 237 (2001), 259--273.

\bibitem{K-33}
{A.N. Kolmogorov, {\it Grundbegriffe der der Wahrscheinlichkeitsrechnung}, Springer-Verlag Berlin Heidelberg, 1933.}

\bibitem{K-P07}
S. Kuhlmann and M. Putinar,
{\it Positive polynomials on fibre products},
 C. R. Math. Acad. Sci. Paris {344} (2007), no. 11, 681--684.

\bibitem{K-P09}
S. Kuhlmann and M. Putinar,
{\it Positive polynomials on projective limits of real algebraic varieties},
Bull. Sci. Math. {133} (2009), no. 1, 92--111. 

\bibitem{Lass} 
J. B. Lasserre, 
{\it The K-moment problem for continuous linear functionals}, 
Trans. Amer. Math. Soc. {365} (2013), no. 5, 2489--2504.

\bibitem{M03} 
M. Marshall,
{\it Approximating positive polynomials using sums of squares}, 
Canadian Mathematical Bulletin {46} (2003), no. 3, 400--418.

\bibitem{MarshBook} 
M. Marshall,
{\it Positive polynomials and sums of squares},
Mathematical Surveys and Monographs {146}, Amer. Math. Soc., Providence, RI 2008.

\bibitem{M} 
M. Marshall, 
{\it Application of localization to the multivariate moment problem}, 
Math. Scand. {115} (2014), no. 2, 269--286.

\bibitem{M-II} 
M. Marshall, 
{\it Application of localization to the multivariate moment problem II}, 
Math. Scand. {120} (2017), no. 1, 124--128.

\bibitem{M63}
{M. M\'etevier, {\it Limites projectives de mesures, martingales, applications.}, Annali di Mat. 4 (1963), 225--352.}

\bibitem{N65}
A. E. Nussbaum,
{\it Quasi-analytic vectors},
Ark. Mat. 6 (1965), 179--191.

\bibitem{Prokh} 
Yu. V. Prokhorov,
{\it Convergence of random processes and limit theorems in probability theory (Russian)}, 
Teor. Veroyatnost. i Primenen. {1} (1956), 177--238.

\bibitem{Riesz} 
M.\! Riesz,
{\it Sur le probl\`eme de moments.\! III},
Ark.\! f.\! Mat., Astr.\! och Fys.\! {17} (1923), no.\! 16, 52~pp. 

\bibitem{Schmu78} 
K. Schm\"{u}dgen,
{\it Positive cones in enveloping algebras},
Rep. Math. Phys. {14} (1978), no. 3, 385--404. 

\bibitem{Schmu90} 
K. Schm\"{u}dgen,
{\it Unbounded operator algebras and representation theory},
Operator Theory: Advances and Applications {37}, Birkh\"auser Verlag, Basel 1990. 

\bibitem{Schm17} 
K. Schm\"udgen, 
{\it The moment problem}, 
Grad. Texts in Math. {277}, Springer, Cham 2017.

\bibitem{Schmu-new} 
K. Schm\"udgen,
{\it On the infinite dimensional moment problem},
Ark. Mat. {56} (2018), no. 2, 441--459.

\bibitem{S73} 
L. Schwartz, 
{\it Radon measures on arbitrary topological spaces and cylindrical measures}, 
Tata Inst. Fund. Res. Stud. Math. {6}, Oxford University Press, London 1973. 
\enlargethispage{1cm}
\bibitem{Y85} 
Y. Yamasaki, 
{\it Measures on infinite-dimensional spaces},
Series in Pure Mathematics {5}, World Scientific Publishing Co., Singapore 1985.
\end{thebibliography}
\end{document}